\documentclass[12pt,reqno]{memo-l}

\usepackage{a4wide}
\usepackage{ifthen,amsfonts,
amsmath,amssymb,epic,eepic,epsfig,color,stmaryrd,mathrsfs,bm,graphicx, mathtools}

\newtheorem{theorem}{Theorem}[chapter]
\newtheorem{lemma}[theorem]{Lemma}
\newtheorem{propos}[theorem]{Proposition}
\newtheorem{corol}[theorem]{Corollary}

\theoremstyle{definition}
\newtheorem{definition}[theorem]{Definition}

\theoremstyle{remark}
\newtheorem{remark}[theorem]{Remark}

\numberwithin{section}{chapter}
\numberwithin{equation}{chapter}

\newcommand{\eps}{{\varepsilon}}

\newcommand{\mint}{-\!\!\!\!\!\!\int}

\newcommand{\Lip}{{\text {Lip}}}
\newcommand{\diam}{{\text {diam}}}
\newcommand{\dist}{{\text {dist}}}

\newcommand{\e}{{\text {e}}}

\newcommand\weak{{\rightharpoonup}\,}

\newcommand\supp{{\rm supp}\,}
\newcommand\res{\mathop{\hbox{\vrule height 7pt width .5pt depth 0pt
\vrule height .5pt width 6pt depth 0pt}}\nolimits}
\newcommand\Id{{\rm Id}\,}
\newcommand\disk{{\mathbb D}}
\newcommand{\Om}{\Omega}
\def\I#1{{\mathcal{A}}_{#1}}
\def\Is#1{{\mathcal{A}}_{#1} (\R{n})}
\newcommand{\Iqs}{{\mathcal{A}}_Q(\R{n})}
\newcommand{\Iq}{{\mathcal{A}}_Q}
\def\a#1{\left\llbracket{#1}\right\rrbracket}
\newcommand{\abs}[1]{\left|#1\right|}
\newcommand{\norm}[2]{\left\|#1\right\|_{#2}}
\newcommand{\ra}{\rightarrow}
\newcommand{\D}{\textup{Dir}}
\newcommand{\de}{\partial}
\newcommand{\xii}{{\bm{\xi}}}
\newcommand{\ro}{{\bm{\rho}}}
\newcommand{\card}{{\textrm{card}}}
\newcommand{\etaa}{{\bm{\eta}}}
\newcommand{\ph}{\varphi}

\newcommand{\cG}{{\mathcal{G}}}

\newcommand{\cH}{{\mathcal{H}}}

\newcommand{\cQ}{{\mathcal{Q}}}
\newcommand{\Pe}{{\mathscr{P}}}
\def\R#1{{\mathbb R}^{#1}}
\newcommand\Z{{\mathbb Z}}

\newcommand\N{{\mathbb N}}
\newcommand\Q{{\mathbb Q}}
\newcommand\s{{\mathbb S}}
\newcommand\C{{\mathbb C}}

\makeindex

\begin{document}

\frontmatter

\title{Q-valued Functions Revisited}
\author{Camillo De Lellis}
\address{Institut f\"ur Mathematik, Universit\"at Z\"urich,
 Winterthurerstrasse 190 CH-8057 Z�rich}
\email{camillo.delellis@math.uzh.ch}

\author{Emanuele Nunzio Spadaro}
\email{emanuele.spadaro@math.uzh.ch}

\date{03.03.2009}

\begin{abstract}
In this note we revisit Almgren's theory of $Q$-valued functions, that
are functions taking values in the space $\Iqs$ 
of unordered $Q$-tuples of points in $\R{n}$. 
In particular:
\begin{itemize}
\item we give shorter versions of Almgren's proofs of the existence 
of $\D$-minimizing $Q$-valued functions,
of their H\"older regularity and of the dimension estimate
of their singular set;
\item we propose an alternative, intrinsic approach to these results,
not relying on Almgren's biLipschitz embedding
$\xii: \Iqs\to\R{N(Q,n)}$;
\item we improve upon the estimate of the singular set 
of planar $\D$-minimizing
functions by showing that it consists of isolated points.
\end{itemize}
\end{abstract}

\maketitle

\setcounter{page}{4}

\tableofcontents

\mainmatter
\chapter*{Introduction}

The aim of this paper is to provide a simple, complete 
and self-contained reference for Almgren's
theory of $\D$-minimizing $Q$-valued functions, 
so to make it an easy step for the understanding of the
remaining parts of the Big regularity paper \cite{Alm}. 
We propose simpler and shorter proofs of the central
results on $Q$-valued functions contained there, 
suggesting new 
points of view on many of them. In addition, parallel
to Almgren's theory, we elaborate
an intrinsic one which reaches his main results 
avoiding the extrinsic mappings $\xii$ and $\ro$ (see
Section \ref{s:biLipschitz} and compare with 1.2 of \cite{Alm}). 
This ``metric'' point of view is clearly an original
contribution of this paper. The second new contribution
is Theorem \ref{t:finite} where we improve Almgren's estimate
of the singular set in the planar case, relying heavily
on computations of White \cite{Wh}
and Chang \cite{Ch}.

Simplified and intrinsic proofs of parts of Almgren's big
regularity paper have already been established in 
\cite{Gob2} and \cite{Gob1}. In fact our proof
of the Lipschitz extension property for $Q$-valued
functions is essentially the one given in \cite{Gob2}
(see Section \ref{s:extend}).
Just to compare this simplified 
approach to Almgren's, note that
the existence of the retraction $\ro$ is actually
an easy corollary of the existence of $\xii$ and
of the Lipschitz extension theorem.
In Almgren's paper, instead,
the Lipschitz extension theorem
is a corollary of the existence of $\ro$, 
which is constructed explicitly (see 1.3 in \cite{Alm}) . 
However, even where our proofs differ most 
from his, we have been clearly influenced
by his ideas and we cannot exclude the existence
of hints to our strategies in \cite{Alm}
or in his other papers \cite{Alm2}
and \cite{Alm1}: the amount of material is very large 
and we have not explored it in all the details. 

Almgren asserts that some of the
proofs in the first chapters of \cite{Alm}
are more involved than apparently needed because
of applications contained in the other chapters,
where he proves his celebrated partial
regularity theorem for area-minimizing currents. 
We instead avoid any complication which
looked unnecessary for the 
theory of $\D$-minimizing $Q$-functions. For
instance, we do not show the existence of Almgren's
improved Lipschitz retraction $\ro^*$ (see 1.3 of
\cite{Alm}), since it is not needed in the theory
of Dir-minimizing $Q$-valued functions. This retraction
is instead used in the approximation of area-minimizing
currents (see Chapter 3 of \cite{Alm}) and will be 
addressed in the forthcoming paper \cite{DLSp}. 

In our opinion the portion of Almgren's 
Big regularity paper regarding the theory
of $Q$-valued functions is simply
a combination of clean ideas from the theory of elliptic partial
differential equations with elementary observations of combinatorial
nature, the latter being much less complicated than what they look
at a first sight. In addition our new ``metric'' point of
view reduces further the combinatorial part, at the
expense of introducing other arguments of more analytic flavor.

\subsection*{The metric space $\Iqs$}
Roughly speaking, our intuition of $Q$-valued functions is
that of mappings taking their values 
in the unordered sets of $Q$ points of $\R{n}$,
with the understanding that multiplicity can occur. We
formalize this idea by identifying the space of 
$Q$ unordered points in $\R{n}$ with the 
set of positive atomic measures of mass $Q$.

\begin{definition}[Unordered $Q$-tuples]\label{d:IQ}
We denote by $\a{P_i}$ the Dirac mass in $P_i\in \R{n}$ and
we define the space of $Q$-points as
\begin{equation*}
\Iqs :=\left\{\sum_{i=1}^Q\a{P_i}\,:\,P_i\in\R{n}\;\textrm{for every  }i=1,\ldots,Q\right\}.
\end{equation*}
\end{definition}

In order to simplify the notation, we use
$\Iq$ in place of $\Iqs$ and we write $\sum_i 
\a{P_i}$ when $n$ and $Q$ are clear from the context. 
Clearly, the points
$P_i$ do not have to be distinct: for instance $Q\a{P}$
is an element of $\Iqs$. 
We endow $\Iqs$ with a metric which makes
it a complete metric space (the completeness is
an elementary exercise left to the reader). 

\begin{definition}\label{d:FeG} For every
$T_1, T_2\in\Iqs$, with $T_1=\sum_i\a{P_i}$ 
and $T_2=\sum_i\a{S_i}$, we define
\begin{equation*}
\cG(T_1,T_2):=\min_{\sigma\in\Pe_Q}\sqrt{\sum_i\abs{P_i-S_{\sigma(i)}}^2},
\end{equation*}
where $\Pe_Q$ denotes the group of permutations 
of $\left\{1,\ldots,Q\right\}$.
\end{definition}

\begin{remark}\label{r:metric}
$(\Iqs, \mathcal{G})$ is a closed subset of a ``convex'' 
complete metric space. Indeed, $\cG$ coincides with the
$L^2$-Wasserstein distance on the space
of positive measures with finite second moment (see 
for instance \cite{AmGiSa} and \cite{Villani}). In Section \ref{s:metric1}
we will also use the fact that
$(\Iqs, \mathcal{G})$ can be embedded isometrically in
a separable Banach space. 
\end{remark}
The metric theory of
$Q$-valued functions starts from this remark. It 
avoids the Euclidean embedding
and retraction theorems of Almgren but is anyway powerful enough to prove
the main results on $Q$-valued functions addressed in this note.
We develop it fully in Chapter 4 after presenting
(in Chapters 1, 2 and 3) Almgren's theory
with easier proofs. However, since the metric point of view 
allows a quick, intrinsic definition of Sobolev mappings and
of the Dirichlet energy, we use it already here
to state immediately the main theorems.

\subsection*{$Q$-valued functions and the Dirichlet
energy}
For the rest of the paper $\Om$ will be a bounded open
subset of the Euclidean space $\R{m}$. If not specified,
we will assume that the regularity of $\partial \Omega$ is Lipschitz.
Continuous, Lipschitz, H\"older and (Lebesgue) measurable 
functions from $\Omega$ into $\Iq$ are defined in the usual way.
As for the spaces $L^p (\Omega, \Iq)$, they consist of those
measurable maps $u: \Omega\to \Iq$ such that 
$\|\cG (u, Q\a{0})\|_{L^p}$ is finite. Observe that,
since $\Omega$ is bounded, this is equivalent to ask
that $\|\cG (u, T)\|_{L^p}$ is finite for every $T\in \Iq$.
 
It is a general fact (and we show it in Section \ref{s:select}) 
that any
measurable $Q$-valued function can be written as the ``sum'' of $Q$
measurable functions.

\begin{propos}[Measurable selection]\label{p:selection}
Let $B\subset\R{m}$ be a measurable set and let
$f: B\to \Iq$ be a measurable function. 
Then, there exist $f_1,\ldots,f_Q$ measurable $\R{n}$-valued
functions such that
\begin{equation}\label{e.10}
f(x)=\sum_i \a{f_i(x)}\quad\text{for a.e. }x\in B.
\end{equation}
\end{propos}

Obviously, such a choice is far from being unique, but, in using
notation (\ref{e.10}), we will always
think of a measurable $Q$-valued function as coming together 
with such a selection.

We now introduce the Sobolev spaces of functions
taking values in the metric space of $Q$-points,
as defined independently by
Ambrosio in \cite{Amb} and Reshetnyak in \cite{Res}.

\begin{definition}[Sobolev $Q$-valued functions]
\label{d:W1p}
A measurable function $f:\Omega\ra\Iq$ is in the Sobolev class
$W^{1,p}$ ($1\leq p\leq\infty$) if there exist $m$ functions
$\varphi_j\in L^p(\Omega,\R{+})$
 such that 
\begin{itemize}
\item[(i)] $x\mapsto\cG (f(x),T)\in W^{1,p}(\Omega)$ for all $T\in \Iq$;
\item[(ii)] $\abs{\de_j\, \cG (f, T)}\leq\varphi_j$ almost everywhere in $\Omega$
for all $T\in \Iq$ and for all $j\in\{1, \ldots, m\}$.
\end{itemize}
\end{definition}

Definition \ref{d:W1p} can be easily generalized when
the domain is a
Riemannian manifold $M$. In this case we simply
ask that $f\circ x^{-1}$ is
a Sobolev $Q$-function for every open set $U\subset M$
and every chart $x:U\to \R{n}$. 
It is not difficult to show the existence 
of minimal functions $\tilde\varphi_j$ fulfilling (ii),
i.e. such that
\begin{equation*}
\tilde\varphi_j\leq\varphi_j\; \text{a.e. }\;\text{for any other $\varphi_j$ satisfying (ii),}
\end{equation*}
(see Proposition \ref{p:|Df|}).
We denote them by $|\de_j f|$.
We will later characterize 
$|\de_j f|$ by the following
property
(cp. with Proposition \ref{p:|Df|}): for every countable
dense subset $\{T_i\}_{i\in\N}$ of $\Iq$ and
for every $j=1,\ldots,m$,
\begin{equation}\label{e:def|D_jf|}
\abs{\de_j f}=\sup_{i\in\N}\abs{\de_j\,\cG(f,T_i)}
\quad\textrm{almost everywhere in } \Om.
\end{equation}
In the same way, given a vector field $X$,
we can define intrinsically $|\de_X f|$ and prove
the formula corresponding to \eqref{e:def|D_jf|}.
For functions $f\in W^{1,2}(\Omega,\Iq)$,
we set
\begin{equation}\label{e:def|Df|}
|Df|^2 :=\sum_{j=1}^m\abs{\de_jf}^2.
\end{equation}
For functions on a general Riemannian manifold $M$, 
we choose an orthonormal frame $X_1, \ldots X_m$
and set $|Df|^2 = \sum |\de_{X_i} f|^2$. This definition
is independent of the choice of coordinates (resp. of frames),
as it can be seen from Proposition \ref{p:equivalence def Dir}.

\begin{definition}\label{d:dirichlet}
The Dirichlet energy of $f\in W^{1,2}(U,\Iq)$, where $U$ is an open
subset of a Riemannian manifold,
is given by $\D(f, U) := \int_U |Df|^2$.
\end{definition}

It is not difficult to see that, when $f$ 
can be decomposed into finitely many 
regular single-valued functions, i.e. 
$f (x) = \sum_i \a{f_i (x)}$ for some
differentiable functions $f_i$, then 
\[
\D(f, U) = \sum_i \int_U |Df_i|^2 = \sum_i
\D(f_i, U) .
\]

The usual notion of trace at the boundary can be easily generalized to
this setting.

\begin{definition}[Trace of Sobolev $Q$-functions]
\label{d:Dirichlet problem}
Let $\Omega\subset \R{m}$ be a Lipschitz bounded open set
and $f\in W^{1,p} (\Omega, \Iq)$. A function $g$ belonging to
$L^p (\partial \Omega , \Iq)$ is said to be the trace of $f$ at $\partial \Omega$
(and we denote it by  $f|_{\partial\Omega}$) if,
for every $T\in \Iq$, the trace of the real-valued Sobolev function
$\cG (f, T)$ coincides with $\cG (g, T)$.
\end{definition}

It is straightforward to check that this notion of trace coincides
with the restriction of $f$ 
to the boundary when $f$ is a continuous
function which extends continuously
to $\overline{\Omega}$. In Section \ref{s:metric2}, we show
the existence and uniqueness of the trace for every $f\in W^{1,p}$. 
Hence, we can formulate a Dirichlet problem for $Q$-valued functions:
$f\in W^{1,2} (\Om,\Iq)$ is said to be $\D$-minimizing if
\begin{equation*}
\D (f, \Omega)\leq \D(g,\Omega)
\quad \mbox{for all $g\in W^{1,2}(\Om,\Iq)$ with }
f|_{\partial \Omega}=g|_{\partial\Omega}.
\end{equation*}

\subsection*{The main results proved in this paper}
We are now ready to state the main theorems of Almgren
reproved in this note:
an existence theorem and two regularity results.

\begin{theorem}[Existence for the Dirichlet Problem]
\label{t:existence}
Let $g\in W^{1,2} (\Om,\Iq)$. Then, there exists a $\D$-minimizing
function
$f\in W^{1,2}(\Om,\Iq)$ such that $f|_{\partial\Omega}=g|_{\partial\Omega}$.
\end{theorem}

\begin{theorem}[H\"older regularity]\label{t:hoelder} 
There exists a positive constant $\alpha=\alpha (m,Q)>0$ with the
following property.
If $f\in W^{1,2} (\Om,\Iq)$ is $\D$-minimizing,
then $f\in C^{0,\alpha} (\Omega')$ 
for every $\Omega'\subset\subset\Omega\subset \R{m}$.
For two-dimensional domains, we have the explicit
constant $\alpha (2, Q)=1/Q$.
\end{theorem}

For the second regularity theorem we need the definition of
singular set of $f$.

\begin{definition}[Regular and singular points]\label{d:regular}
A $Q$-valued function $f$ is regular at a point
$x\in \Omega$ if there exists a neighborhood $B$ of $x$
and $Q$ analytic functions $f_i:B\to \R{n}$
such that
\begin{equation*}
f (y)= \sum_i \a{f_i (y)}
\quad \mbox{for almost every $y\in B$}
\end{equation*}
and either $f_i (x)\neq f_j (x)$ for every $x\in B$ or $f_i\equiv f_j$.
The singular set $\Sigma_f$ of $f$ is the complement
of the set of regular points.
\end{definition}

\begin{theorem}[Estimate of the singular set]\label{t:structure}
Let $f$ be a $\D$-minimizing function. Then, the singular set $\Sigma_f$ of
$f$ is relatively closed in $\Omega$. Moreover, if $m=2$,
then $\Sigma_f$ is at most countable, and if
$m\geq 3$, then the Hausdorff dimension of $\Sigma_f$ is at most
$m-2$. 
\end{theorem}

Following in part ideas of \cite{Ch}, we improve this last
theorem in the following way.

\begin{theorem}[Improved estimate of the singular set]
\label{t:finite}
Let $f$ be $\D$-minimizing and $m=2$. Then, the singular set $\Sigma_f$ of
$f$ consists of isolated points.
\end{theorem}

This note is divided into five parts.
Chapter 1 gives the ``elementary theory''
of $Q$-valued functions, while
Chapter 2 focuses on the ``combinatorial results''
of Almgren's theory. In particular we give there 
very simple proofs of
the existence of Almgren's biLipschitz embedding 
$\xii:
\Iqs\to \R{N(Q,n)}$ and 
of a Lipschitz retraction $\ro$ of $\R{N(Q,n)}$ onto ${\bf \xii} (\R{N(Q,n)})$.
Following Almgren's approach, $\xii$ and $\ro$ are then used
to generalize the classical Sobolev theory to 
$Q$-valued functions. In Chapter 4 we develop the 
intrinsic theory and show how the results of 
Chapter 2 can be recovered independently of the maps
$\xii$ and $\ro$.
Chapter 3 gives simplified proofs of Almgren's regularity 
theorems for $Q$-valued functions
and Chapter 5 contains the improved estimate
of Theorem \ref{t:finite}.
Therefore, to get a proof of the four main Theorems
listed above, the reader can choose
to follow Chapters 1, 2, 3 and 5, or to follow Chapters 1, 4, 3 and 5.  

\subsection*{Acknowledgements}
The first author
is indebted with Bernd Kirchheim for many enlightening
discussions on some topics of this paper. Both authors
acknowledge the support of the Swiss National Foundation.
The second author aknowledges the Forschungskredit of
the University of Z\"urich. 

\chapter{The elementary theory of $Q$-valued functions}

This chapter consists of three sections. The first one 
introduces a recurrent theme:
decomposing $Q$-valued functions in simpler pieces.
We will often build on this and prove our statements inductively
on $Q$, relying ultimately on well-known properties
of single-valued functions.
Section \ref{s:extend} contains an elementary proof of
the following fact: any Lipschitz map from a subset of $\R{m}$
into $\Iq$ can be extended to a Lipschitz map on the whole 
Euclidean space. This extension theorem, combined with
suitable truncation techniques,
is the basic tool of various approximation results.
Section \ref{s:rademacher} introduces a 
notion of differentiability for $Q$-valued maps and contains
some chain--rule formulas and a generalization of the
classical theorem of Rademacher. These are the main ingredients
of several computations in later sections.

\section{Decomposition and selection for
$Q$-valued functions}\label{s:select}

Given two elements $T\in \Is{Q_1}$ and $S\in \Is{Q_2}$,
the sum $T+S$ of the two measures belongs 
to $\Is{Q}=\Is{Q_1+Q_2}$. This observation 
leads directly to the following definition.

\begin{definition}\label{d:dec&sel}
Given finitely many $Q_i$-valued functions $f_i$, 
the map $f_1+f_2+ \ldots + f_N$ defines a $Q$-valued function
$f$, where $Q=Q_1+Q_2+\ldots+Q_N$. This will be called a 
{\em decomposition of $f$ into $N$ simpler functions}.
We speak of measurable (Lipschitz, H\"older, etc.) decompositions, 
when the $f_i$'s are measurable (Lipschitz, H\"older,
etc.). In order to avoid confusions with the summation of vectors in
$\R{n}$, we will write, with a slight abuse of notation, 
\begin{equation*}
f=\a{f_1}+\ldots+\a{f_N}.
\end{equation*}
If $Q_1=\ldots = Q_N =1$, the decomposition is called a {\em selection}.
\end{definition}

Proposition \ref{p:selection} ensures the existence of a measurable
selection for any measurable $Q$-valued function.
The only role of this proposition is to 
simplify our notation.

\subsection{Proof of Proposition \ref{p:selection}}
We prove the proposition by induction on $Q$. The case $Q=1$ is of course
trivial. For the general case, we will make use
of the following elementary observation:
\begin{itemize}
\item[(D)] if $\bigcup_{i\in\N} B_i$ is a covering of
$B$ by measurable sets, 
then it suffices to find a measurable selection of $f|_{B_i\cap B}$ 
for every $i$. 
\end{itemize}
 
Let first $\mathcal{A}_0\subset \Iq$ be the closed set of points of
type $Q\a{P}$ and set $B_0=f^{-1} (\mathcal{A}_0)$. Then,
$B_0$ is measurable
and $f|_{B_0}$ has trivially a measurable selection.

Next we fix a point $T\in\Iq\setminus \mathcal{A}_0$, $T=\sum_i \a{P_i}$.
We can subdivide the set of indexes $\{1,\ldots ,Q\} = I_L\cup I_K$ into two 
nonempty sets of cardinality $L$ and $K$, with
the property that
\begin{equation}\label{e:far}
|P_k-P_l|> 0 \quad  \mbox{for every $l\in I_L$ and $k\in I_K$}.
\end{equation}
For every $S= \sum_i \a{Q_i}$,
let $\pi_S\in \Pe_Q$ be a permutation such that
\[
\cG(S,T)^2 = \sum_i |P_i - Q_{\pi_S (i)}|^2.
\]
If $U$ is a sufficiently small neighborhood of $T$ in $\Iq$, by \eqref{e:far}, the maps
\[
\tau: U\ni S \mapsto \sum_{l\in I_L} \a{Q_{\pi_S (l)}}\in \I{L},
\quad
\sigma: U\ni S \mapsto \sum_{k\in I_K} \a{Q_{\pi_S (k)}}\in \I{K}
\]
are continuous. Therefore, $C=f^{-1} (U)$ is measurable and
$\a{\sigma\circ f|_C} + \a{\tau\circ f|_C}$
is a measurable decomposition of $f|_C$.
Then, by inductive hypothesis, $f|_C$ has a measurable selection.

According to this argument, it is possible
to cover $\Iq\setminus \mathcal{A}_0$ with 
open sets $U$'s such that, if $B=f^{-1} (U)$,
then $f|_{B}$ has a measurable selection. Since $\Iq\setminus
\mathcal{A}_0$ is an open subset of a separable metric
space, we can find a countable covering $\{U_i\}_{i\in\N}$ of this type.
Being $\{B_0\}\cup \{f^{-1} (U_i)\}_{1=1}^\infty$
a measurable covering of $B$, from (D) we conclude the proof.

\subsection{One dimensional $W^{1,p}$-decomposition}
\label{ss:dec1d}

A more serious problem is to find
selections which are as regular as $f$ itself. 
Essentially, this is always possible when the
domain of $f$ is $1$-dimensional. For our purposes
we just need the Sobolev case of this principle,
which we prove in the next two propositions.

In this subsection $I=[a,b]$ is a closed bounded interval of $\R{}$ 
and the space of absolutely continuous functions
$AC (I, \Iq)$ is defined as the space
of those continuous $f: I\to \Iq$ such that, for every $\eps>0$,
there exists $\delta>0$ with the following property:
for every $a\leq t_1<t_2<...<t_{2N}\leq b$,
\[ 
\sum_i (t_{2i} - t_{2i-1}) < \delta\quad\textrm{implies}\quad
\sum_i \cG \bigl(f (t_{2i}), f (t_{2i-1})\bigr) < \eps .
\]

\begin{propos}\label{p:Wselection-1} Let $f\in W^{1,p} (I, \Iq)$.
Then,
\begin{itemize}
\item[$(a)$] $f\in AC(I,\Iq)$ and, moreover,
$f\in C^{0,1-\frac{1}{p}}(I,\Iq)$ for $p>1$;
\item[$(b)$] there exists a selection $f_1,\ldots,f_Q\in W^{1,p}(I,\R{n})$
of $f$ such that $\abs{Df_i}\leq\abs{Df}$ almost everywhere.
\end{itemize}
\end{propos}

\begin{remark}
A similar selection theorem holds for continuous 
$Q$-functions. This result needs a subtler combinatorial argument 
and is proved in Almgren's Big regularity paper \cite{Alm}
(Proposition $1.10$, p. $85$). The proof of Almgren 
uses the Euclidean structure, whereas a more general argument
has been proposed in \cite{dLGT}. 
\end{remark}

Proposition \ref{p:Wselection-1} cannot be extended to maps $f\in W^{1,p}
(\s^1, \Iq)$. 
For example, we identify $\R{2}$ with the complex
plane $\C$ and $\s^1$ with the set $\{z\in \C: |z|=1\}$ 
and we consider the map $f:\s^1\to \Iq (\R{2})$
given by $f(z) = \sum_{\zeta^2=z} \a{\zeta}$.
Then, $f$ is Lipschitz (and hence belongs to $W^{1,p}$ for every $p$)
but it does not have a continuous selection.  Nonetheless, 
we can use Proposition \ref{p:Wselection-1}
to decompose any $f\in W^{1,p}
(\s^1, \Iq)$ into ``irreducible pieces''. 

\begin{definition}\label{d:irreducible}
$f\in W^{1,p} (\s^1, \Iq)$ is called
{\em irreducible} if there is no decomposition of $f$
into $2$ simpler $W^{1,p}$ functions. 
\end{definition}

\begin{propos}\label{p:Wselection} 
For every $Q$-function $g\in W^{1,p} (\s^1, \Iq (\R{n}))$,
there exists a decomposition
$g=\sum_{j=1}^J\a{g_j}$, where each $g_j$ is
an irreducible $W^{1,p}$ map.
A function $g$ is irreducible if and only if
\begin{itemize}
\item[$(i)$] $\card\, (\supp (g (z))) = Q$ for every $z\in\s^1$
and 
\item[$(ii)$] there exists a $W^{1,p}$ map $h: \s^1 \to \R{n}$
with the property that $f(z) = \sum_{\zeta^Q =z} \a{h (\zeta)}.$
\end{itemize}
Moreover, for every irreducible $g$,
there are exactly $Q$ maps $h$ fulfilling $(ii)$.
\end{propos}

The existence of an irreducible decomposition
in the sense above is an obvious consequence
of the definition of irreducible maps. 
The interesting part of the proposition is 
the characterization of the irreducible pieces,
a direct corollary of Proposition \ref{p:Wselection-1}.

\begin{proof}[Proof of Proposition \ref{p:Wselection-1}]
We start with (a). Fix a dense set
$\{T_i\}_{i\in\N}\subset \mathcal{A}_Q$.
Then, for every $i\in\N$,
there is a negligible set $E_i\subset I$ such that, for every $x<y\in I\setminus E_i$,
\begin{equation*}
\big|\cG (f(x), T_i)-\cG (f(y), T_i)\big|
\leq\left|\int_x^y \cG (f, T_i)^\prime\right|\leq \int^y_x |Df|.
\end{equation*}
Fix $x<y\in I\setminus \cup_i E_i$ and choose a 
sequence $\{T_{i_l}\}$ converging to $f(x)$. Then,
\begin{equation}\label{e:difference}
\cG(f(x),f(y)) = \lim_{l\to\infty} 
\big|\cG (f(x), T_{i_l})-\cG (f(y), T_{i_l})\big|
\leq \int_x^y \abs{Df} .
\end{equation}
Clearly, \eqref{e:difference} gives the absolute continuity of $f$
outside $\cup_i E_i$. Moreover, $f$ can be redefined in a unique way
on the exceptional
set so that the estimate \eqref{e:difference}
holds for every pair $x,y$.
In the case $p>1$, we improve \eqref{e:difference} 
to $\cG(f(x),f(y))\leq \|\abs{Df}\|_{L^p}\abs{x-y}^{(p-1)/p}$,
thus concluding the H\"older continuity.

\medskip

For $(b)$, the strategy is to find $f_1,\ldots,f_Q$ 
as limit of approximating piecewise linear functions. 
To this aim, fix $k\in\N$ and set
\[
\Delta_k :=\frac{b-a}{k}\quad\textrm{and}\quad t_l:=a+l
\,\Delta_k,\quad\textrm{with}\quad l=0,\ldots,k.
\]
By (a), without loss of generality, we assume that $f$ is
continuous and we consider the points
$f(t_l)=\sum_i\a{P_i^l}$. Moreover, after possibly reordering
each $\{P_i^l\}_{i\in \{1, \ldots, Q\}}$, we can assume that
\begin{equation}\label{e:permut}
\cG(f(t_{l-1}),f(t_{l}))^2 = \sum_i\abs{P_i^{l-1}
-P_i^l}^2.
\end{equation}
Hence, we define the functions $f_i^k$ as the linear
interpolations between the points $(t_{l},P_i^{l})$, that is,
for every $l=1,\ldots,k$ and every $t\in[t_{l-1},t_l]$, we set
\begin{equation*}
f_i^k(t)=\frac{t_{l}-t}{\Delta_k}\,P^{l-1}_i+\frac{t-t_{l-1}}{\Delta_k}\,P^{l}_i.
\end{equation*}
It is immediate to see that the $f_i^k$'s are $W^{1,1}$
functions; moreover, for every $t\in(t_{l-1},t_l)$,
thanks to \eqref{e:permut},
the following estimate holds,
\begin{equation}\label{e:gradient}
\abs{Df_i^k(t)}=\frac{\abs{P^{l-1}_i
-P^{l}_i}}{\Delta_k}\leq
\frac{\cG(f(t_{l-1}),f(t_l))}{\Delta_k}
\leq\mint_{t_{l-1}}^{t_{l}}\abs{Df}(\tau)\, d\tau=:h^k(t).
\end{equation}
Since the functions $h^k$ converge in $L^p$ to $\abs{Df}$ for
$k\ra+\infty$, we conclude that the $f_i^k$'s are
equi-continuous and equi-bounded. Hence, up to passing to a subsequence,
which we do not relabel, there exist functions
$f_1,\ldots,f_Q$ such that $f_i^k\ra f_i$ uniformly.
Passing to the limit, \eqref{e:gradient} implies that
$\abs{Df_i}\leq \abs{Df}$ and it is a very simple
task to verify that $\sum_{i}\a{f_i}=f$.
\end{proof}

\begin{proof}[Proof of Proposition \ref{p:Wselection}]
The decomposition of $g$ into irreducible maps is a trivial
corollary of the definition of irreducibility. Moreover,
it is easily seen that a map satisfying $(i)$ and $(ii)$ is
necessarily irreducible.

Let now $g$ be an irreducible $W^{1,p}$ $Q$-function.
Consider $g$ as a function on $[0,2\pi]$ with
the property that $g(0)=g (2\pi)$ and let  $h_1, \ldots, h_Q$ in 
$W^{1,p} ([0,2\pi], \R{n})$ be a selection as
in Proposition \ref{p:Wselection-1}.
Since we have $g (0)=g (2\pi)$,  
there exists a permutation $\sigma$ such that
$h_i (2\pi) = h_{\sigma (i)} (0)$.
We claim that any such $\sigma$ is necessarily a $Q$-cycle.
If not, there is a partition
of $\{1, \ldots, Q\}$ into two disjoint nonempty subsets 
$I_L$ and $I_K$, with cardinality $L$ and $K$ respectively,
such that $\sigma (I_L) = I_L$ and $\sigma (I_K)= I_K$.
Then, the functions
\[
g_L= \sum_{i\in I_L} \a{h_i}
\quad \mbox{and} \quad
g_K= \sum_{i\in I_K} \a{h_i}
\]
would provide a decomposition of $f$ into two simpler
$W^{1,p}$ functions.

The claim concludes the proof. Indeed,
for what concerns $(i)$, we note that, if the support of
$g (0)$ does not consist of
$Q$ distinct points, there is always a permutation $\sigma$
such that $h_i (2\pi) = h_{\sigma (i)} (0)$ and which is not
a $Q$-cycle.
For $(ii)$, without loss of generality, we can order
the $h_i$ in such a way that $\sigma (Q)=1$ and $\sigma (i)= i+1$
for $i\leq Q-1$.
Then, the map $h: [0,2\pi] \to \R{n}$ defined by
\[
h (\theta) = h_i (Q \theta -2(i-1)\pi ),
\quad \mbox{for }\; \theta\in [2(i-1)\pi/Q, 2i\pi/Q],
\]
fulfils $(ii)$.
Finally, if a map
$\tilde{h}\in W^{1,p} (\s^1, \R{n})$ satisfies
\begin{equation}\label{e:unrolls-1}
g (\theta)= \sum_i \a{\tilde{h} ((\theta+2i\pi)/Q)}
\quad \mbox{for every $\theta$},
\end{equation}
then there is $j\in \{1,\ldots, Q\}$ such that $\tilde{h} (0) =
h (2j\pi/Q)$. By $(i)$ and the continuity of $h$ and $\tilde{h}$,
the identity $\tilde{h} (\theta) = h (\theta+2j\pi/Q)$ holds for
$\theta$ in a neighborhood of $0$. Therefore, since $\s^1$ is connected,
a simple continuation argument shows that $\tilde{h} (\theta) 
= h (\theta+2j\pi/Q)$ for every $\theta$.
On the other hand, all the $\tilde{h}$ of this form are different (due to $(i)$) and enjoy
\eqref{e:unrolls-1}: hence, there are exactly $Q$ distinct $W^{1,p}$ functions with this
property.
\end{proof}

\subsection{Lipschitz decomposition}

For general domains of dimension $m\geq 2$,
there are well-known obstructions to the
existence of regular selections. However, 
it is clear that, when $f$ is continuous
and the support of $f(x)$ does not consist of a single point,
in a neighborhood $U$ of $x$, there is a decomposition of $f$ 
into two continuous simpler functions. When $f$ is Lipschitz,
this decomposition holds 
in a sufficiently large ball, whose radius can be estimated
from below with a simple
combinatorial argument. This fact will play a key role
in many subsequent arguments.

\begin{propos}\label{p.10}
Let $f:B\subset\R{m}\ra\Iq$ be a Lipschitz function,
$f=\sum_{i=1}^Q\a{f_i}$. Suppose that there
exist $x_0\in B$ and
$i,j\in\{1,\ldots,Q\}$ such that
\begin{equation}\label{e.15}
\abs{f_i(x_0)-f_j(x_0)}>3\, (Q-1)\,\Lip(f)\,\diam(B).
\end{equation} 
Then, there is a decomposition of $f$ into
two simpler Lipschitz functions $f_K$ and $f_L$ with
$\Lip(f_K),\Lip(f_L)\leq \Lip(f)$ and $\supp (f_K(x))\cap \supp (f_L(x))=\emptyset$
for every $x$.
\end{propos}

\begin{proof} Call a ``squad'' any subset of indices
$I\subset\left\{1,\ldots,Q\right\}$ such that
\begin{equation*}\label{e:squad}
\abs{f_l(x_0)-f_r(x_0)}\leq 3\, (|I|-1)\, \Lip(f)\,\diam(B)
\quad\textrm{for all $l,r\in I$},
\end{equation*}
where $|I|$ denotes the cardinality of $I$.
Let $I_L$ be a maximal squad containing $1$, where $L$ stands for its
cardinality. By (\ref{e.15}), $L<Q$. Set
$I_K=\left\{1,\dots,Q\right\}\setminus I_L$. Note that,
whenever $l\in I_L$ and $k\in I_K$, 
\begin{equation}\label{e:distanti}
|f_l (x_0)-f_k (x_0)|>
3 \, \Lip(f)\,\diam(B),
\end{equation}
otherwise $I_L$ would not be maximal. For every $x$,
$y\in B$, we let $\pi_x$, $\pi_{x,y}\in\Pe_Q$ be
permutations such that
\begin{align*}
\cG(f(x_0),f(x))^2&=\sum_i\abs{f_i(x_0)-f_{\pi_x(i)}(x)}^2,\\
\cG(f(x),f(y))^2&=\sum_i\abs{f_i(x)-f_{\pi_{x,y}(i)}(y)}^2.
\end{align*}
We define the functions $f_L$ and $f_K$ as
\begin{equation*}
f_L(x)=\sum_{i\in I_L}\a{f_{\pi_x(i)}(x)}\quad
\textrm{and}\quad f_K(x)=\sum_{i\in I_K}\a{f_{\pi_x(i)}(x)}.
\end{equation*}
Observe that $f=\a{f_L}+\a{f_K}$: it remains to show the Lipschitz estimate.
For this aim, we claim that $\pi_{x,y} (\pi_x (I_L))= \pi_y (I_L)$
for every $x$ and $y$. Assuming the claim, we conclude that, for every $x,y\in B$,
\begin{equation*}
\cG(f(x),f(y))^2=\cG(f_L(x),f_L(y))^2+\cG(f_K(x),f_K(y))^2,
\end{equation*}
and hence $\Lip (f_L), \Lip (f_K)\leq\Lip (f)$.

To prove the claim, we argue by contradiction: 
if it is false, choose $x$,
$y\in B$, $l\in I_L$ and $k\in I_K$ with $\pi_{x,y} (\pi_x (l))
= \pi_y (k)$. Then,
$\abs{f_{\pi_x(l)}(x)-f_{\pi_y(k)}(y)}\leq\cG(f(x),f(y))$,
which in turn implies 
\begin{align*}
3\,\Lip(f)\,\diam(B) &\;\stackrel{\mathclap{\eqref{e:distanti}}}{<}\;\abs{f_l(x_0)-f_k(x_0)}\\
&\;\leq\abs{f_l(x_0)-f_{\pi_x(l)}(x)}+\abs{f_{\pi_x(l)}(x)-f_{\pi_y(k)}(y)}+
\abs{f_{\pi_y(k)}(y)-f_k(x_0)}\\
&\;\leq  \cG(f(x_0),f(x))+\cG(f(x),f(y))+\cG(f(y),f(x_0))\\
&\;\leq  \Lip(f)\left(\abs{x_0-x}+\abs{x-y}+\abs{y-x_0}\right)\\
&\;\leq  3\,\Lip(f)\,\diam(B).
\end{align*}
This is a contradiction and, hence, the proof is complete.
\end{proof}

\section{Extension of Lipschitz $Q$-valued functions}
\label{s:extend}
This section is devoted to prove the following extension
theorem.

\begin{theorem}[Lipschitz Extension]
\label{thm.ext}
Let $B\subset\R{m}$ and 
$f:B\ra\Iqs$ be Lipschitz. Then, there
exists an extension $\bar{f}:\R{m}\ra\Iqs$ of $f$,
with $\Lip(\bar{f})\leq C(m,Q)\,\Lip(f)$. Moreover, if $f$
is bounded, then, for every $P\in \R{n}$,
\begin{equation}\label{e:LinftyBound}
\sup_{x\in\R{m}}\, \cG (\bar{f}(x), Q\a{P})
\leq C(m, Q)\, \sup_{x\in B}\, \cG (f(x), Q\a{P}).
\end{equation}
\end{theorem}

Note that, in his Big regularity paper, Almgren deduces
Theorem \ref{thm.ext} from the existence of the maps
$\xii$ and $\ro$ of Section \ref{s:biLipschitz}. We instead
follow a sort of reverse path and conclude the existence of $\ro$
from that of $\xii$ and from Theorem \ref{thm.ext}.

It has already been observed by Goblet in \cite{Gob2}
that the Homotopy Lemma \ref{l:hom} below can be combined
with a Whitney-type argument to yield an easy direct proof
of the Lipschitz extension Theorem, avoiding Almgren's
maps $\xii$ and $\ro$. In \cite{Gob2} the author
refers to the general theory built in \cite{LanSch}
to conclude Theorem \ref{thm.ext} from Lemma \ref{l:hom}.
For the sake of completeness, we give here the complete argument.

\subsection{Homotopy Lemma}
Let $C$ be a cube with sides
parallel to the coordinate axes.
As a first step, we show the existence of extensions to $C$
of Lipschitz $Q$-valued functions defined on $\partial C$. 
This will be the key point  
in the Whitney type argument used in the proof of
Theorem \ref{thm.ext}.

\begin{lemma}[Homotopy lemma]\label{l:hom}
There is a constant $c(Q)$ with the following property.
For any closed cube with sides parallel to the coordinate axes
and any Lipschitz $Q$-function $h:\partial C\ra \Iqs$, there exists 
an extension $f:C\ra \Iqs$ of $h$ which is Lipschitz
with $\Lip(f)\leq c(Q) \Lip(h)$. Moreover, for every $P\in \R{n}$,
\begin{equation}\label{e:LinftyBound2}
\max_{x\in{C}}\, \cG (f(x), Q\a{P})
\leq 2\,Q\, \max_{x\in \partial C}\, \cG (h(x), Q\a{P}).
\end{equation}
\end{lemma}
\begin{proof} By rescaling and translating, it suffices to prove 
the lemma when $C=[0,1]^m$. Since $C$ is biLipschitz equivalent
to the closed unit ball $\overline{B_1}$ centered at $0$, it
suffices to prove the lemma with $\overline{B_1}$ in place of 
$C$. In order to prove this case,
we proceed by induction on $Q$. For $Q=1$,
the statement is a well-known fact (it is 
very easy to find an extension $\bar{f}$ with
$\Lip\, (\bar{f})\leq \sqrt{n}\, \Lip (f)$; the existence 
of an extension with the same Lipschitz constant is a classical,
but subtle, result of Kirszbraun, see 2.10.43 in \cite{Fed}).
We now assume that the lemma is true for every $Q<Q^*$, and prove it
for $Q^*$. 

Fix any $x_0\in\partial B_1$. We distinguish two cases:
either (\ref{e.15}) of Proposition \ref{p.10} is satisfied with
$B=\de B_1$, or it is not. In the first case we can decompose $h$ as
$\a{h_L}+\a{h_K}$, where $h_L$ and $h_K$ are Lipschitz
functions taking values in $\I{L}$ and $\I{K}$, and $K$ and
$L$ are positive integers.
By the induction hypothesis, we can find extensions of $h_L$
and $h_K$ satisfying the requirements of the
lemma, and it is not
difficult to verify that $f=\a{f_L}+\a{f_K}$ is the desired extension
of $h$ to $\overline{B_1}$.

In the second case, for any pair of indices $i,j$ we have that
\[
\abs{h_i(x_0)-h_j(x_0)}\leq 6\,Q^*\,\Lip(h).
\]
We use the following cone-like construction:
set $P:=h_1(x_0)$ and define
\begin{equation}\label{e:cone}
f(x)=\sum_i\a{\abs{x} h_i\left(\frac{x}{|x|}\right)+\big(1-\abs{x}\big)\,P}.
\end{equation}
Clearly $f$ is an extension of $h$. For the Lipschitz regularity,
note first that
\begin{equation*}
\Lip(f\vert_{\de B_r})=\Lip(h),\;\textrm{for every }0<r\leq1. 
\end{equation*} 
Next, for any $x\in\partial B$, on the segment $\sigma_x = [0,x]$ 
we have
\begin{equation*}
\Lip{f\vert_{\sigma_x}}\leq 
Q^*\max_{i}\abs{h_i(x)-P}\leq 6\,(Q^*)^2\, \Lip(h).
\end{equation*}
So, we infer that
$\Lip(f)\leq12\,(Q^*)^2\,\Lip(h)$.
Moreover, \eqref{e:LinftyBound2} follows easily from
\eqref{e:cone}.
\end{proof}

\subsection{Proof of Theorem \ref{thm.ext}}\label{ss:ext}
Without loss of generality, we can assume that $B$ is closed.
Consider a Whitney decomposition 
 $\left\{C_k\right\}_{k\in\N}$ 
of $\R{m}\setminus B$ (see Figure \ref{f:whitney}). More precisely
(cp. with Theorem 3, page 16 of \cite{St1}):
\begin{itemize}
\item[(W1)] each $C_k$ is a closed dyadic cube, i.e.
the length $l_k$ of the side is $2^{k}$ for
some $k\in \Z$ and the coordinates of the vertices
are integer multiples of $l_k$;
\item[(W2)] distinct cubes have disjoint interiors and
\begin{equation}\label{e:comparable}
c(m)^{-1}\dist(C_k,B)\leq l_k\leq c(m)\, \dist(C_k,B).
\end{equation}
\end{itemize}
As usual, we call $j$-skeleton the union of
the $j$-dimensional faces of 
$C_k$. We now construct the extension $\bar{f}$ by defining
it recursively on the skeletons.

\begin{figure}[htbp]
\begin{center}
    \input{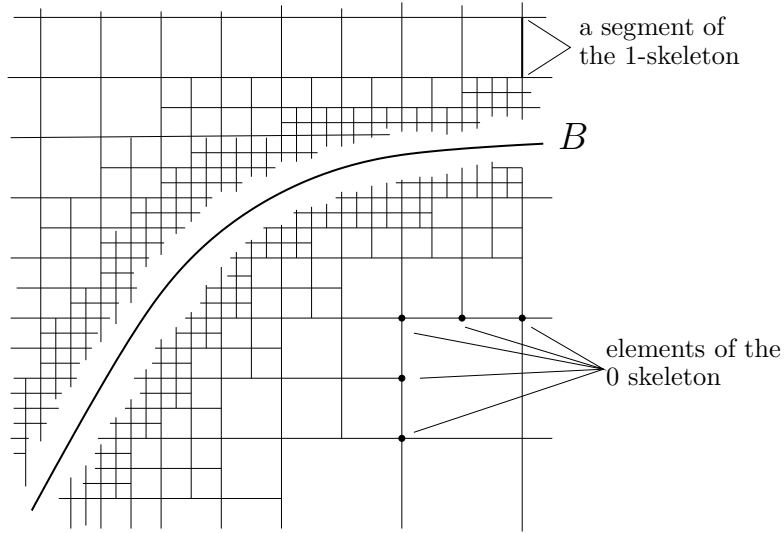}
    \caption{The Whitney decomposition of $\R{2}\setminus B$.}
    \label{f:whitney}
\end{center}
\end{figure}

Consider the $0$-skeleton, i.e. the set of the
vertices of the cubes. For each vertex $x$,
we choose $\tilde{x}\in B$
such that $\abs{x-\tilde{x}}=\dist(x,B)$ and set
$\bar{f}(x)=f(\tilde{x})$.
If $x$ and $y$ are two adjacent vertices of the same cube
$C_k$, then
\begin{equation*}
\max\big\{\abs{x-\tilde{x}}, \abs{y-\tilde{y}}\big\}
\leq \dist (C_k, B) \leq c\, l_k = c\, \abs{x-y}.
\end{equation*} 
Hence, we have
\begin{align*}
\cG \left(\bar{f}(x), \bar{f}(y)\right)&=
\cG \left(f(\tilde{x}),f(\tilde{y})\right)
\leq\Lip(f)\abs{\tilde{x}-\tilde{y}}
\leq\Lip(f)\big(\abs{\tilde{x}-x}+\abs{x-y}+\abs{y-\tilde{y}}\big)\nonumber\\
&\leq c\,\Lip(f)\abs{x-y}\label{e:est1}.
\end{align*}
Using the Homotopy Lemma \ref{l:hom},
we extend $f$ to $\bar{f}$ on each side of the $1$-skeleton .
On the boundary of any $2$-face 
$\bar{f}$ has Lipschitz constant smaller than
$9\, C(m,Q)\,\Lip(f)$. Applying Lemma \ref{l:hom}
recursively we find an extension of
$\bar f$ to all $\R{m}$ such that \eqref{e:LinftyBound} holds
and which is Lipschitz in
each cube of the decomposition,
with constant smaller than $C(m,Q)\, \Lip(f)$.

It remains to show that $\bar f$ is Lipschitz on the whole
$\R{m}$. Consider $x,\,y\in\R{m}$, not lying in the
same cube of the decomposition. Our aim is to 
show the inequality 
\begin{equation}\label{e:goalexte}
\cG \left(\bar{f} (x), \bar{f} (y)\right)\leq C\,\Lip(f)\, |x-y|,
\end{equation} with some $C$ depending only on $m$ and $Q$. 
Without loss of generality, we can assume that
$x\not\in B$. 
We distinguish then two possibilities:
\begin{itemize}
\item[(a)] $\left[x,y\right]\cap B\neq\emptyset$;
\item[(b)] $\left[x,y\right]\cap B = \emptyset$.
\end{itemize}
In order to deal with (a), assume first that $y\in B$.
Let $C_k$ be a cube of the decomposition containing $x$ and 
let $v$ be one of the nearest vertices
of $C_k$ to $x$. Recall, moreover, that $\bar{f} (v)
= f (\tilde{v})$ for some $\tilde{v}$ with $|\tilde{v}-v|=
\dist (v, B)$. We have then
\begin{align*}
\cG \left(\bar{f}(x),\bar f(y)\right)&\leq
\cG\left(\bar{f}(x),\bar{f}(v)\right)+\cG\left(\bar f(v),f(y)\right)
=\cG\left(\bar{f}(x),\bar{f}(v)\right)+\cG \left(f(\tilde v), f(y)\right)\\
&\leq C\,\Lip(f)\abs{x-v}+\Lip(f)\abs{\tilde v-y}\\
&\leq C\,\Lip(f)\big(\abs{x-v}+\abs{\tilde
    v-v}+\abs{v-x}+\abs{x-y}\big)\\
&\leq C\,\Lip(f)\big(l_k+\dist(C_k,B)+{\rm diam}\, (C_k)
+\abs{x-y}\big)\\
&\stackrel{\mathclap{\eqref{e:comparable}}}{\leq}
C\,\Lip(f)\abs{x-y}.
\end{align*}
If (a) holds but $y\not \in B$, then let $z\in ]a,b[\cap B$.
From the previous argument we know $\cG (\bar{f} (x), \bar{f} (z))
\leq C |x-z|$ and $\cG \left(\bar{f} (y), \bar{f} (z)\right)\leq C|y-z|$,
from which \eqref{e:goalexte} follows easily.

If (b) holds, then $\left[x,y\right]=\left[x,P_1\right]
\cup\left[P_1,P_2\right]\cup\ldots\cup\left[P_s,y\right]$
where each interval belongs to a cube of the decomposition.
Therefore \eqref{e:goalexte}
follows trivially from the Lipschitz estimate for
$\bar{f}$ in each cube of the decomposition.

\section{Differentiability and Rademacher's Theorem}
\label{s:rademacher}

In this section we introduce the notion of differentiability for
$Q$-valued functions and prove two related theorems. 
The first one gives chain-rule formulas for $Q$-valued
functions and the
second is the extension to the $Q$-valued setting of the
classical result of Rademacher. 

\begin{definition}\label{d:diff}
Let $f:\Om\ra\Iq$ and $x_0\in\Om$. We say that $f$ is differentiable
at $x_0$ if there exist $Q$ matrices $L_i$ satisfying:
\begin{itemize}
\item[(i)]$\cG(f(x),T_{x_0} f)=o(\abs{x-x_0})$, where
\begin{equation}\label{e:taylor1st}
T_{x_0} f(x):=\sum_i\a{L_i\cdot(x-x_0)+f_i(x_0)};
\end{equation}
\item[(ii)] $L_i=L_j$ if $f_i(x_0)=f_j(x_0)$.
\end{itemize}
The $Q$-valued map $T_{x_0} f$
will be called the {\em first-order approximation} of $f$ at $x_0$.
The point $\sum_i \a{L_i} \in \Iq (\R{n\times m})$
will be called the differential of $f$ at $x_0$ and is
denoted by $Df (x_0)$. 
\end{definition}

\begin{remark}\label{r:compdiff}
What we call ``differentiable'' is called
``strongly affine approximable'' by Almgren.
\end{remark}

\begin{remark}\label{r:convention}
The differential $Df(x_0)$ of a $Q$-function $f$
does not determine univocally
its first-order approximation $T_{x_0}f$.
To overcome this ambiguity, we write $Df_i$ for $L_i$ 
in Definition \ref{d:diff},
thus making evident which matrix has to be associated 
to $f_i(x_0)$ in (i). Note that (ii) implies that this notation
is consistent: namely, if $g_1, \ldots, g_Q$ is a different 
selection for $f$, $x_0$ a point of differentiability
and $\pi$ a permutation such that $g_i (x_0) = f_{\pi (i)} (x_0)$
for all $i\in \{1, \ldots, Q\}$,
then $Dg_i (x_0) = Df_{\pi (i)} (x_0)$. Even though
the $f_i$'s are not, in general, differentiable,
observe that, when they are differentiable and $f$ is differentiable,
the $Df_i$'s coincide with the classical differentials.
\end{remark}

If $D$ is the set of points of differentiability
of $f$, the map $x\mapsto Df (x)$ is a $Q$-valued map, which we denote
by $Df$.
In a similar fashion, we define the directional
derivatives
$\partial_\nu f (x) = \sum_i \a{Df_i (x)\cdot\nu}$
and establish the notation $\partial_\nu f = \sum_i \a{\partial_\nu f_i}$.

\subsection{Chain rules}\label{ss:chain}
In what follows, we will deal with several
natural operations defined on $Q$-valued functions.
Consider a function $f:\Om \to \Iqs$.
For every $\Phi: \tilde{\Omega}\to \Omega$,
the right composition $f\circ \Phi$ defines
a $Q$-valued function on $\tilde{\Omega}$. 
On the other hand, given a map $\Psi: \Omega \times \R{n} \to \R{k}$, we
can consider the left composition,
$x \mapsto \sum_i \a{\Psi (x,f_i (x))}$,
which defines a $Q$-valued function denoted, with
a slight abuse of notation, by $\Psi (x,f)$. 

The third operation involves
maps $F: (\R{n})^Q\to \R{k}$
such that, for every $Q$ points $(y_1,\ldots,y_Q)\in\left(\R{n}\right)^Q$ and
$\pi\in\Pe_Q$,
\begin{equation}\label{e:symmetry}
F(y_1, \ldots, y_Q)
= F\left(y_{\pi (1)}, \ldots, y_{\pi (Q)}\right).
\end{equation}
Then, $x\mapsto F (f_1 (x), \ldots, f_Q (x))$ is
a well defined map, denoted by $F\circ f$.

\begin{propos}[Chain rules]\label{p:chain}
Let $f:\Omega\to \Iqs$ be differentiable at $x_0$.
\begin{itemize}
\item[$(i)$] Consider $\Phi: \tilde{\Omega}\to \Omega$
such that $\Phi (y_0)=x_0$ and assume that
$\Phi$ is differentiable at $y_0$. 
Then, $f\circ \Phi$ is differentiable at $y_0$ and
\begin{equation}\label{e:interna}
D (f\circ \Phi) (y_0) = \sum_i 
\a{D f_i (x_0)\cdot D\Phi (y_0)}.
\end{equation}
\item[$(ii)$] Consider $\Psi: \Omega_x\times \R{n}_u\to \R{k}$
such that $\Psi$ is differentiable at $(x_0, f_i (x_0))$
for every $i$. Then, $\Psi (x,f)$ is differentiable at $x_0$ and
\begin{equation}\label{e:esterna}
D\Psi (x,f)) (x_0)=\sum_i 
\a{D_u\Psi (x_0, f_i (x_0))\cdot Df_i (x_0)
+ D_x \Psi (x_0, f_i (x_0))}.
\end{equation}
\item[$(iii)$] Consider $F:(\R{n})^Q\to
\R{k}$ as in \eqref{e:symmetry} and differentiable at $(f_1 (x_0), \ldots, f_Q (x_0))$.
Then, $F \circ f$ is differentiable at $x_0$ and
\begin{equation}\label{e:reale}
D (F\circ f) (x_0)=
\sum_i D_{y_i} F (f_1 (x_0), \ldots, f_Q (x_0))
\cdot Df_i (x_0).
\end{equation}
\end{itemize}
\end{propos}

\begin{proof} All the formulas are just routine modifications
of the classical chain-rule.
The proof of $(i)$ follows easily from Definition \ref{d:diff}.
Since $f$ is differentiable at $x_0$, we have
\begin{align}\label{e:espansione1}
\cG \left( f \circ \Phi (y), 
\sum_i \a{Df_i (x_0)\cdot (\Phi (y)-\Phi (y_0)) +f_i (\Phi
(y_0))}\right)&= o\, (|\Phi (y)-\Phi (y_0)|)\notag\\
&= o\, (|y-y_0|),
\end{align}
where the last equality follows from the differentiability
of $\Phi$ at $y_0$. Moreover, again due to
the differentiability of $\Phi$,
we infer that
\begin{equation}\label{e:espansione2}
Df_i (x_0)\cdot (\Phi (y)-\Phi (y_0))
= Df_i (x_0)\cdot D\Phi (y_0)\cdot (y-y_0)
+ o (|y-y_0|).
\end{equation}
Therefore, \eqref{e:espansione1} and \eqref{e:espansione2}
imply \eqref{e:interna}.

For what concerns $(ii)$, we note that
we can reduce to the case of $\card(f(x_0))=1$,
i.e.
\begin{equation}\label{e:mult1}
f(x_0)=Q\a{u_0}\quad\textrm{and}\quad Df(x_0)=Q\a{L}.
\end{equation}
Indeed, since $f$ is differentiable (hence, continuous) in $x_0$, 
in a neighborhood of $x_0$ we can decompose $f$ as the sum of 
differentiable multi-valued functions $g_k$,
$f=\sum_k\a{g_k}$, such that $\card(g_k(x_0))=1$.
Then, $\Psi(x,f)=\sum_k\a{\Psi(x,g_k)}$ in a neighborhood of $x_0$,
and the differentiability of $\Psi(x,f)$ follows from the differentiability of the $\Psi(x,g_k)$'s.
So, assuming \eqref{e:mult1}, without loss of generality,
we have to show that
\begin{equation*}
h(x) 
=Q\a{D_u\,\Psi (x_0,u_0)\cdot L
\cdot (x-x_0)+D_x\,\Psi(x_0,u_0)\cdot (x-x_0)+\Psi(x_0,u_0)}
\end{equation*}
is the first-order approximation of $\Psi(x,f)$ in $x_0$.
Set
\[
A_i(x)=D_u\,\Psi (x_0,u_0)\cdot (f_i(x)-u_0)
+D_x\,\Psi(x_0,u_0)\cdot (x-x_0)+\Psi(x_0,u_0).
\]
From the differentiability of $\Psi$, we deduce that
\begin{equation}\label{e:diff psi}
\cG\left(\Psi(x,f),\sum_i\a{A_i(x)}\right)
=o\big(|x-x_0|+\cG(f(x),f(x_0))\big)=o\left(|x-x_0|\right),
\end{equation}
where we used the differentiability of $f$ in the last step.
Hence, we can conclude \eqref{e:esterna}, i.e.
\begin{align*}
\cG\left(\Psi(x,f),h(x)\right)&\leq 
\cG\left(\Psi(x,f),\sum_i\a{A_i(x)}\right)
+\cG\left(\sum_i\a{A_i(x)},h(x)\right)\\
&\leq o\left(|x-x_0|\right)+\|D_u\,\Psi(x_0,u_0)\| 
\cG\left(\sum_i\a{f_i(x)},Q\a{L\cdot(x-x_0)+u_0}\right)\\
&= o\left(|x-x_0|\right).
\end{align*}
where $\|D_u\psi(x_0,u_0)\|$ denotes the Hilbert--Schmidt
norm of the matrix $D_u\,\Psi(x_0,u_0)$.

Finally, to prove $(iii)$, fix $x$ and let $\pi$ be such that
\begin{equation*}
\cG\big(f(x),f(x_0)\big)^2=\sum_i|f_{\pi (i)} (x)-f_i(x_0)|^2.
\end{equation*}
By the continuity of $f$ and (ii) of Definition \ref{d:diff},
for $|x-x_0|$ small enough we have
\begin{equation}\label{e:selection is good}
\cG\big(f(x), T_{x_0} f (x)\big)^2
= \sum_i |f_{\pi (i)} (x) - Df_i (x_0)\cdot (x-x_0) - z_i|^2.
\end{equation}
Set $f_i(x_0)=z_i$ and $z=(z_1,\ldots,z_Q)\in\left(\R{n}\right)^Q$.
The differentiability of $F$ implies
\begin{equation}\label{e:diff F}
\abs{F\circ f(x)-F\circ f(x_0)-
\sum_i D_{y_i}F(z)\cdot\left(f_{\pi (i)} (x)-z_i\right)}=
o\left(\cG(f(x),f(x_0)\right)=o(\abs{x-x_0}).
\end{equation}
Therefore, for $|x-x_0|$ small enough, we conclude
\begin{multline}\label{e:diff f}
\abs{\sum_i D_{y_i}F(z)\cdot\big(f_{\pi (i)} (x)-z_i-
Df_i(x_0)\cdot(x-x_0)\big)}\leq\\
\leq C\sum_i|f_{\pi (i)} (x)-Df_i(x_0)\cdot(x-x_0)-z_i|
\stackrel{\eqref{e:selection is good}}{=} o (|x-x_0|),
\end{multline}
with $C=\sup_i\left\|D_{y_i}F(z)\right\|$.
Therefore, using \eqref{e:diff F} and \eqref{e:diff f},
we conclude \eqref{e:reale}.
\end{proof}

\subsection{Rademacher's Theorem} In this
subsection we extend
the classical theorem of Rademacher on the differentiability
of Lipschitz functions to the $Q$-valued setting. 
Our proof is direct and elementary,
whereas in Almgren's work the theorem is a corollary
of the existence of the biLipschitz embedding $\xii$
(see Section \ref{s:biLipschitz}). An intrinsic
proof has been already proposed in \cite{Gob1}. However
our approach is considerably simpler.

\begin{theorem}[Rademacher]\label{t:rademacher}
Let $f:\Om\ra\Iq$ be a Lipschitz function. Then, $f$ is differentiable
almost everywhere in $\Om$.
\end{theorem}
\begin{proof}
We proceed by induction on the number of values $Q$.
The case $Q=1$ is the classical Rademacher's theorem
(see, for instance, $3.1.2$ of \cite{EG}). We next assume that the
theorem is true for every $Q<Q^*$ and we show its validity for $Q^*$.

We write $f=\sum_{i=1}^{Q^*}\a{f_i}$, where the $f_i$'s are
a measurable selection. We let $\tilde{\Omega}$ be the set of points
where $f$ takes a single value with multiplicity $Q$:
\[
\tilde\Om=\big\{x\in\Om\,:\,f_1(x)=f_i(x)\;\forall i\big\}.
\]
Note that $\tilde\Om$ is closed. In $\Om\setminus\tilde\Om$, $f$ is 
differentiable almost everywhere by
inductive hypothesis. Indeed, by Proposition \ref{p.10}, 
in a neighborhood of any point $x\in \Om\setminus\tilde\Om$,
we can decompose $f$ in the sum of two Lipschitz
simpler multi-valued functions, $f=\a{f_L}+\a{f_K}$, with
the property that $\supp (f_L (x))\cap \supp (f_K (x))=\emptyset$.
By inductive hypothesis, $f_L$ and $f_K$ are
differentiable, hence, also $f$ is.

It remains to prove that $f$ is differentiable a.e. in
$\tilde\Om$. Note that $f_1|_{\tilde\Om}$ is a Lipschitz
vector valued function and consider a Lipschitz extension of it to all $\Om$, denoted by $g$.
We claim that $f$ is differentiable in all the points $x$ where 
\begin{itemize}
\item[(a)] $\tilde\Om$ has density $1$;
\item[(b)] $g$ is differentiable.
\end{itemize} 
Our claim would conclude the proof. In order to show it,
let $x_0\in\tilde\Om$ be any given point fulfilling (a) and (b)
and let $T_{x_0}g (y)=L\cdot (y-x_0)+f_1(x_0)$ be the first order Taylor
expansion of $g$ at $x_0$, that is
\begin{equation}\label{e.25}
\abs{g(y)-L\cdot(y-x_0)-f_1(x_0)}=o(|y-x_0|).
\end{equation}
We will show that $T_{x_0}f (y) := Q\a{L\cdot(y-x_0)+f_1(x_0)}$
is the first order expansion of $f$ at $x_0$.
Indeed, for every
$y\in\R{m}$, let $r= |y-x_0|$ and choose 
$y^*\in\tilde\Om\cap\overline{B_{2r}(x_0)}$
such that
\[
\abs{y-y^*}=\dist\left(y,\tilde\Om\cap\overline{B_{2r}(x_0)}\right).
\]
Being $f$, $g$ and $Tg$ Lipschitz with constant at
most $\Lip (f)$,
using (\ref{e.25}), we infer that 
\begin{align}
\cG\big(f(y),T_{x_0}f(y)\big)\leq{}&
\cG\big(f(y),f(y^*)\big)+\cG\big(T_{x_0}f (y^*), T_{x_0}f(y)\big)
+\cG\big(f(y^*),T_{x_0}f(y^*)\big)\nonumber\\
\leq{}&\Lip(f)\abs{y-y^*}+Q\,\Lip(f)\abs{y-y^*}+\nonumber\\
&+\cG\big(Q\a{g(y^*)},Q\a{L\cdot(y^*-x_0)+f_1(x_0)}\big)
\nonumber\\
\leq{}&(Q+1)\,\Lip(f)\abs{y-y^*}+o\big(\abs{y^*-x_0}\big)\label{e:3.20}.
\end{align}
Since $|y^*-x_0|\leq 2\,r=2\,|y-x_0|$, it remains to
estimate $\rho:= |y-y^*|$. Note that the ball
$B_\rho (y)$ is contained in $B_r (x_0)$ and does not intersect
$\tilde{\Omega}$. Therefore
\begin{equation}
\abs{y-y^*}=\rho\leq C\,\left|B_{2r} (x_0)
\setminus \tilde{\Omega}\right|^{1/m}\leq
C(m)\, r 
\left(\frac{|B_{2r} (x_0)\setminus\tilde\Om|}
{|B_{2r} (x_0)|}\right)^\frac{1}{m}.
\label{e:3.21}
\end{equation}
Since $x_0$ is a point of density $1$, we can conclude
from \eqref{e:3.21} that $|y-y^*|= |y-x_0|\, o (1)$.
Inserting this inequality in \eqref{e:3.20}, we conclude 
that $\cG (f(y), T_{x_0}f (y))= o (|y-x_0|)$,
which shows that $T_{x_0} f$ is the first order expansion
of $f$ at $x_0$.
\end{proof}

\chapter{Almgren's extrinsic theory}

Two ``extrinsic maps'' play a pivotal role in the theory of 
$Q$-functions developed in \cite{Alm}. The first one is 
a biLipschitz embedding $\xii$ of $\Iqs$ into $\R{N(Q,n)}$,
where $N(Q,n)$ is a sufficiently large integer. Almgren 
uses this map to define Sobolev $Q$-functions as classical
$\R{N}$-valued Sobolev maps taking values 
in $\cQ:=\xii (\Iqs)$. Using $\xii$, many
standard facts of Sobolev maps can be extended
to the $Q$-valued setting with little effort.
The second map
$\ro$ is a Lipschitz retraction of $\R{N(Q,n)}$ onto
$\cQ$, which is used in various approximation
arguments. 

The existence
of the maps $\xii$ and $\ro$ is proved in
Section \ref{s:biLipschitz}. In Section
\ref{s:Qsobolev} we show that Sobolev $Q$-valued functions 
in the sense of Almgren coincide with
those of Definition \ref{d:W1p} and we use $\xii$
to derive their basic properties. Finally,
Section \ref{s:existence} shows that our definition
of Dirichlet's energy coincides with Almgren's one and
proves the Existence Theorem \ref{t:existence}.
Except for Section \ref{s:Qsobolev}, no other
portion of this paper makes direct use 
of $\xii$ or of $\ro$: the regularity theory
of Chapters 3 and 5 needs only 
the propositions stated in Section \ref{s:Qsobolev},
which we are going to prove again in Chapter 4 within the frame of an ``intrinsic''
approach, that is independent of $\xii$ and $\ro$.

\section{The biLipschitz embedding $\xii$ and
the retraction $\ro$}
\label{s:biLipschitz}

\begin{theorem}\label{t:xi}
There exist $N=N(Q,n)$ and an injective map
$\xii:\Iqs\ra\R{N}$ such that:
\begin{itemize}
\item[$(i)$] $\Lip(\xii)\leq1$;
\item[$(ii)$] if $\cQ= \xii(\Iq)$, then
$\Lip(\xii^{-1}|_\cQ)\leq C(n,Q)$.
\end{itemize}
Moreover, there exists a Lipschitz map $\ro:\R{N}\to\cQ$
which is the identity on $\cQ$.
\end{theorem}

The existence of $\ro$ is a trivial consequence of 
the Lipschitz regularity of $\xii^{-1}|_\cQ$ 
and of the Extension Theorem \ref{thm.ext}.

\begin{proof}[Proof of the existence of $\ro$ given $\xii$]
Consider $\xii^{-1}: \cQ\to \Iq$. Since
this map is Lipschitz, by Theorem \ref{thm.ext} there exists
a Lipschitz extension $f$ of $\xii^{-1}$ to the entire space.
Therefore, $\ro = \xii\circ f$ is the desired retraction.
\end{proof}

For the proof of the first part
of Theorem \ref{t:xi}, we follow instead the
ideas of Almgren.
A slight modification of these ideas,
moreover, leads to the construction of a special biLipschitz
embedding: this observation, due to B. White,
was noticed in \cite{Ch}.

\begin{corol}\label{c:BW}
There exist $M=M(Q,n)$ and an injective map
$\xii_{BW}:\Iqs\to\R{M}$ with the following properties:
$\xii_{BW}$ satisfies (i) and (ii) of Theorem \ref{t:xi} 
and, for every $T\in\Iqs$, there exists $\delta>0$ such that
\begin{equation}\label{e:BW}
|\xii_{BW}(T)-\xii_{BW}(S)|=\cG(T,S)\quad\forall\;S\in B_{\delta}(T)\subset\Iqs.
\end{equation}
\end{corol}

We point out that we will not make any use in the following of such special embedding
$\xii_{BW}$, since all the properties of $Q$-valued functions are independent of the
embedding we choose. Nevertheless, we give a proof of Corollary \ref{c:BW} 
because it provides a better intuition on 
$Q$-valued functions (see Proposition
\ref{p:dir BW}) and can be used to give shorter proofs of 
several technical lemmas (see \cite{DLSp}).

\subsection{A combinatorial Lemma} The key of
the proof of Theorem \ref{t:xi} 
is the following combinatorial statement.

\begin{lemma}[Almgren's combinatorial Lemma]\label{l:comb}
There exist $\alpha=\alpha(Q,n)>0$ and a set of
$h=h(Q,n)$ unit vectors $\Lambda=\{e_1, \ldots e_h\}\subset
\s^{n-1}$ with the following property: given any set of
$Q^2$ vectors, $\left\{v_1,\ldots,v_{Q^2}\right\}\subset\R{n}$, 
there exists $e_l\in\Lambda$ such that
\begin{equation}\label{e:combineq}
\abs{v_k\cdot\e_l}\geq\alpha\abs{v_k}\quad\textrm{for all }k\in\big\{1,\ldots,Q^2\big\}.
\end{equation}
\end{lemma}
\begin{proof} Choose a unit vector $e_1$ and
let $\alpha(Q,n)$ be small enough in order to ensure that the set
$E := \left\{x\in\s^{n-1}\,:\,\abs{x\cdot e_1}<\alpha\right\}$
has sufficiently small measure, that is
\begin{equation}\label{e:alpha}
\cH^{n-1}(E)\leq\frac{\cH^{n-1}(\s^{n-1})}{8\cdot 5^{n-1}\,Q^2}.
\end{equation}
Note that $E$ is just the $\alpha$-neighborhood of an equatorial 
$(n-2)$-sphere
of $\s^{n-1}$. Next, we use Vitali's covering Lemma
(see $1.5.1$ of \cite{EG}) to find 
a finite set $\Lambda=\{e_1, \ldots, e_h\}\subset \s^{n-1}$
and a finite number of radii $0< r_i <\alpha$ such that
\begin{itemize}
\item[$(a)$] the balls $B_{r_i} (e_i)$ are disjoint; 
\item[$(b)$] the balls $B_{5\,r_i}(e_i)$ cover the whole sphere.
\end{itemize}
We claim that $\Lambda$ satisfies the requirements of the lemma.
Let, indeed, $V=\left\{v_1,\ldots,v_{Q^2}\right\}$ be a set of
vectors. We want to show the existence of $e_l\in \Lambda$ which
satisfies \eqref{e:combineq}. Without loss of generality, 
we assume that each $v_i$ is nonzero.
Moreover, we consider the
sets $C_k=\big\{x\in\s^{n-1}:\abs{x\cdot v_k}<\alpha\abs{v_k}\big\}$
and we let $C_V$ be the union of the $C_k$'s.
Each $C_k$ is the $\alpha$-neighborhood of the equatorial sphere
given by the intersection of $\s^{n-1}$ with the hyperplane orthogonal to
$v_i$. Thus, by (\ref{e:alpha}), 
\begin{equation}\label{e:less}
\cH^{n-1}\left(C_V\right)\leq
\frac{\cH^{n-1}(\s^{n-1})}{8\cdot 5^{n-1}}.
\end{equation}
Note that, due to the bound $r_i<\alpha$, 
\begin{equation}\label{e:more}
e_i\in C_V \quad\Rightarrow\quad
\cH^{n-1}\left(C_V\cap B_{r_i} (e_i)\right)\geq\frac{\cH^{n-1}(B_{r_i}(e_i)
\cap \s^{n-1})}{2}.
\end{equation}
By our choices, there must be one $e_l$ which does not belong to
$C_V$, otherwise
\begin{align*}
\frac{\cH^{n-1} (\s^{n-1})}{2\cdot 5^{n-1}}&\quad\stackrel{\mathclap{(a)\,\&\,(b)}}{\leq}\quad
\sum_i \cH^{n-1} \left(B_{r_i} (e_i)\cap\s^{n-1}\right)
\stackrel{\eqref{e:more}}{\leq}2\sum_i \cH^{n-1}\left(C_V\cap B_{r_i} (e_i)\right)\\
&\quad\stackrel{\mathclap{(a)}}{\leq}\quad 2 \,\cH^{n-1}\left(C_V\right)
\stackrel{\eqref{e:less}}{\leq}
\frac{\cH^{n-1}(\s^{n-1})}{4\cdot 5^{n-1}},
\end{align*}
which is a contradiction (here we used the fact that, though the sphere 
is curved, for $\alpha$ sufficiently small the 
$(n-1)$-volume of $B_{r_i} (e_i)\cap \s^{n-1}$
is at least $2^{-1}5^{-n+1}$ times the volume of $B_{5\,r_i} (e_i)
\cap \s^{n-1}$).
Having chosen $e_l\not\in C_V$, we have
$e_l\not \in C_k$ for every $k$, which in turn
implies \eqref{e:combineq}.
\end{proof}

\subsection{Proof of the existence of $\xii$}
Let $\Lambda =\{e_1, \ldots e_h\}$ be a
set satisfying the conclusion of Lemma \ref{l:comb} and set $N=Q\,h$.
Fix $T\in\Iqs$, $T=\sum_i \a{P_i}$. For any $e_l\in\Lambda$,
we consider the $Q$ projections of the points $P_i$ on the
$e_l$ direction, that is $P_i\cdot e_l$. This gives an array
of $Q$ numbers, which we rearrange in increasing order,
getting a $Q$-dimensional vector $\pi_l (T)$.
The map $\xii: \Iq\to \R{N}$ is, then, defined by
$\xii (T) = h^{-1/2} (\pi_1 (T), \ldots, \pi_h (T))$.

The Lipschitz regularity of $\xii$ is a trivial corollary
of the following rearrangement inequality:
\begin{itemize}
\item[(Re)] if $a_1\leq\ldots\leq a_n$ and $b_1\leq\ldots\leq b_n$, then,
for every permutation $\sigma$ of the indices,
\[
(a_1-b_1)^2+\cdots+(a_n-b_n)^2\leq 
(a_1 - b_{\sigma (1)})^2+\cdots+(a_n - b_{\sigma (n)})^2.
\] 
\end{itemize}
Indeed, fix two points $T=\sum_i\a{P_i}$ and $S=\sum_i\a{R_i}$
and assume, without loss of generality, that
\begin{equation}\label{e.54}
\cG(T,S)^2=\sum_i\abs{P_i-R_{i}}^2.
\end{equation}
Fix an $l$. Then, by (Re), 
$\abs{\pi_l(T)-\pi_l(S)}^2\leq \sum ((P_i-R_{i})\cdot\e_l)^2$.
Hence, we get
\begin{align*}
\abs{\xii(T)-\xii(S)}^2
&\leq \frac{1}{h}
\sum_{l=1}^h \sum_{i=1}^Q((P_i-R_{i})\cdot\e_l)^2
\leq\frac{1}{h}\sum_{l=1}^h \sum_{i=1}^Q\abs{P_i-R_{i}}^2\\
&\stackrel{\mathclap{\eqref{e.54}}}{=} \frac{1}{h}\sum_{l=1}^h \cG (T,S)^2 = \cG (T,S)^2.
\end{align*}
Next, we conclude the proof by showing, for $T=\sum_i\a{P_i}$ and $S=\sum_i\a{R_i}$,
the inequality
\begin{equation}\label{e:to show}
\cG(T,S)\leq\frac{\sqrt{h}}{\alpha}\, \abs{\xii(T)-\xii(S)},
\end{equation}
where $\alpha$ is the constant in Lemma \ref{l:comb}.
Consider, indeed, the $Q^2$ vectors
$P_i-R_j$, for $i,j\in\{1,\ldots, Q\}$.
By Lemma \ref{l:comb}, we can select a unit vector $e_l\in\Lambda$ such that
\begin{equation}\label{e:all}
\abs{\left(P_i-R_j\right)\cdot e_l}\geq\alpha\abs{P_i-R_j},
\quad\text{for all }\,i,j\in\{1,\ldots,Q\}.
\end{equation}
Let $\tau$ and $\lambda$ be permutations such that
\begin{equation*}
\pi_l(T) = (P_{\tau (1)}\cdot e_l, \ldots, P_{\tau (Q)}\cdot e_l)
\quad \mbox{and} \quad\pi_l (S) = (R_{\lambda (1)}\cdot e_l, \ldots, 
R_{\lambda (Q)} \cdot e_l).
\end{equation*}
Then, we conclude \eqref{e:to show},
\begin{align*}
\cG(T,S)^2&\leq\sum_{i=1}^Q\abs{P_{\tau(i)}-R_{\lambda(i)}}^2
\stackrel{\eqref{e:all}}{\leq}
\alpha^{-2}\sum_{i=1}^Q\left((P_{\tau(i)} - R_{\lambda(i)})\cdot e_l\right)^2=\alpha^{-2}\,\abs{\pi_l(T)-\pi_l(S)}^2\\
&\leq 
\alpha^{-2}\,h\,\abs{\xii(T)-\xii(S)}^2.
\end{align*}

\subsection{Proof of Corollary \ref{c:BW}} Let
$\Lambda=\{e_1, \ldots e_h\}$ be the set of unit vectors
in the proof of Theorem \ref{t:xi}.
We consider the enlarged set
$\Gamma$ of $n\,h$ vectors containing an orthonormal frame for
each $e_l\in\Lambda$,
\[
\Gamma=\big\{e_1^1,\ldots, e_1^n,\ldots,e_h^1,\ldots, e_h^n\big\},
\]
where, for every $\alpha\in\{1,\ldots, h\}$, $e_\alpha^1=e_\alpha$ and
$\{e_\alpha^1,\ldots, e_\alpha^n\}$ is an orthonormal basis of $\R{n}$.
Note that, in principle, the vectors $e_\alpha^\beta$ may not be all distinct: this can happen, for example,
if there exist two vectors $e_j$ and $e_l$ which are orthogonal.
Nevertheless,
we can assume, without loss of generality,
that $\Gamma$ is made of $n\,h$ distinct vectors
(in passing, this is can always be reached by perturbing slightly $\Lambda$).

Then, we define the map $\xii_{BW}$ in the same way as $\xii$, with $\Gamma$ replacing
$\Lambda$: for $T=\sum_i\a{P_i}$,
\[
\xii_{BW} (T) = h^{-1/2} \big(\pi_1^1 (T),\ldots,\pi_1^n(T),\ldots,\pi_h^1(T) \ldots, \pi_h^n (T)\big),
\]
where $\pi_\alpha^\beta(T)$ is the array
of $Q$ scalar products $P_i\cdot e_\alpha^\beta$ rearranged in increasing order.

Clearly, $\xii_{BW}$ satisfies the conclusion of Theorem \ref{t:xi}. We need only
to show \eqref{e:BW}.

To this aim, we start noticing that, given $T=\sum_i\a{P_i}\in\Iq$,
there exists $\delta>0$ with the following property:
for every $S=\sum_i\a{R_i}\in B_\delta(T)$
and every $\pi_\alpha^\beta$, assuming that
$\cG(T,S)^2=\sum_i|P_i-R_i|^2$,
there exists a permutation $\sigma_\alpha^\beta\in\Pe_Q$ such that
the arrays $(P_i\cdot e_\alpha^\beta)$ and $(R_i\cdot e_\alpha^\beta)$
are ordered increasingly by the same permutation
$\sigma_\alpha^\beta$, i.e.
\[
\pi_\alpha^\beta(T)=\left(P_{\sigma_\alpha^\beta(1)}\cdot e_\alpha^\beta,
\ldots,P_{\sigma_\alpha^\beta(Q)}\cdot e_\alpha^\beta\right)
\quad\text{and}\quad
\pi_\alpha^\beta(S)=\left(R_{\sigma_\alpha^\beta(1)}\cdot e_\alpha^\beta,
\ldots,R_{\sigma_\alpha^\beta(Q)}\cdot e_\alpha^\beta\right).
\]
It is enough to choose
$
4\,\delta=\min_{\alpha,\beta}\left\{|P_i\cdot e_\alpha^\beta-P_j\cdot e_\alpha^\beta|\,:\,
P_i\cdot e_\alpha^\beta\neq P_j\cdot e_\alpha^\beta\right\}$.
Indeed, let us assume that $R_{i}\cdot e_\alpha^\beta\leq R_{j}\cdot e_\alpha^\beta$.
Then, two cases occur:
\begin{itemize}
\item[(a)] $R_{j}\cdot e_\alpha^\beta- R_{i}\cdot e_\alpha^\beta\geq 2\delta$,
\item[(b)] $R_{j}\cdot e_\alpha^\beta- R_{i}\cdot e_\alpha^\beta< 2\delta$.
\end{itemize}
In case (a), since $S\in B_\delta(T)$,
we deduce that
$P_{i}\cdot e_\alpha^\beta\leq
R_{i}\cdot e_\alpha^\beta+\delta\leq R_{j}\cdot e_\alpha^\beta-\delta
\leq P_{j}\cdot e_\alpha^\beta$.
In case (b), instead, we infer that
$|P_{j}\cdot e_\alpha^\beta- P_{i}\cdot e_\alpha^\beta|\leq
R_{j}\cdot e_\alpha^\beta+\delta- R_{i}\cdot e_\alpha^\beta-\delta<4\,\delta$,
which, in turn, by the choice of $\delta$,
leads to $P_{j}\cdot e_\alpha^\beta=P_{i}\cdot e_\alpha^\beta$.
Hence, in both cases we have $P_{i}\cdot e_\alpha^\beta\leq P_{j}\cdot e_\alpha^\beta$,
which means that $P_i\cdot e_\alpha^\beta$ can be ordered in increasing way
by the same permutation $\sigma_\alpha^\beta$.

Therefore, exploiting the fact that the vectors $\pi_\alpha^\beta(T)$
and $\pi_\alpha^\beta(S)$ are ordered by the same permutation $\sigma_\alpha^\beta$,
we have that, for $T$ and $S$ as above, it holds
\begin{align*}
|\xii_{BW}(T)-\xii_{BW}(S)|^2&=
h^{-1}\sum_{\alpha=1}^h\sum_{\beta=1}^n|\pi_\alpha^\beta(T)-\pi_\alpha^\beta(S)|^2\\
&=h^{-1}\sum_{\alpha=1}^h\sum_{\beta=1}^n\sum_{i=1}^Q
|P_{\sigma_\alpha^\beta(i)}\cdot e_\alpha^\beta-R_{\sigma_\alpha^\beta(i)}\cdot e_\alpha^\beta|^2
=h^{-1}\sum_{\alpha=1}^h\sum_{i=1}^Q|P_i-R_i|^2\\
&=h^{-1}\sum_{\alpha=1}^h\cG(T,S)^2=\cG(T,S)^2.
\end{align*}
This concludes the proof of the corollary.

\section{Properties of $Q$-valued Sobolev functions}\label{s:Qsobolev}

In this section we prove some of the basic properties of
Sobolev $Q$-functions which will be used in the proofs of the regularity theorems.
It is clear that, using $\xii$,
one can identify measurable, Lipschitz and H\"older
$Q$-valued functions $f$ with the corresponding maps
$\xii\circ f$ into $\R{N}$, which are, respectively,
measurable, Lipschitz, H\"older functions
taking values in $\cQ$ a.e.
We now show 
that the same holds for the Sobolev classes
of Definition \ref{d:W1p}.

\begin{theorem}\label{t:Sob=Sob}
Let $\xii$ be the map of Theorem \ref{t:xi}. Then, a $Q$-valued 
function $f$ belongs to the Sobolev space $W^{1,p}(\Om,\Iq)$
according to Definition \ref{d:W1p} if and only if 
$\xii\circ f$ belongs to $W^{1,p}(\Om, \R{N})$.
Moreover, there exists a constant $C=C(n,Q)$ such that
\[
|D(\xii\circ f)|\leq |Df|\leq C\,|D(\xii\circ f)|. 
\]
\end{theorem}
\begin{proof} Let $f$ be a $Q$-valued function such that
$g=\xii\circ f\in W^{1,p}(\Om, \R{N})$. Note that 
the map $\Upsilon_T:\cQ\ni y \mapsto \cG (\xii^{-1} (y), T)$
is Lipschitz, with a Lipschitz constant $C$ that can be
bounded independently of $T\in\Iq$.
Therefore, $\cG (f, T) =\Upsilon_T\circ g$
is a Sobolev function and 
$\left|\de_j \left(\Upsilon_T \circ g\right)\right|
\leq C |\de_j g|$ for every $T\in\Iq$. So, $f$ fulfils
the requirements (i) and (ii) of
Definition \ref{d:W1p}, with $\varphi_j 
= C\, \left|\de_j g\right|$, from which, in particular, $|Df|\leq C\,|D(\xii\circ f)|$.

Vice versa, assume that $f$ is
in $W^{1,p}(\Om,\Iq)$ and let $\varphi_j$ be as in Definition
\ref{d:W1p}. Choose a countable dense subset
$\{T_i\}_{i\in N}$ of $\Iq$, and recall
that any Lipschitz real-valued 
function $\Phi$ on $\Iq$ can be written as
\[
\Phi (\cdot)= \sup_{i\in\N} \big\{\Phi(T_i) - \Lip \left(\Phi\right)\, \cG (\cdot, T_i)\big\}.
\]
This implies that $\de_j \left(\Phi\circ f\right)\in L^p$ with 
$\left|\de_j \left(\Phi \circ f\right)\right|\leq \Lip (\Phi)\, \varphi_j$. Therefore, since
$\Omega$ is bounded,
$\Phi\circ f\in W^{1,p}(\Om)$.
Being $\xii$ a Lipschitz map with $\Lip(\xii)\leq 1$,
we conclude that $\xii\circ f\in W^{1,p}(\Om,\R{N})$ with $|D(\xii\circ f)|\leq |Df|$.
\end{proof}

We now use the theorem above to transfer in a straightforward way
several classical properties
of Sobolev spaces to the framework of $Q$-valued mappings.
In particular, in the subsequent subsections we
deal with Lusin type approximations,
trace theorems, Sobolev and Poincar\'e inequalities,
and Campanato--Morrey estimates. Finally Subsection \ref{ss:tech}
contains a useful technical lemma estimating the energy
of interpolating functions on spherical shells.

\subsection{Lipschitz approximation and approximate
differentiability}\label{ss:app}
We start with the Lipschitz approximation property for
$Q$-valued Sobolev functions.

\begin{propos}[Lipschitz approximation]\label{p:lipapprox}
Let $f$ be in $W^{1,p}(\Om,\Iq)$.
For every $\lambda>0$, there exists a Lipschitz $Q$-function
$f_\lambda$ such that ${\rm Lip}\, (f_\lambda) \leq \lambda$
and
\begin{equation}\label{e:approx}
\abs{\big\{x\in\Om\,:\,f(x)\neq
f_\lambda(x)\big\}}\leq \frac{C}{\lambda^p}\int_{\Omega} 
\big(|Df|^p+\cG(f,Q\a{0})^p\big),
\end{equation}
where the constant $C$ depends only on $Q$, $m$ and $\Omega$.
\end{propos}

\begin{proof}
Consider $\xii\circ f$: by the Lusin-type approximation theorem for 
classical Sobolev functions (see, for instance, \cite{AceFus} or $6.6.3$ of \cite{EG}), 
there exists a Lipschitz function $h_\lambda:\Om\ra \R{N}$
such that $\abs{\left\{x\in\Om\,:\,\xii\circ f(x)\neq
h_\lambda(x)\right\}}\leq \left(C/\lambda^p\right) \|\xii\circ f\|_{W^{1,p}}^p$. 
Clearly, the function
$f_\lambda=\xii^{-1}\circ\ro\circ h_\lambda$ has the desired
property.
\end{proof}

A direct corollary of the Lipschitz approximation and of
Theorem \ref{t:rademacher}
is that any Sobolev $Q$-valued map is approximately differentiable almost everywhere.

\begin{definition}[Approximate Differentiability]\label{d:appdiff}
A $Q$-valued function $f$ is approximately differentiable in
$x_0$ if there exists a measurable subset $\tilde{\Omega}
\subset \Omega$ containing $x_0$
such that $\tilde{\Omega}$ has
density $1$ at $x_0$ and $f|_{\tilde{\Omega}}$ is differentiable
at $x_0$. 
\end{definition}

\begin{corol}\label{c:appdiffae}
Any $f\in W^{1,p} (\Omega ,\Iq)$ is approximately differentiable a.e.
\end{corol}

The approximate differential of $f$ at $x_0$ can
then be defined as $D(f|_{\tilde{\Omega}})$ because it is
independent of the set $\tilde{\Omega}$. With 
a slight abuse of notation, we will denote it by $Df$,
as the classical differential. Similarly,
we can define the approximate directional derivatives.
Moreover, for these quantities we use the notation
of Section \ref{s:rademacher}, that is
\begin{equation*}
Df = \sum_i \a{Df_i}
\quad\text{and}\quad \partial_\nu f = \sum_i \a{\partial_\nu f_i},
\end{equation*}
with the same convention as in Remark \ref{r:convention}, 
i.e. the first-order
approximation is given by 
$T_{x_0}f= \sum_i\a{f_i(x_0)+Df_i(x_0)\cdot(x-x_0)}$.

\begin{proof}[Proof of Corollary \ref{c:appdiffae}]
For every $k\in\N$, choose a Lipschitz function $f_k$ such that
$\Omega\setminus \Omega_k :=\{f\neq f_k\}$ has measure
smaller than $k^{-p}$. By Rademacher's Theorem
\ref{t:rademacher}, $f_{k}$ is differentiable a.e. on
$\Om$. Thus, $f$ is approximately differentiable
at a.e. point of $\Omega_{k}$. Since 
$|\Omega\setminus \cup_k \Omega_{k}|=0$, this completes the proof.
\end{proof}

Finally, observe that the chain-rule
formulas of Proposition \ref{p:chain} have an obvious
extension to approximate differentiable functions.

\begin{propos}\label{p:chainapp}
Let $f:\Omega\to \Iqs$ be approximate
differentiable at $x_0$. If $\Psi$ and $F$ are as
in Proposition \ref{p:chain}, then \eqref{e:esterna}
and \eqref{e:reale} holds. Moreover,
\eqref{e:interna} holds when $\Phi$ is a diffeomorphism.
\end{propos}
\begin{proof} The proof follows trivially from
Proposition \ref{p:chain} and Definition \ref{d:appdiff}.
\end{proof}

\subsection{Trace properties}\label{ss:trace}
Next, we show that the trace of a Sobolev $Q$-function
as defined in Definition \ref{d:Dirichlet problem}
corresponds to the classical trace for $\xii\circ f$.
First we introduce the definition of weak convergence for $Q$-valued functions.

\begin{definition}[Weak convergence]\label{d:weak convergence}
Let $f_k, f\in W^{1,p}(\Om,\Iq)$. We say 
that $f_k$ converges weakly 
to $f$ for $k \to \infty$, 
(and we write $f_k\rightharpoonup f$) in $W^{1,p}(\Om,\Iq)$,
if
\begin{itemize}
\item[(i)] $\int\cG(f_k,f)^p\ra0$, for $k\ra\infty$;
\item[(ii)] there exists a constant
$C$ such that $\int |Df_k|^p\leq C<\infty$ for every $k$.
\end{itemize}
\end{definition}

\begin{propos}[Trace of Sobolev $Q$-functions]\label{p:trace}
Let $f\in W^{1,p}(\Om,\Iq)$. Then, there is a unique function
$g\in L^p(\de \Om,\Iq)$ such that $f\vert_{\de\Om}=g$ 
in the sense of Definition 
\ref{d:Dirichlet problem}.
Moreover, $f\vert_{\de\Om}=g$ if and only if
$\xii\circ f\vert_{\de\Om}=\xii\circ g$
in the usual sense, and the set of mappings
\begin{equation}\label{e:trace g}
W^{1,2}_g(\Omega, \Iq) := 
\left\{ f\in W^{1,2} (\Omega, \Iq)\,:\,
f|_{\partial \Omega}
= g\right\}
\end{equation}
is sequentially weakly closed in $W^{1,p}$.
\end{propos}

\begin{proof}
For what concerns the existence,
let $g=\xii^{-1}(\xii\circ f\vert_{\de\Om})$.
Since $\xii\circ f\vert_{\de\Om}=\xii\circ g$,
for every Lipschitz real-valued map $\Phi$
on $\cQ$, we clearly have 
$\Phi\circ \xii\circ f\vert_{\de\Om}=\Phi\circ
\xii\circ g$.
Since
$\Upsilon_T(\cdot):= \cG (\xi^{-1}(\cdot), T)$ is
a Lipschitz map on $\cQ$ for every $T\in \Iq$, we conclude that
$f\vert_{\de\Om}=g$ in the sense
of Definition \ref{d:Dirichlet problem}.

The uniqueness is an easy consequence of the following observation:
if $h$ and $g$ are maps in $L^p (\de\Om, \Iq)$
such that $\cG (h(x), T)=\cG (g(x), T)$ for $\cH^{n-1}$-a.e. $x$
and for every $T\in \Iq$, then $h=g$.
Indeed, fixed a countable dense subset $\{T_i\}_{i\in\N}$ of
$\Iq$, we have
\[
\cG\big(h(x),g(x)\big)=\sup_i\big|\cG (h(x), T_i)-\cG (g(x), T_i)\big|=0
\quad \cH^{n-1}\textrm{-a.e.}
\]
The last statement of the proposition follows easily and the proof is left to the reader.
\end{proof}

\subsection{Sobolev and Poincar\'e inequalities}\label{ss:sob}
As usual, for $p<m$ we set
$\frac{1}{p^*}=\frac{1}{p}-\frac{1}{m}$.

\begin{propos}[Sobolev Embeddings]\label{p:Sembeddings}
The following embeddings hold:
\begin{itemize}
\item[$(i)$] if $p<m$, then $W^{1,p}(\Om,\Iq)
\subset L^{q}(\Om,\Iq)$ for
every $q\in [1,p^*]$, and the inclusion is compact
when $q<p^*$;
\item[$(ii)$] if $p=m$, then $W^{1,p}(\Om,\Iq)
\subset L^{q}(\Om,\Iq)$, for
every $q\in [1,+\infty)$, with compact inclusion; 
\item[$(iii)$] if $p>m$, then $W^{1,p}(\Om,\Iq)
\subset C^{0,\alpha}(\Om,\Iq)$, for
$\alpha=1-\frac{m}{p}$, with compact inclusion.
\end{itemize}
\end{propos}
\begin{proof}
Since $f$ is a $L^q$ (resp. H\"older) $Q$-function if and
only if $\xii\circ f$ is $L^q$ (resp. H\"older), 
the proposition follows trivially from Theorem
\ref{t:Sob=Sob} and the Sobolev embeddings for $\xii\circ f$
(see, for example, \cite{Adams} or \cite{Ziemer}).
\end{proof}

\begin{propos}[Poincar\'e inequality]\label{p:poincare'}
Let $M$ be a connected
bounded Lipschitz open set of an $m$-dimensional
Riemannian manifold and let $p<m$.
There exists a constant $C=C(p,m,n,Q,M)$ with the following
property:
for every $f\in W^{1,p}(M,\Iq)$,
there exists a point $\overline f\in \Iq$ such that
\begin{equation}\label{e:poincare'}
\left(\int_{M} \cG \big(f, \overline f\big)^{p^*}\right)^{\frac{1}{p^*}}
\leq C \left(\int_{M}|Df|^{p}\right)^{\frac{1}{p}}.
\end{equation}
\end{propos}
\begin{remark}
Note that the point $\overline f$ in the Poincar\'e inequality
is not uniquely determined. Nevertheless, 
in analogy with the classical setting,
we call it a \textit{mean} for $f$.
\end{remark}

\begin{proof}
Set $h:=\xii\circ f:M\ra\cQ\subset\R{N}$.
By Theorem \ref{t:Sob=Sob}, $h\in W^{1,p}(M,\R{N})$.
Recalling the classical Poincar\'e inequality
(see, for instance, \cite{Adams} or \cite{Ziemer}), 
there exists a constant $C=C(p,m,M)$ such that,
if $\overline h=\mint_{M} h$, then
\begin{equation}\label{e:classic poincare'}
\left(\int_{M}\left|h(x)-\overline h\right|^{p^*} dx\right)^{\frac{1}{p^*}}
\leq C \left(\int_{M}|Dh|^{p}\right)^{\frac{1}{p}}.
\end{equation}
Let now $v\in\cQ$ be such that $\left|\overline h - v\right|=\dist\left(\overline h,\cQ\right)$
($v$ exists because $\cQ$ is closed). Then, since $h$ takes values 
in $\cQ$ almost
everywhere, by \eqref{e:classic poincare'} we infer
\begin{equation}\label{e:media}
\left(\int_{M}\left|\overline h -v\right|^{p^*} dx\right)^{\frac{1}{p^*}}
\leq \left(\int_{M}\left|\overline h -h(x)\right|^{p^*} dx\right)^{\frac{1}{p^*}}
\leq C\left(\int_{M}|Dh|^{p}\right)^{\frac{1}{p}}.
\end{equation}
Therefore, using \eqref{e:classic poincare'} and \eqref{e:media},
we end up with
\begin{equation*}
\norm{h-v}{L^{p^*}}\leq \norm{h-\overline h}{L^{p^*}}+\norm{\overline h-v}{L^{p^*}}
\leq 2\,C \norm{Dh}{L^p}.
\end{equation*}
Hence, it is immediate to verify, using the biLipschitz continuity of $\xii$,
that \eqref{e:poincare'} is satisfied with $\overline f=\xii^{-1}(v)$ and a constant $C(p,m,n,Q,M)$.
\end{proof}

\subsection{Campanato--Morrey estimates}\label{ss:camp}
We prove next the Campanato--Morrey estimates
for $Q$-functions, a crucial tool in the
proof of Theorem \ref{t:hoelder}.

\begin{propos}\label{p:Campanato Morrey}
Let $f\in W^{1,2}(B_1,\Iq)$ and $\alpha\in(0,1]$ be such that
\begin{equation*}
\int_{B_r (y)}|Df|^2
\leq A\; r^{m-2+2\alpha}\quad
\mbox{for every $y\in B_1$ and a.e. $r\in ]0,1-|y|[$.}
\end{equation*}
Then, for every $0<\delta<1$,
there is a constant $C=C(m,n,Q,\delta)$ with
\begin{equation}\label{e:Campanato Morrey}
\sup_{x,y\in
\overline{B_{\delta}}}\frac{\cG(f(x),f(y))}{\abs{x-y}^\alpha}
=: \left[f\right]_{C^{0,\alpha}(\overline{B_{\delta}})}
\leq C\,\sqrt{A}.
\end{equation}
\end{propos}

\begin{proof}
Consider $\xii\circ f$: as shown in Theorem \ref{t:Sob=Sob},
there exists a constant $C$ depending on $\Lip(\xii)$ and $\Lip(\xii^{-1})$ such that 
\begin{equation*}
\int_{B_r (y)}|D(\xii\circ f)(x)|^2 dx\leq C\, A\,
r^{m-2+2\alpha}
\end{equation*}
Hence, the usual Campanato--Morrey estimates (see, for example, $3.2$ in
\cite{HanLin}) provide
the existence of
a constant $C=C(m,\alpha,\delta)$ such that
\begin{equation*}
\abs{\xii\circ f(x)-\xii\circ f(y)}\leq C\,\sqrt{A}\,
\abs{x-y}^\alpha\quad\mbox{for every $x,y\in
\overline{B_{\delta}}$.}
\end{equation*}
Thus, composing with $\xii^{-1}$, we
conclude the desired estimate \eqref{e:Campanato Morrey}.
\end{proof}

\subsection{A technical Lemma}\label{ss:tech}
This last subsection contains a technical 
lemma which estimates the Dirichlet energy of an interpolation between 
two functions defined on concentric spheres.
The lemma is particularly useful to construct competitors
for $\D$-minimizing maps.

\begin{lemma}[Interpolation Lemma]\label{l:technical}
There is a constant $C=C(m,n,Q)$ with the following property.
Let $r>0$, $g\in W^{1,2}(\de B_r,\Iq)$ and 
$f\in W^{1,2}(\partial B_{r(1-\eps)},\Iq)$.
Then, there exists $h\in W^{1,2}(B_r\setminus B_{r(1-\eps)},\Iq)$ 
such that $h|_{\de B_r} = g$, $h|_{\de B_{r(1-\eps)}}=f$ and 
\begin{multline}
\D(h,B_r\setminus B_{r(1-\eps)})\leq
 C\,\eps\,r\, 
\big[\D(g,\de B_r)+\D(f,\de B_{r(1-\eps)})\big]+\\
+\frac{C}{\eps\,r}\int_{\de B_r}
\cG\left(g(x),f\left((1-\eps)\,x\right)\right)^2 dx.
\label{e:joining}
\end{multline}
\end{lemma}

\begin{proof}
By a scaling argument, it is enough to prove the lemma for $r=1$.
As usual, we consider $\psi=\xii\circ g$ and
$\varphi=\xii\circ f$. For $x\in \de B_1$ and
$t\in [1-\eps,1]$, we define
\begin{equation*}
\Phi (t\, x)=\frac{(t-1+\eps)\,\psi (x)+
(1-t)\,\varphi\left((1-\eps)\,x\right)}{\eps},
\end{equation*}
and $\overline\Phi=\ro\circ \Phi$.
It is straightforward to verify that $\overline\Phi$ 
belongs to $W^{1,2}(B_1\setminus B_{1-\eps},\cQ)$.
Moreover, the Lipschitz continuity of $\ro$ and an easy 
computation yield the following
estimate,
\begin{align*}
\int_{B_1\setminus B_{1-\eps}}\left|D\,\overline\Phi\right|^2
\leq{}& C\,\int_{B_1\setminus B_{1-\eps}}\left|D\,\Phi\right|^2 
\nonumber\\
\leq {}&  C\,\int_{1-\eps}^1\int_{\de B_1}
\left(\left|\de_\tau \varphi(x)\right|^2
+\left|\de_\tau\psi (x)\right|^2
+\left|\frac{\psi (x)-
\varphi \left((1-\eps)x\right)}{\eps}\right|^2\right) dx \,dt\\
={}&  C\,\eps\left\{\D(\psi,\de B_1)+\D(\varphi,\de B_{1-\eps})\right\}+\\
&+C\,\eps^{-1}\int_{\de B_1}\left|\psi (x)-
\varphi\left((1-\eps)x\right)\right|^2 dx,
\end{align*}
where $\de_\tau$ denotes the tangential derivative.
Consider, finally, $h=\xii^{-1}\circ \overline\Phi$:
\eqref{e:joining} follows easily from the
biLipschitz continuity of $\xii$.
\end{proof}

The following is a straightforward corollary.

\begin{corol}\label{c:fill}
There exists a constant $C=C(m,n,Q)$ with the following property.
For every $g\in W^{1,2}(\de B_1,\Iq)$, there is
$h\in W^{1,2} (B_1,\Iq)$ with $h|_{\de B_1} = g$ and
\begin{equation*}
\D (h, B_1) \leq C\, \D (g, \de B_1) +C\int_{\de B_1} \cG (g, Q\a{0})^2.
\end{equation*}
\end{corol}

\section{Existence of $\D$-minimizing $Q$-valued functions}
\label{s:existence}

In this section we prove Theorem \ref{t:existence}.
We first remark that Almgren's definition of Dirichlet
energy differs from ours. More precisely,
using our notations, Almgren's definition 
of the Dirichlet energy is simply
\begin{equation}\label{e:aldir}
\int_{\Omega} \sum_{\stackrel{i=1,\ldots, Q}{j=1,\ldots, m}}|\partial_j f_i (x)|^2\, dx,
\end{equation}
where $\partial_j f_i$ are the approximate partial derivatives
of Definition \ref{d:appdiff}, which exist almost everywhere thanks
to Corollary \ref{c:appdiffae}.
Moreover, \eqref{e:aldir} makes sense because the integrand
does not depend upon the particular selection chosen
for $f$. Before proving Theorem \ref{t:existence} we will
show that our Dirichlet energy coincides with Almgren's.

\begin{propos}[Equivalence of the definitions] 
\label{p:equivalence def Dir}
For every $f\in W^{1,2}(\Om,\Iq)$ and every $j=1,\ldots,m$, we have
\begin{equation}\label{e:goal}
|\de_jf|^2 = \sum_i |\partial_j f_i|^2
\quad\mbox{a.e.}
\end{equation}
Therefore the Dirichlet energy
$\D(f,\Om)$ of
Definition \ref{d:dirichlet} coincides with \eqref{e:aldir}.
\end{propos}

\begin{remark}\label{r:independent}
Fix a point $x_0$ of approximate differentiability for $f$
and consider $T_{x_0} (x) = \sum \a{f_i (x_0)+Df_i(x_0)\cdot (x-x_0) }$
its first order approximation at $x_0$. Note that
the integrand in \eqref{e:aldir} coincides with
$\sum_i |Df_i(x_0)|^2$ (where $|L|$ denotes the Hilbert-Schmidt 
norm of the matrix $L$) 
and it is independent of the orthonormal coordinate system
chosen for $\R{m}$. Thus, Proposition \ref{p:equivalence def Dir}
(and its obvious counterpart when the domain is a Riemannian manifold)
implies that $\D (f, \Om)$ is as well independent of this choice.
\end{remark}

\begin{remark}
In the sequel, we will often use the following notation: given a $Q$-point $T\in\Iqs$,
$T=\sum_i\a{P_i}$, we set
\begin{equation*}
|T|^2:=\cG(T,Q\a{0})^2= \sum_i|P_i|^2.
\end{equation*}
In the same fashion, for $f:\Om\ra\Iq$, we define the function $|f|:\Om\ra\R{}$
by setting $|f|(x)=|f(x)|$. Then,
Proposition \ref{p:equivalence def Dir} asserts that, since we understand $Df$ and $\de_j f$ 
as maps into, respectively, $\Iq (\R{n\times m})$ and
$\Iq (\R{n})$, this notation is consistent with the
definitions of $|Df|$ and $|\de_j f|$ given in 
\eqref{e:def|Df|} and \eqref{e:def|D_jf|}.
\end{remark}

The Dirichlet energy of a function $f\in W^{1,2}$ can be
recovered, moreover, as the energy of the composition $\xii_{BW}\circ f$,
where $\xii_{BW}$ is the biLipschitz embedding in Corollary
\ref{c:BW} (compare with Theorem \ref{t:Sob=Sob}).

\begin{propos}\label{p:dir BW}
For every $f\in W^{1,2}(\Om,\Iq)$, it holds
$|Df| = |D(\xii_{BW}\circ f)|$ a.e.
In particular, $\D(f,\Om)=\int_{\Omega}|D(\xii_{BW}\circ f)|^2$.
\end{propos}

Although this proposition gives a great intuition about the energy of
$Q$-valued functions, as already pointed out, we will not use it
in the rest of the paper, the reason being that, the theory is in fact
independent of the biLipschitz embedding.

\subsection{Proof the equivalence of the definitions}
\begin{proof}[Proof of Proposition \ref{p:equivalence def Dir}.]
We recall the definition of $|\de_jf|$ and $|Df|$
given in \eqref{e:def|D_jf|} and \eqref{e:def|Df|}: chosen a countable
dense set $\{T_l\}_{l\in\N}\subset \Iq$, we define
\begin{equation*}
\abs{\de_jf}=\sup_{l\in\N}\abs{\de_j\cG(f,T_l)}
\quad\mbox{and}\quad |Df|^2 := 
\sum_{j=1}^m\abs{\de_jf}^2.
\end{equation*}
By Proposition \ref{p:lipapprox}, we can consider
a sequence $g^k= \sum_{i=1}^Q \a{g^k_i}$ of 
Lipschitz functions with the
property that $|\{g^k\neq f\}|\leq 1/k$. Note that
$|\partial_j f| = |\partial_j g^k|$ and $\sum_i |\partial_j g^k_i|^2
= \sum_i |\partial_j f_i|^2$ almost everywhere on
$\{g^k=f\}$. Thus, it suffices to prove the
proposition for each Lipschitz function $g^k$.

Therefore, we assume from now on that $f$ is Lipschitz. 
Note next that on the set $E_l= \{x\in\Om:f(x)=T_l\}$ both
$|\de_j f|$ and $\sum_i |\partial_j f_i|^2$ vanish
a.e. Hence, it suffices to show \eqref{e:goal}
on any point $x_0$ where $f$ and all
$\cG (f, T_l)$ are differentiable
and $f(x_0)\not\in \{T_l\}_{l\in\N}$.

Fix such a point, which, without loss of generality, 
we can assume to be the origin, $x_0=0$.
Let $T_0 f$ be the first oder approximation of $f$ at $0$.
Since $\cG (\cdot, T_l)$ is a Lipschitz function, we have
$\cG (f (y), T_l) = \cG (T_0 f (y), T_l) + o (|y|)$. Therefore,
$g (y) := \cG (T_0 f (y), T_l)$ is differentiable
at $0$ and $\partial_j g (0) = \partial_j \cG (f, T_l) (0)$. 

We assume, without loss of generality, that $\cG(f(0),T_l)^2=\sum_i|f_i (0)-P_i|^2$,
where $T_l=\sum_i\a{P_i}$. Next, we consider the function
\begin{equation*}
h(y):= \sqrt{\sum_i|f_i (0) +Df_i (0) \cdot y-P_i|^2}.
\end{equation*}
Then, $g\leq h$. Since $h(0)=g(0)$, we conclude that
$h-g$ has a minimum at $0$. Recall that both $h$ and $g$ 
are differentiable at $0$ and $h(0)=g(0)$. Thus, we conclude
$\nabla h (0)=\nabla g (0)$, which in turn
yields the identity
\begin{equation}\label{e:computation3}
\partial_j\, \cG (f, T_l) (0)=\de_jg(0)
=\de_jh(0)
= \sum_i\frac{(f_i (0)-P_i)\cdot\de_jf_i(0)}
{\sqrt{\sum_i |f_i (0) - P_i|^2}}.
\end{equation}
Using the Cauchy-Schwartz inequality and 
\eqref{e:computation3}, we deduce that
\begin{equation}\label{e:quasi2}
|\de_jf|(0)^2=\sup_{l\in\N}\abs{\de_j\cG(f,T_l)(0)}^2
\leq\sum_i\abs{\de_jf_i(0)}^2.
\end{equation}

If the right hand side of \eqref{e:quasi2}
vanishes, then we clearly have equality.
Otherwise, let $Q_i = f_i (0) + \lambda\,\de_jf_i(0)$, where $\lambda$ is a small constant
to be chosen later, 
and consider $T=\sum_i \a{Q_i}$.
Since $\{T_l\}$ is a dense subset of $\Iq$, for every $\eps>0$
we can find a point $T_l=\sum_i\a{P_i}$ such that
\begin{equation*}
P_i=f_i(0)+\lambda\,\de_j f_i(0)+\lambda\,R_i,\quad
\text{with } |R_i|\leq\eps \;\text{for every }i.
\end{equation*}
Now we choose $\lambda$ and $\eps$ small enough to ensure that $\cG(f(0),T_l)^2=\sum_i|f_i (0)-P_i|^2$ (indeed, recall that, if
$f_i(0)=f_k(0)$, then $\de_jf_i(0)=\de_jf_k(0)$).
So, we can repeat the computation
above and deduce that
\begin{equation*}
\partial_j\, \cG (f, T_l) (0)= \sum_i\frac{(f_i (0)-P_i)\cdot\de_jf_i(0)}
{\sqrt{\sum_i |f_i (0) - P_i|^2}}
=\sum_i\frac{(\de_j f_i(0)+R_i)\cdot\de_jf_i(0)}
{\sqrt{\sum_i |\de_j f_i(0)+R_i|^2}}.
\end{equation*}
Hence,
\begin{equation*}
|\de_jf|(0) \geq \sum_i\frac{(\de_j f_i(0))^2 + 
\eps|\partial_j f_i (0)|}
{\sqrt{\sum_i (|\de_j f_i(0)|+\eps)^2}}.
\end{equation*}
Letting $\eps\to0$, we obtain the inequality
$|\de_j f| (0) \geq \sum_j (\de_j f_i (0))^2$.
\end{proof}

\begin{proof}[Proof of Proposition \ref{p:dir BW}]
As for Proposition \ref{p:equivalence def Dir}, it is enough
to show the proposition for a Lipschitz function $f$.
We prove that the functions
$|Df|$ and $|D(\xii_{BW}\circ f)|$
coincide on each point of differentiability of $f$.

Let $x_0$ be such a point and let
$T_{x_0}f(x)=\sum_i\a{f_i(x_0)+Df_i(x_0)\cdot(x-x_0)}$ be the first order
expansion of $f$ in $x_0$.
Since $\cG(f(x),T_{x_0}f(x))=o(|x-x_0|)$ and $\Lip(\xii_{BW})=1$,
it is enough to prove that
$|Df|(x_0)=|D(\xii_{BW}\circ T_{x_0}f)(x_0)|$.

Using the fact that $Df_i(x_0)=Df_j(x_0)$ when $f_i(x_0)=f_j(x_0)$,
it follows easily that, for every $x$ with $|x-x_0|$ small enough,
\[
\cG(T_{x_0}f(x),f(x_0))^2=\sum_i|Df_i(x_0)\cdot(x-x_0)|^2.
\]
Hence, since $\xii_{BW}$ is an isometry in a neighborhood of each point,
for $|x-x_0|$ small enough, we infer that
\begin{equation}\label{e:diff quad}
|\xii_{BW}(T_{x_0}f(x))-\xii_{BW}(f(x_0))|^2=\sum_i|Df_i(x_0)\cdot(x-x_0)|^2.
\end{equation}
For $x=t\,\e_j+x_0$ in \eqref{e:diff quad}, where the $e_j$'s
are the canonical basis in $\R{m}$, taking the limit as $t$ goes to zero,
we obtain that
\[
|\de_j(\xii_{BW}\circ T_{x_0}f)(x_0)|^2=\sum_i|\de_jf_i|^2(x_0).
\]
Summing in $j$ and using Proposition \ref{p:equivalence def Dir}, we
conclude that $|Df|(x_0)=|D(\xii_{BW}\circ T_{x_0}f)(x_0)|$,
which concludes the proof.
\end{proof}

\subsection{Proof of Theorem \ref{t:existence}}
Let $g\in W^{1,2}(\Om,\Iq)$ be given.
Thanks to Propositions \ref{p:trace} and
\ref{p:Sembeddings}, 
it suffices to verify
the sequential weak lower semicontinuity
of the Dirichlet energy.
To this aim, let $f_k\rightharpoonup f$ in $W^{1,2}(\Om,\Iq)$:
we want to show that
\begin{equation}\label{e:semi}
\D (f, \Omega) \leq \liminf_{k\to \infty}
 \D (f_k, \Omega).
\end{equation}
Let $\{T_l\}_{l\in\N}$ be a dense subset of
$\Iq$ and recall that $|\partial_j f|^2 = \sup_l \bigl(\partial_j
\cG (f, T_l)\bigr)^2$.
Thus, if we set
\[
h_{j,N}= \max_{l\in\{1, \ldots, N\}}
\bigl(\partial_j \cG (f, T_l)\bigr)^2,
\]
we conclude that $h_{j,N}\uparrow |\partial_j f|^2$.
Next, for every $N$, denote by $\mathcal{P}_N$
the collections $P=\{E_l\}_{l=1}^N$ of $N$ disjoint measurable subsets 
of $\Omega$. Clearly, it holds
\[
h_{j,N} = \sup_{P\in\mathcal{P}} \sum_{E_l\in P} \bigl(\partial_j 
\cG (f, T_l)\bigr)^2\, {\bf 1}_{E_l}.
\]
By the Monotone Convergence Theorem, 
we conclude
\begin{equation*}
\D (f, \Omega) =
\sum_{j=1}^m \sup_N \int h_{j,N}^2
=\sum_{j=1}^m \sup_N \sup_{P\in \mathcal{P}_N}
\sum_{E_l\in P}\int_{E_l}
\bigl(\partial_j \cG (f, T_l)\bigr)^2.
\end{equation*}
Fix now a partition $\{F_1, \ldots, F_N\}$ such that,
for a given $\eps>0$, 
$$
\sum_l \int_{F_l} \bigl(\partial_j \cG (f, T_l)\bigr)^2
\geq \sup_{P\in \mathcal{P}_N}
\sum_{E_l\in P}\int_{E_l}
\bigl(\partial_j \cG (f, T_l)\bigr)^2 - \eps .
$$
Then, we can find compact sets $\{K_1, \ldots, K_N\}$
with $K_l\subset F_l$ and 
$$
\sum_l \int_{K_l} \bigl(\partial_j \cG (f, T_l)\bigr)^2
\geq \sup_{P\in \mathcal{P}_N}
\sum_{E_l\in P}\int_{E_l}
\bigl(\partial_j \cG (f, T_l)\bigr)^2 - 2 \eps .
$$
Since the $K_l$'s are disjoint compact sets, we can find
disjoint open sets $U_l\supset K_l$. So, denote by
$\mathcal{O}_N$ the collections of $N$ pairwise disjoint
open sets of $\Omega$. We conclude 
\begin{equation}\label{e:rappresenta1000}
\D (f, \Omega) =
\sum_{j=1}^m \sup_N \int h_{j,N}^2
=\sum_{j=1}^m \sup_N \sup_{P\in \mathcal{O}_N}
\sum_{U_l\in P}\int_{U_l}
\bigl(\partial_j \cG (f, T_l)\bigr)^2.
\end{equation}
Note that, since $\cG(f_k,T_l)\to \cG(f,T_l)$
strongly in $L^2 (\Omega)$, then
$\partial_j\cG(f_k,T_l)
\rightharpoonup \partial_j \cG(f,T_l)$ in $L^2 (U)$ for every open
$U\subset \Omega$. Hence,
for every $N$ and every $P\in \mathcal{O}_N$,
we have
\begin{equation*}
\sum_{U_l\in P}\int_{U_l} \bigl(\partial_j
\cG (f, T_l)\bigr)^2
\leq \liminf_{k\to+\infty}
\sum_{U_l\in P} \int_{U_l} \bigl(\partial_j
\cG (f_k, T_l)\bigr)^2
\leq \liminf_{k
\to\infty} \int_\Om |\partial_j f_k|^2.
\end{equation*}
Taking the supremum in $\mathcal{O}_N$ and in $N$, and then summing in
$j$, in view of  \eqref{e:rappresenta1000}, we achieve \eqref{e:semi}.
\chapter{Regularity theory}

This chapter is devoted to the proofs of the two Regularity Theorems
\ref{t:hoelder} and \ref{t:structure}.
In Section \ref{s:first} we derive some Euler-Lagrange conditions for $\D$-minimizers, whereas in
Section \ref{s:maximum principle} we prove a maximum principle
for $Q$-valued functions.
Using these two results, we prove
Theorem \ref{t:hoelder} in Section \ref{s:hoelder}.
Then, in Section \ref{s:frequency} we introduce Almgren's frequency
function and prove his fundamental estimate. 
The frequency function is the main tool
for the blow-up analysis of Section \ref{s:blowup},
which gives useful information on the rescalings of
$\D$-minimizing $Q$-functions.
Finally, in Section \ref{s:structure} we combine 
this analysis with a version of Federer's reduction argument
to prove Theorem \ref{t:structure}.

\section{First variations}\label{s:first}
There are two natural types of variations that can
be used to perturb $\D$-minimizing $Q$-valued
functions. The first ones, which we
call inner variations, are generated by right compositions
with diffeomorphisms of the domain.
The second, which we call
outer variations, correspond to ``left compositions''
as defined in Subsection \ref{ss:chain}.
More precisely, let $f$ be a $\D$-minimizing
$Q$-valued map. 
\begin{itemize}
\item[(IV)] Given $\varphi\in C^\infty_c (\Omega,
\R{m})$, for $\eps$ sufficiently small,
$x\mapsto \Phi_\eps (x)=x +\eps \varphi (x)$ is
a diffeomorphism of $\Omega$ which leaves $\partial
\Om$ fixed. Therefore,
\begin{equation}\label{e:IV}
0= \left.\frac{d}{d\eps}\right|_{\eps=0}
\int_\Omega |D (f \circ \Phi_\eps)|^2.
\end{equation}
\item[(OV)] Given $\psi\in C^\infty (\Omega\times
\R{n}, \R{n})$ such that $\supp(\psi)\subset \Om'\times \R{n}$ for
some $\Om'\subset\subset\Om$, we set $\Psi_\eps (x)
=\sum_i \a{f_i (x) + \eps \psi (x, f_i (x))}$ and
derive
\begin{equation}\label{e:OV}
0= \left.\frac{d}{d\eps}\right|_{\eps=0}
\int_\Omega |D\, \Psi_\eps|^2.
\end{equation}
\end{itemize}
The identities \eqref{e:IV} and \eqref{e:OV}
lead to the following proposition.

\begin{propos}[First variations]\label{p:first}
For every $\varphi\in C^\infty_c (\Omega, \R{m})$, we have
\begin{equation}\label{e:IVbis}
2\int \sum_i \big\langle Df_i  : 
Df_i \cdot D\varphi \big\rangle\,
- \int |Df|^2 \,{\rm div}\, \varphi = 0.
\end{equation}
For every $\psi\in C^\infty (\Omega_x\times \R{n}_u,
\R{n})$ such that
$$
\supp(\psi)\subset \Om'\times \R{n}\quad
\mbox{for some $\Om'\subset\subset\Om$,}
$$
and
\begin{equation}\label{e:hp}
|D_u\psi|\leq C<\infty\quad\textrm{and}\quad
|\psi|+|D_x\psi|\leq C\left(1+|u|\right)\,,
\end{equation}
we have
\begin{equation}\label{e:OVbis}
\int \sum_i \big\langle D f_i (x) : 
D_x \psi (x, f_i (x))\big\rangle\, dx
+ \int \sum_i \big\langle Df_i (x) : 
D_u \psi (x, f_i (x))\cdot
Df_i (x)\rangle\, d x= 0. 
\end{equation}
\end{propos}

Testing \eqref{e:IVbis} and \eqref{e:OVbis} with
suitable $\varphi$ and $\psi$, we get
two key identities.
In what follows, $\nu$ will always denote the outer unit normal
on the boundary $\partial B$ of a given ball. 

\begin{propos}\label{p:import}
Let $x\in \Omega$.
Then, for a.e. $0< r<
\dist(x, \partial\Om)$, we have
\begin{equation}\label{e:cono}
(m-2) \int_{B_r (x)} |Df|^2 =
r \int_{\partial B_r (x)} |Df|^2 - 2\,
r \int_{\partial B_r (x)} \sum_i |\partial_\nu f_i|^2,
\end{equation}
\begin{equation}\label{e:perparti}
\int_{B_r (x)} |Df|^2 = \int_{\partial B_r (x)}
\sum_i \langle \partial_\nu f_i, f_i\rangle.
\end{equation}
\end{propos}

\begin{remark}
The identities \eqref{e:cono} and \eqref{e:perparti}
are classical facts for $\R{n}$-valued harmonic maps $f$, which
can be derived from the Laplace equation $\Delta f = 0$. 
\end{remark}

\subsection{Proof of Proposition \ref{p:first}}
We apply formula \eqref{e:interna} of Proposition \ref{p:chainapp}
to compute
\begin{equation}\label{e:chain10}
D (f \circ \Phi_\eps) (x)=
\sum_i \a{D f_i (x+\eps \varphi (x))
+ \eps [Df_i (x+\eps \varphi (x))]\cdot D \varphi (x)}.
\end{equation}
For $\eps$ sufficiently small, $\Phi_\eps$
is a diffeomorphism. We denote by $\Phi_\eps^{-1}$ its inverse. Then,
inserting \eqref{e:chain10} in \eqref{e:IVbis}, changing variables
in the integral ($x=\Phi_\eps^{-1} (y)$) 
and differentiating in $\eps$, we get
\begin{align*}
0&= \left.\frac{d}{d\eps}\right|_{\eps=0}
\int_\Omega \sum_i |Df_i (y)+
\eps Df_i\cdot D\varphi (\Phi_\eps^{-1} (y))|^2
\,\det\, (D\Phi_\eps^{-1} (y))\, dy\nonumber\\
&= 2\int \sum_i \big\langle Df_i (y) : 
Df_i (y)\cdot D\varphi (y)\big\rangle\, d y
- \int \sum_i |Df_i (y)|^2 
{\rm div}\, \varphi (y)\, d y.
\end{align*}
This shows \eqref{e:IVbis}. As for \eqref{e:OVbis},
using \eqref{e:esterna} and then differentiating in
$\eps$, the proof is straightforward
(the hypotheses in \eqref{e:hp}
ensure the summability of the various 
integrands involved in the computation).

\subsection{Proof of Proposition \ref{p:import}}
Without loss of generality,
we assume $x=0$.
We test \eqref{e:IVbis} with a function $\varphi$
of the form $\varphi (x) = \phi (|x|)\,x$, where $\phi$
is a function in $C^\infty([0,\infty))$, with $\phi\equiv 0$ on 
$[r, \infty)$,
$r< \dist (0, \partial \Omega)$, and $\phi\equiv 1$ in a neighborhood
of $0$. Then,
\begin{equation}\label{e:radial}
D\varphi (x) = \phi (|x|) \,\Id + \phi' (|x|)\, x\otimes
\frac{x}{|x|}\, \quad\textrm{and}\quad {\rm div}\, \varphi (x) =
m\, \phi (|x|) + |x|\, \phi' (|x|),
\end{equation}
where $\Id$ denotes the $m\times m$ identity
matrix. Note that
\begin{equation*}
\partial_\nu f_i (x)= D f_i(x) \cdot \frac{x}{|x|}.
\end{equation*}
Then, inserting \eqref{e:radial} into
\eqref{e:IVbis}, we get
\begin{align*}
0 ={}& 2 \int |Df (x)|^2\, \phi (|x|)\, dx +
2 \int \sum_{i=1}^Q |\partial_\nu f_i (x)|^2\, \phi' (|x|) \,|x|\, dx\nonumber\\
&- m \int |Df (x)|^2 \,\phi (|x|)\, dx - \int |Df (x)|^2\, \phi' (|x|) \,|x|\,
dx. 
\end{align*}
By a standard approximation procedure, it is
easy to see that we can test with 
\begin{equation}\label{e:test}
\phi (t)= \phi_n (t):=
\left\{\begin{array}{ll}
1 & \mbox{for $t\leq r-1/n$,}\\
n\, (r-t) &\mbox{for $r-1/n\leq t\leq r$.}
\end{array}\right.
\end{equation}
With this choice we get
\begin{align*}
0 ={}& (2-m) \int |Df (x)|^2\,\phi_n (|x|)\, dx  -
\frac{2}{n} \int_{B_r\setminus B_{r-1/n}} 
\sum_{i=1}^Q |\partial_\nu f_i (x)|^2 \,|x|\,dx\notag\\
&+ \frac{1}{n} \int_{B_r \setminus B_{r-1/n}} 
|Df (x)|^2\,|x|\,dx.
\end{align*}
Let $n\uparrow\infty$. Then, the first integral converges 
towards $(2-m)\int_{B_r} |Df|^2$. As for the
second and third integral, for a.e. $r$, they converge,
respectively, to
$$
- r \int_{\partial B_r} \sum_{i=1}^Q |\partial_\nu f_i|^2
\quad\mbox{and}\quad
r \int_{\partial B_r} |Df|^2.
$$
Thus, we conclude \eqref{e:cono}.

\medskip

Similarly, test \eqref{e:OVbis} with
$\psi (x, u) = \phi (|x|)\, u$. Then,
\begin{equation}\label{e:radial2}
D_u \psi (x, u) = \phi (|x|) \,\Id
\quad\mbox{and}\quad 
D_x \psi (x,u) = \phi' (|x|) \,u\otimes \frac{x}{|x|}\,.
\end{equation}
Inserting \eqref{e:radial2} into \eqref{e:OVbis}
and differentiating in $\eps$, we get
\begin{equation*}
0= \int |Df (x)|^2\,\phi (|x|)\, dx
+\int \sum_{i=1}^Q \langle f_i (x), \partial_\nu f_i (x)\rangle\,
\phi' (|x|)\, dx. 
\end{equation*}
Therefore, choosing $\phi$ as in \eqref{e:test}, we can argue as
above and, for $n\uparrow\infty$, we conclude \eqref{e:perparti}.

\section{A maximum principle for $Q$-valued functions}
\label{s:maximum principle}

The two propositions of this section play a key role
in the proof of the H\"older regularity for $\D$-minimizing
$Q$-functions when the domain has dimension strictly larger
than two. Before stating them, we introduce 
two important functions on $\Iqs$. 

\begin{definition}[Diameter and separation]\label{d:diamsep}
Let $T = \sum_i \a{P_i} \in \Iq$. The {\em diameter}
and the {\em separation} of $T$ are defined, respectively,
as
\begin{equation*}
d(T):= \max_{i,j} |P_i-P_j| 
\quad \mbox{and}\quad
s (T):= \min \big\{ |P_i - P_j| : P_i\neq P_j\big\},
\end{equation*}
with the convention that $s(T)=+\infty$ if $T = Q\a{P}$.
\end{definition}

The following proposition is an elementary extension
of the usual maximum principle for harmonic functions.

\begin{propos}[Maximum Principle]\label{p:maximum}
Let $f: \Omega \to \Iq$ be $\D$-minimizing, $T\in \Iq$ and
$r< s(T)/4$. Then,
$\cG (f(x), T)\leq r$ for $\mathcal{H}^{m-1}$-a.e. $x\in\partial \Omega$ implies that $\cG (f, T)\leq r$ almost everywhere on $\Omega$.
\end{propos}

The next proposition allows to decompose $\D$-minimizing
functions and, hence, to argue inductively on the number
of values.
Its proof is based on Proposition
\ref{p:maximum} and a simple combinatorial lemma.

\begin{propos}[Decomposition for $\D$-minimizers]\label{p:split}
There exists a positive constant $\alpha (Q)>0$ with
the following property.
If $f: \Omega\to \Iq$ is $\D$-minimizing and there exists $T\in\Iq$
such that
$\cG (f(x), T)\leq \alpha (Q)\, d(T)$  for $\mathcal{H}^{m-1}$-a.e. $x\in\partial\Omega$,
then there exists a decomposition of $f=\a{g}+ \a{h}$
into two simpler $\D$-minimizing
functions.
\end{propos}

\subsection{Proof of Proposition \ref{p:maximum}}
The proposition follows from the next lemma.

\begin{lemma}\label{l:retraction}
Let $T$ and $r$ be as in Proposition
\ref{p:maximum}. Then,
there exists a retraction 
$\vartheta:\Iq\ra \overline{B_{r}(T)}$ such that
\begin{itemize}
\item[$(i)$] $\cG(\vartheta (S_1), \vartheta (S_2))<\cG(S_1,S_2)$ 
if $S_1\notin \overline{B_r (T)}$,
\item[$(ii)$] $\vartheta (S)= S$ for every $S\in \overline{B_r (T)}$.
\end{itemize}
\end{lemma}

We assume the lemma for the moment and argue by contradiction
for Proposition \ref{p:maximum}. We assume, therefore, the
existence of a $\D$-minimizing $f$ with the following
properties:
\begin{itemize}
\item[$(a)$] $f(x) \in \overline{B_r (T)}$ for a.e. $x\in \partial \Omega$;
\item[$(b)$] $f (x)\not\in \overline{B_r (T)}$ for every $x\in E\subset
\Omega$, where $E$ is a set of positive measure.
\end{itemize}
Therefore, there exist $\eps>0$ and a set $E'$ with positive measure
such that $f (x)\not \in B_{r+\eps} (T)$ for every $x\in E'$.
By $(ii)$ of Lemma \ref{l:retraction} and $(a)$, $\vartheta\circ f$ has the same
trace as $f$. Moreover, by $(i)$ of Lemma \ref{l:retraction}, $|D (\vartheta\circ f)|\leq
|Df|$ a.e. and, by $(i)$ and $(b)$, $|D (\vartheta\circ f)|
< |Df|$ a.e. on $E'$. This implies
$\D (\vartheta\circ f, \Omega)<\D (f, \Omega)$,
contradicting the minimizing property of $f$.

\begin{proof}[Proof of Lemma \ref{l:retraction}]
First of all, we write 
\[
T=\sum_{j=1}^J k_j \a{Q_j},
\]
where $|Q_j-Q_i|> 4\,r$ for every $i\neq j$.

If $\cG (S,T)< 2r$, then 
$S=\sum_{j=1}^J\a{S_j}$
with $S_j\in B_{2r} (k_j\a{Q_j})\subset \I{k_j}$.
If, in addition, $\cG (S,T)\geq r$, then we set
\[
S_j = \sum_{l=1}^{k_j} \a{S_{l,j}},
\]
and we define
\[
\vartheta (S) = \sum_{j=1}^J
\sum_{l=1}^{k_j} \a{\frac{2r- \cG(T,S)}{\cG(T,S)} 
(S_{l,j} - Q_j) + Q_j}.
\]
We then extend $\vartheta$ 
to $\Iq$ by setting
\begin{equation*}\label{e:def theta_S}
\vartheta (S)=
\begin{cases}
T & \textrm{if }\, S\notin B_{2r}(T),\\
S & \textrm{if }\, S\in B_{r} (T).
\end{cases}
\end{equation*}
It is immediate to verify that $\vartheta$ is continuous and
has all the required properties.
\end{proof}

\subsection{Proof of Proposition \ref{p:split}}
The key idea is simple. If the
separation of $T$ were not too small, we could
apply directly Proposition \ref{p:maximum}.
When the separation of $T$ is small, we can find
a point $S$ which is not too far from $T$ and
whose separation is sufficiently large.
Roughly speaking, it suffices to ``collapse'' the
points of the support of $T$ which are too close.

\begin{lemma}\label{l:separation}
For every $0<\eps<1$, we set $\beta(\eps,Q) = (\eps/3)^{3^Q}$. 
Then, for every $T\in\Iq$ with $s(T)<+\infty$, there exists a point  $S\in\Iq$ such that
\begin{align}
\beta (\eps, Q)\,d(T) &\leq s(S)<+\infty,\label{e:lontananza} \\
\cG(S,T) &\leq \eps\,s(S).\label{e:vicinanza}
\end{align}
\end{lemma}

Assuming Lemma \ref{l:separation}, we conclude the proof
of Proposition \ref{p:split}.
Set $\eps=1/8$ and $\alpha (Q) =\eps\,\beta (\eps, Q) = 24^{-3^Q}/8$.
From Lemma \ref{l:separation},
we deduce the existence of an $S$ satisfying \eqref{e:lontananza} and
\eqref{e:vicinanza}.
Then, there exists $\delta>0$ such that, for almost every $x\in\de\Om$,
\begin{equation*}
\cG(f(x),S)\leq \cG(f(x),T)+\cG(T,S)\stackrel{\eqref{e:vicinanza}}{\leq}
\alpha(Q)\,d(T)+\frac{s(S)}{8}-\delta\stackrel{\eqref{e:lontananza}}{\leq}
\frac{s(S)}{4}-\delta.
\end{equation*}
So, we may apply Proposition \ref{p:maximum} and infer that
$\cG(f(x),S)\leq \frac{s(S)}{4}-\delta$ for almost every $x$ in $\Om$.
The decomposition of $f$ in simpler $\D$-minimizing functions
is now a simple consequence of the definitions. More precisely,
if $S=\sum_{j=1}^Jk_j\a{Q_j}\in\Iq$, with the $Q_j$'s all different, then
$f(x)=\sum_{j=1}^J\a{f_j(x)}$, where the $f_j$'s are $\D$-minimizing
$k_j$-valued functions
with values in the balls
$B_{\frac{s(S)}{4}-\delta}(k_j\a{Q_j})$.

\begin{proof}[Proof of Lemma \ref{l:separation}]
For $Q\leq 2$, we have $d(T)\leq s(T)$
and it suffices to choose $S=T$.
We now prove the general case by induction. Let $Q\geq 3$
and assume the lemma holds for $Q-1$.
Let $T=\sum_i\a{P_i}\in\Iq$. Two cases can occur:
\begin{itemize}
\item[$(a)$] either $s(T)\geq(\eps/3)^{3^Q}\,d(T)$;
\item[$(b)$] or $s(T)< (\eps/3)^{3^Q}\,d(T)$.
\end{itemize}
In case $(a)$, since the separation of $T$ is sufficiently large,
the point $T$ itself, i.e. $S=T$, fulfils 
\eqref{e:vicinanza} and \eqref{e:lontananza}.
In the other case, since the points
$P_i$ are not all equal ($s(T)<\infty$),
we can take $P_1$ and $P_2$
realizing the separation of $T$, i.e.
\begin{equation}\label{e:scelta}
|P_1-P_2| = s(T) \leq \left(\frac{\eps}{3}\right)^{3^Q}d(T).
\end{equation}
Moreover, since $Q\geq 3$, we may also assume
that, suppressing $P_1$, we do not reduce the diameter, i.e. that
\begin{equation}\label{e:scelta3}
d(T)=d\big(\tilde T\big),\quad\textrm{where}
\quad \tilde T=\sum_{i=2}^Q\a{P_i}.
\end{equation}
For $\tilde T$, we are now in the position to use the inductive hypothesis (with $\eps/3$ in place of $\eps$).
Hence, there
exists $\tilde S=\sum_{j=1}^{Q-1}\a{Q_j}$ such that
\begin{equation}\label{e:inductive step}
\left(\frac{\eps}{9}\right)^{3^{Q-1}} d \left(\tilde{T}\right) 
\leq s\left(\tilde S\right)\quad\textrm{and}\quad
\cG\left(\tilde S,\tilde T\right)\leq \frac{\eps}{3}s\left(\tilde S\right).
\end{equation}
Without loss of generality, we can assume that
\begin{equation}\label{e:scelta2}
|Q_1-P_2|\leq\cG\big(\tilde S,\tilde T\big) .
\end{equation}
Therefore, $S=\a{Q_1}+\llbracket\tilde S\rrbracket\in\Iq$ 
satisfies \eqref{e:lontananza} and \eqref{e:vicinanza}.
Indeed, since $s(S)=s(\tilde S)$, we infer
\begin{equation}\label{e:lontananza3}
\left(\frac{\eps}{3}\right)^{3^Q} d(T)
\stackrel{\eqref{e:scelta3}}{\leq}
\frac{\eps}{3} \left(\frac{\eps}{9}\right)^{3^{Q-1}}
\,d\left(\tilde T\right)\stackrel{\eqref{e:inductive step}}{\leq}
\frac{\eps}{3}\, s\left(\tilde S\right)= \frac{\eps}{3}\, s(S),
\end{equation}
and
\begin{align*}
\cG(S,T) &\leq \cG\big(\tilde S,\tilde T\big)+|Q_1-P_1|\leq
\cG\big(\tilde S,\tilde T\big)+|Q_1-P_2|+|P_2-P_1|\\
& \stackrel{\mathclap{\eqref{e:scelta},\,\eqref{e:scelta2}}}{\leq} \quad 
2\,\cG\big(\tilde S,\tilde T\big)+\left(\frac{\eps}{3}\right)^{3^Q}\,d(T)
\stackrel{\eqref{e:inductive step}, \,\eqref{e:lontananza3}}{\leq }\frac{2\,\eps}{3}\,s(S)+\frac{\eps}{3}\,s(S)=\eps\, s(S).
\end{align*}
\end{proof}

\section{H\"older regularity}\label{s:hoelder}

Now we pass to prove the H\"older continuity of
$\D$-minimizing $Q$-valued functions.
Theorem \ref{t:hoelder} is indeed a simple consequence of the 
following theorem.

\begin{theorem}\label{t:hoelder2}
There exist constants $\alpha=\alpha(m,Q)\in ]0,1[$
(with $\alpha=\frac{1}{Q}$ when $m=2$) and 
$C=C(m,n,Q,\delta)$ with the following property.
If $f:B_1\ra\Iq$ is $\D$-minimizing, then
\begin{equation*}
\left[f\right]_{C^{0,\alpha}\left(\overline{B_{\delta}}\right)}
=\sup_{x,y\in
\overline{B_{\delta}}}\frac{\cG(f(x),f(y))}{\abs{x-y}^\alpha}
\leq C\,\D(f,\Om)^{\frac{1}{2}} \quad \mbox{for every $0<\delta<1$.}
\end{equation*}
\end{theorem}

The proof of Theorem \ref{t:hoelder2} consists of two parts: 
the first is stated
in the following proposition which gives the crucial estimate;
the second is a standard application of the Campanato--Morrey
estimates (see Section 
\ref{s:Qsobolev}, Proposition \ref{p:Campanato Morrey}).

\begin{propos}\label{p:basic}
Let $f\in W^{1,2}(B_r,\Iq)$ be $\D$-minimizing and suppose that
\begin{equation*}
g=f\vert_{\de B_r}\in W^{1,2}(\de B_r,\Iq).
\end{equation*}
Then, we have that
\begin{equation}\label{e:basic}
\D(f,B_r)\leq C(m)\,r\,\D(g,\de B_r),
\end{equation}
where $C(2)=Q$ and $C(m)<(m-2)^{-1}$.
\end{propos}

The minimizing property of $f$ enters heavily in the proof of
this last proposition, where the estimate is achieved by
exhibiting a suitable competitor.
This is easier in dimension $2$ because 
we can use Proposition \ref{p:Wselection} for $g$.
In higher
dimension the argument is more complicated
and relies on Proposition \ref{p:split} to argue by induction on $Q$.
Now, assuming Proposition \ref{p:basic}, we proceed with the proof of Theorem \ref{t:hoelder2}.
\subsection{Proof of Theorem \ref{t:hoelder2}.}
Set 
\[
\gamma (m):=\left\{
\begin{array}{ll}
2\, Q^{-1} & \mbox{for $m=2$},\\
C(m)^{-1}-m+2& \mbox{for $m>2$,}
\end{array}\right.
\]
where $C(m)$ is the constant in \eqref{e:basic}. 
We want to prove that
\begin{equation}\label{e:scaling}
\int_{B_r}|Df|^2\leq
r^{m-2+\gamma}\int_{B_1}|Df|^2
\quad\mbox{for every $0<r\leq 1$.} 
\end{equation}
Define $h(r)=\int_{B_r}|Df|^2$. Note that $h$ is absolutely
continuous and that
\begin{equation}\label{e:g'}
h^\prime(r)=\int_{\de B_r}|Df|^2
\geq\D(f,\de B_{r}) \quad \mbox{for a.e. $r$,}
\end{equation}
where, according to Definitions \ref{d:W1p}
and \ref{d:dirichlet}, $\D(f,\de B_{r})$ is
given by
\[
\D(f,\de B_{r})=\int_{\de B_r}|\de_{\tau} f|^2,
\]
with $|\de_{\tau} f|^2=|D f|^2-\sum_{i=1}^Q|\de_{\nu} f_i|^2$.
Here $\de_\tau$ and $\de_\nu$ denote, respectively, 
the tangential and the normal derivatives.
We remark further that \eqref{e:g'} can be improved for $m=2$.
Indeed, in this case the outer variation formula \eqref{e:cono}, 
gives an equipartition of the Dirichlet energy in the 
radial and tangential parts, yielding
\begin{equation}\label{e:g'2}
h^\prime(r)=\int_{\de B_r}|Df|^2=\frac{\D(f,\de B_r)}{2}.
\end{equation}
Therefore, \eqref{e:g'} (resp. \eqref{e:g'2} when
$m=2$) and \eqref{e:basic} imply
\begin{equation}\label{e:scaling2}
(m-2+\gamma)\,h(r)\leq r\,h^\prime(r).
\end{equation}
Integrating this differential inequality, we obtain (\ref{e:scaling}):
\begin{equation*}
\int_{B_r}|Df|^2
= h(r)\leq r^{m-2+\gamma}\,h(1)
= r^{m-2+\gamma}\int_{B_1}|Df|^2.
\end{equation*}

Now we can use the Campanato--Morrey estimates for $Q$-valued
functions given in Proposition \ref{p:Campanato Morrey} in order to conclude the H\"older continuity of $f$ with exponent $\alpha=\frac{\gamma}{2}$.

\subsection{Proof of Proposition \ref{p:basic}: the planar case.}
It is enough to prove \eqref{e:basic} for $r=1$,
because the general case follows from an easy scaling argument.
We first prove the following simple lemma.

\begin{remark}\label{r:complex notation} 
In this subsection we introduce a complex notation
which will be also useful later. We identify the plane $\R{2}$
with $\C$ and therefore we regard the unit disk as 
\[
\disk=\{z\in\C\,:\,|z|<1\}=\{r\,e^{i\theta}\,:\,0\leq r<1,\,\theta\in\R{}\}
\] 
and the unit circle as 
\[
\s^1=\partial\, \disk=\{z\in\C\,:\, |z|=1\} =
\{e^{i\theta}\,:\, \theta\in \R{}\}.
\]
\end{remark}

\begin{lemma}\label{l:rolling}
Let $\zeta\in W^{1,2}(\disk,\R{n})$ and consider
the $Q$-valued function $f$ defined by
\begin{equation*}
f(x)=\sum_{z^{Q}=x}\a{\zeta(z)}.
\end{equation*}
Then, the function $f$ belongs to $W^{1,2}(\disk,\Iq)$ and 
\begin{equation}\label{e:invariance Dir}
\D(f,\disk)=\int_{\disk}\abs{D\zeta}^2.
\end{equation}
Moreover, if $\zeta\vert_{\s^1}\in W^{1,2}(\s^1,\R{n})$, then $f\vert_{\s^1}\in W^{1,2}(\s^1,\Iq)$ and
\begin{equation}\label{e:Dir al bordo}
\D(f\vert_{\s^1},\s^1)=\frac{1}{Q}\int_{\s^1}\abs{\de_\tau\zeta}^2.
\end{equation}
\end{lemma}

\begin{proof} Define the following subsets of the unit disk,
\[
\mathcal{D}_j=\left\{r\,e^{i\theta}:0<r<1,\,
(j-1)\,2\pi/Q<\theta<j\,2\pi/Q\right\}\;
\textrm{and }\; \mathcal{C}=\left\{r\,e^{i\theta}:0<r<1,\,\theta\neq0\right\},
\]
and let $\varphi_j:\mathcal{C}\ra \mathcal{D}_j$ 
be determinations of the $Q^{th}$-root, i.e. 
\[
\varphi_j \left(r e^{i\theta}\right)=r^{\frac{1}{Q}}\,
e^{i\left(\frac{\theta}{Q}+(j-1)\,\frac{2\pi}{Q}\right)}.
\]
It is easily recognized that $f\vert_{\mathcal{C}}=\sum_{j}\a{\zeta\circ \varphi_j}$.
So, by the invariance of the Dirichlet energy under conformal mappings,
one deduces that
$f\in W^{1,2}(\mathcal{C},\Iq)$ and
\begin{equation}\label{e:invariance2}
\D(f,\mathcal{C})=\sum_{i=1}^Q\D(\zeta\circ \varphi_i,\mathcal{C})=
\int_{\disk}\abs{D\zeta}^2.
\end{equation}
From the above argument and from \eqref{e:invariance2}, it is straightforward to infer that $f$ belongs to $W^{1,2}(\disk,\Iq)$ and \eqref{e:invariance Dir}  holds. Finally, \eqref{e:Dir al bordo} is a 
simple computation left to the reader.
\end{proof}

We now prove Proposition \ref{p:basic}.
Let $g=\sum_{j=1}^J\a{g_j}$ be a decomposition
into irreducible $k_j$-functions 
as in Proposition \ref{p:Wselection}.
Consider, moreover, the $W^{1,2}$ 
functions $\gamma_j :\s^1\to \R{n}$ ``unrolling'' the $g_j$
as in Proposition \ref{p:Wselection} $(ii)$:
\begin{equation*}
g_j (x)=\sum_{z^{k_j} = x} \a{\gamma_j (z)}.
\end{equation*}
We take the harmonic extension $\zeta_l$ of $\gamma_l$ in $\disk$, and
consider the $k_l$-valued functions $f_l$ obtained
``rolling'' back the $\zeta_l$: $f_l (x)=\sum_{z^{k_l}=x} 
\a{\zeta_l (z)}$.
The $Q$-function $\tilde{f}=\sum_{l=1}^J\a{f_l}$ is
an admissible competitor for $f$, 
since $\tilde{f}\vert_{\s^1}= f|_{\s^1}$.
By a simple computation on planar harmonic functions, it is easy to see that
\begin{equation}\label{e:standard}
\int_{\disk}\abs{D\zeta_l}^2\leq \int_{\s^1}
\abs{\de_\tau\gamma_l}^2.
\end{equation}
Hence, from \eqref{e:invariance Dir},
\eqref{e:Dir al bordo} and \eqref{e:standard}, we easily conclude \eqref{e:basic}:
\begin{align*}
\D(f,\disk) & \leq \D\left(\tilde{f},\disk\right) =
\sum_{l=1}^J\D(f_l,\disk)\stackrel{\eqref{e:invariance Dir}}{=}\sum_{l=1}^J\int_{\disk}\abs{D\zeta_l}^2\\
& \stackrel{\mathclap{\eqref{e:standard}}}{\leq}\;
\sum_{l=1}^J\int_{\s^1}\abs{\de_\tau\gamma_l}^2
\stackrel{\eqref{e:Dir al bordo}}{=}
\sum_{l=1}^J k_l\, \D(g_l,\s^1) \leq Q\, \D(g,\s^1).
\end{align*}

\subsection{Proof of Proposition \ref{p:basic}: the case $m\geq3$.}
\label{ss:m>3}

To understand the strategy of the proof, fix a $\D$-minimizing
$f$ and consider the ``radial'' competitor
$h(x)=f (x/|x|)$.
An easy computation shows
the inequality $\D(h,B_1)\leq (m-2)^{-1}\D(f,\de B_1)$.
In order to find a better competitor, set
$\tilde f (x)=\sum_i \a{\varphi(\abs{x})f_i (x/|x|)}$.
With a slight abuse of notation, we will denote
this function by $\varphi (|x|) f (x/|x|)$.
We consider moreover functions $\varphi$ which are 
$1$ for $t=1$ and smaller than $1$
for $t<1$. These competitors are, however, good
only if $f|_{\de B_1}$ is not too far from
$Q\a{0}$. 

Of course, we can use competitors of the form
\begin{equation}\label{e:radialcomp}
\sum_i \a{v + \varphi (|x|) 
\left(f_i \left(\frac{x}{\abs{x}}\right)
-v\right)},
\end{equation}
which are still suitable if, roughly speaking,
\begin{itemize}
\item[(C)] on $\partial B_1$,
$f (x)$ is not too far from $Q\a{v}$, i.e. from a point of
multiplicity $Q$.
\end{itemize}
A rough strategy of the proof could then be
the following.
We approximate $f|_{\partial B_1}$
with a $\tilde{f}=\a{f_1}+\ldots
+ \a{f_J}$ decomposed into simpler
$W^{1,2}$ functions $f_j$ each of which satisfies (C). 
We interpolate on a corona $B_1\setminus B_{1-\delta}$
between $f$ and $\tilde{f}$, and we then use 
the competitors of the form \eqref{e:radialcomp} to
extend $\tilde{f}$ to $B_{1-\delta}$. 
In fact, we shall use a variant of this idea, 
arguing by induction on $Q$.

\medskip

Without loss of generality, we assume that
\begin{equation}\label{e:normalization}
\D(g,\de B_1)=1.
\end{equation}
Moreover, we recall the notation $|T|$ and $|f|$ introduced
in Remark \ref{r:convention} and fix the following one for 
the translations:
\begin{equation*}
\mbox{if $v\in \R{n}$},\; \textrm{then}
\; \tau_v (T) := \sum_i \a{T_i-v},\;\textrm{for every }\,
T=\sum_i \a{T_i}\in \Iq.
\end{equation*}

\medskip

\textit{Step 1. Radial competitors.}
Let $\overline g=\sum_i \a{P_i}\in \Iq$
be a mean for $g$, so that the Poincar\'e
inequality in Proposition \ref{p:poincare'} holds,
and assume that the diameter of $\bar g$ (see
Definition \ref{d:diamsep}) is smaller than a constant $M>0$,
\begin{equation}\label{e:vicini}
d(\overline g)\leq M.
\end{equation}

Let $P=Q^{-1} \sum_{i=1}^QP_i$ be the barycenter 
of $\overline g$ and consider
$\tilde{f}=\tau_P\circ f$ and $h =\tau_P\circ g$.
It is clear that $h =\tilde f\vert_{\de B_1}$
and that $\overline{h}=\tau_P(\overline g)$ is a
mean for $h$.
Moreover, by \eqref{e:vicini},
\[
\left|\overline{h}\right|^2=\sum_{i}|P_i-P|^2\leq Q\, M^2.
\]
So, using the Poincar\'e inequality, we get
\begin{equation}\label{e:poicare' vicini}
\int_{\de B_1}|h|^2 \leq
2\int_{\de B_1} \cG \left(h, \overline{h}\right)^2 
+ 2\int_{\de B_1}\left|\overline{h}\right|^2
\leq C \,\D (g, \partial B_1) +C\,M^2\stackrel{\eqref{e:normalization}}{\leq} C_M,
\end{equation}
where $C_M$ is a constant depending on $M$.

We consider the $Q$-function
$\hat{f} (x) := \varphi(\abs{x})\, h \left(\frac{x}{\abs{x}}\right)$,
where $\varphi$ is a $W^{1,2}([0,1])$ function with $\varphi(1)=1$.
From \eqref{e:poicare' vicini} and the chain-rule in Proposition \ref{p:chain},
one can infer the following estimate:
\begin{align*}
\int_{B_1}\left|D\hat{f}\right|^2 & = 
\left(\int_{\de B_1} |h|^2\right) \int_0^1\varphi^\prime(r)^2\,r^{m-1}dr
+\left(\int_{\de B_1}|D h|^2 \right)
\int_0^1\varphi(r)^2\,r^{m-3}dr\nonumber\\
& \leq 
\int_0^1\big(\varphi(r)^2\,r^{m-3}+C_M\varphi^\prime(r)^2\,r^{m-1}\big)dr
=:  I(\varphi).
\end{align*}
Since $\tau_{-P} \big(\hat{f}\big)$ is a suitable competitor for
$f$, one deduces that
\begin{equation*}
\D (f, B_1) \leq \inf_{\stackrel{\varphi
\in W^{1,2}([0,1])}{\varphi(1)=1}} I(\varphi).
\end{equation*}
We notice that $I(1)=\frac{1}{m-2}$, as pointed out 
at the beginning of the section.
On the other hand, $\varphi\equiv 1$ cannot be a minimum
for $I$ because it does not satisfy the corresponding 
Euler--Lagrange equation. So, 
there exists a constant $\gamma_M>0$ such that
\begin{equation}\label{e:basic case 1}
\D(f,B_1)\leq \inf_{\stackrel{\varphi
\in W^{1,2}([0,1])}{\varphi(1)=1}} I(\varphi)=
\frac{1}{m-2}-2\,\gamma_M.
\end{equation}
In passing, we note that, when $Q=1$, $d(T)=0$ and hence this argument
proves the first induction step of the proposition (which, however,
can be proved in several other ways).

\medskip

\textit{Step 2. Splitting procedure: the inductive step.}
Let $Q$ be fixed and assume that
the proposition holds for every $Q^*<Q$.
Assume, moreover, that the diameter of $\overline g$ is bigger
than a constant $M>0$, which will be chosen later:
\begin{equation*}
d(\overline g)>M
\end{equation*}
Under these hypotheses, 
we want to construct a suitable competitor for $f$.
As pointed out at the beginning of the proof, the strategy is
to decompose $f$ in suitable pieces in order to apply the inductive
hypothesis. To this aim:
\begin{itemize}
\item[$(a)$] let $S=\sum_{j=1}^Jk_j\a{Q_j}\in\Iq$ 
be given by Lemma \ref{l:separation}
applied to $\eps=\frac{1}{16}$ and $T=\overline g$, 
i.e. $S$ such that
\begin{align}
&\beta\, M\leq \beta \,d(\overline g)< s(S)
=\min_{i\neq j} |Q_i-Q_j|,\label{e:lontananza2}\\
&\cG(S,\overline g)<\frac{s(S)}{16},\label{e:vicinanza2}
\end{align}
where $\beta=\beta(1/16,Q)$ is the constant 
of Lemma \ref{l:separation};
\item[$(b)$] let $\vartheta:\Iq\ra B_{s(S)/8}(S)$ be
given by Lemma \ref{l:retraction}
applied to $T=S$ and $r=\frac{s(S)}{8}$.
\end{itemize}
We define $h \in W^{1,2}(\de B_{1-\eta})$  by $h
\left((1-\eta )x\right)=\vartheta \left(g(x)\right)$, where $\eta>0$ is
a parameter to be fixed later, and take 
$\hat{h}$ a $\D$-minimizing $Q$-function on $B_{1-\eta}$
with trace $h$. Then, we consider the following competitor,
\begin{equation*}
\tilde f=
\begin{cases}
\hat h \quad\textrm{on}\quad B_{1-\eta}\\
\textrm{interpolation between $\hat h$ and $g$
as in Lemma \ref{l:technical},}
\end{cases}
\end{equation*}
and we pass to estimate its Dirichlet energy.

By Proposition \ref{p:split}, since $\hat h$ has values in $\overline{B_{s(S)/8}(S)}$, $\hat{h}$ can be decomposed
into two $\D$-minimizing $K$ and $L$-valued functions,
with $K, L < Q$. So, by inductive hypothesis,
there exists a positive constant $\zeta$ such that
\begin{equation}\label{e:guadagno}
\D \left(\hat{h}, B_{1-\eta}\right) \leq \left(\frac{1}{m-2} - \zeta\right)
(1-\eta)\,\D (h, \de B_{1-\eta} )\leq \left(\frac{1}{m-2} - \zeta\right)
\D (g, \de B_1),
\end{equation}
where the last inequality follows from $\Lip(\vartheta)=1$. 

Therefore, combining \eqref{e:guadagno}
with Lemma \ref{l:technical}, we can estimate
\begin{equation}\label{e:energy estimate}
\D\left(\tilde f,B_1\right)\leq
  \left(\frac{1}{m-2} - \zeta + C \eta \right)
\D (g, \de B_1) +\frac{C}{\eta}\int_{\de B_1}
\cG\big(g,\vartheta (g)\big)^2,
\end{equation}
with $C=C(n,m,Q)$.
Note that
\begin{equation*}
\cG\left(\overline g, \vartheta (g(x))\right)
\leq \cG\left(g(x),\overline g\right)\quad
\textrm{for every } x\in \de B_1,
\end{equation*}
because, by \eqref{e:vicinanza2}, $\vartheta(\overline g)=\overline g$.
Hence, if we define 
\[
E:=\big\{x\in \de B_1: g(x)\neq \vartheta (g(x))\big\}=
\left\{x\in \de B_1: g(x)\notin\overline{B_{s(S)/8}(S)}\right\},
\]
the last term in \eqref{e:energy estimate} 
can be estimated as follows:
\begin{align}\label{e:troncamento}
\int_{\de B_1}\cG\big(g,\vartheta (g)\big)^2
& =  \int_E
\cG\big(g,\vartheta (g)\big)^2
\leq 2\int_E \left[\cG\big(g,\overline g\big)^2
+ \cG\big(\overline g,\vartheta (g)\big)^2\right]\nonumber\\
&\leq 4\,\int_E \cG\big(g,\overline g\big)^2 dx
\leq 4\, \|\cG \big(g,\overline{g}\big)^2\|_{L^q}\,
|E|^{(q-1)/q}\nonumber\\
& \leq C\, \D (g,\de B_1) \,|E|^{(q-1)/q}
= C \,|E|^{(q-1)/q},
\end{align}
where the exponent $q$ can be chosen to be $(m-1)/(m-3)$
if $m>3$, otherwise any $q<\infty$ if $m=3$.

We are left only with the estimate of $|E|$.
Note that, for every $x\in E$,
\[
\cG(g(x),\overline g)\geq\cG(g(x),S)-\cG(\overline g,S)
\stackrel{\eqref{e:vicinanza2}}{\geq}\frac{s(S)}{8}-\frac{s(S)}{16}=\frac{s(S)}{16}.
\]
So, we deduce that
\begin{equation}\label{e:bound1000}
|E| \leq
\left|\left\{\cG(g,\overline g)\geq\frac{s(S)}{16}\right\}\right|
\leq
\frac{C}{s(S)^2}\int_{\de B_1}\cG(g,\overline g)^2
\stackrel{\eqref{e:lontananza2}}{\leq}\frac{C}{M^2}
\D(g,\de B_1).
\end{equation}

Hence, collecting the bounds \eqref{e:guadagno},
\eqref{e:troncamento} and \eqref{e:bound1000}, we conclude that
\begin{equation}\label{e:basic case 2}
\D\big(\tilde f,B_1\big)\leq \left(\frac{1}{m-2} - \zeta + C\eta 
+ \frac{C}{\eta M^\nu}\right),
\end{equation}
where $C=C(n,m,Q)$ and $\nu=\nu(m)$.

\medskip

\textit{Step 3. Conclusion.}
We are now ready to conclude. First of all,
note that $\zeta$ is a fixed positive constant given
by the inductive assumption that the proposition
holds for $Q^*<Q$. We then choose $\eta$ so that 
$C\eta < \zeta/2$ and
$M$ so large that $C/(\eta M^{\nu})< \zeta/4$,
where $C$
is the constant in \eqref{e:basic case 2}.
Therefore, the constants $M$, $\gamma_M$ and $\eta$ depend only on
$n$,$m$ and $Q$.
With this choice, Step 2 shows that
\begin{equation*}
\D(f,B_1)\leq \D\big(\tilde f, B_1\big)
\stackrel{\eqref{e:basic case 2}}{\leq}
\left(\frac{1}{m-2} - \frac{\zeta}{4}\right)
\D (g, \partial B_1),
\quad\mbox{if $d (\overline{g})>M$};
\end{equation*}
whereas Step 1 implies
\begin{equation*}
\D(f,B_1)\stackrel{\eqref{e:basic case 1}}{\leq}
\left(\frac{1}{m-2}-2\,\gamma_M\right) \D (g, \partial B_1),
\quad\mbox{if $d (\overline{g})\leq M$}.
\end{equation*}
This concludes the proof.

\section{Frequency function}\label{s:frequency}

We next introduce Almgren's frequency function and
prove his celebrated estimate.

\begin{definition}[The frequency function]\label{d:frequency}
Let $f$ be a $\D$-minimizing function, $x\in \Omega$
and $0<r<\dist (x, \partial \Omega)$. We define the
functions
\begin{equation}\label{e:frequency}
D_{x,f} (r)=\int_{B_r (x)} |Df|^2, \quad
H_{x,f} (r)=\int_{\partial B_r} |f|^2\quad\textrm{and}\quad
I_{x,f} (r)=\frac{r D_{x,f} (r)}{H_{x,f} (r)}. 
\end{equation}
$I_{x,f}$ is called the {\em frequency function}.
\end{definition}

When $x$ and $f$ are clear from the context, we will often
use the shorthand notation $D(r)$, $H(r)$ and $I(r)$.

\begin{remark}\label{r:welldefined}
Note that, by Theorem \ref{t:hoelder2}, $|f|^2$ is a continuous 
function. Therefore, $H_{x,f} (r)$ is a well-defined
quantity for every $r$. Moreover, if $H_{x,f} (r)=0$,
then, by minimality, $f|_{B_r (x)} \equiv 0$. 
So, except for this case, $I_{x,f} (r)$ is always well defined. 
\end{remark}

\begin{theorem}\label{t:frequency}
Let $f$ be $\D$-minimizing and $x\in \Omega$.
Either there exists $\varrho$ such that $f|_{B_\varrho (x)}\equiv 0$
or $I_{x,f} (r)$ is an absolutely continuous nondecreasing positive function
on $]0, \dist (x, \partial \Omega)[$.
\end{theorem}

A simple corollary of Theorem \ref{t:frequency} is the
existence of the limit
\begin{equation*}
I_{x,f}(0)=\lim_{r\to0}I_{x,f}(r),
\end{equation*}
when the frequency function is defined for every $r$.
The same computations as those in Theorem \ref{t:frequency}
yield the following two corollaries.

\begin{corol}\label{c:homogeneous}
Let $f$ be $\D$-minimizing in $B_\varrho$.
Then, $I_{0,f} (r) \equiv \alpha$ if and only if
$f$ is $\alpha$-homogeneous, i.e.
\begin{equation}\label{e:homogeneous}
f(y)=|y|^\alpha f\left(\frac{y\,\varrho}{|y|}\right).
\end{equation}
\end{corol}

\begin{remark}
In \eqref{e:homogeneous}, with a slight abuse of notation, we use
the following convention (already
adopted in Subsection \ref{ss:m>3}). 
If $\beta$ is a scalar function and
$f=\sum_i \a{f_i}$ a $Q$-valued function, we denote by
$\beta f$ the function $\sum_i \a{\beta\, f_i}$.
\end{remark}

\begin{corol}\label{c:estimates frequency}
Let $f$ be $\D$-minimizing in $B_\varrho$.
Let $0<r<t\leq\varrho$ and suppose that $I_{0,f} (r)=I(r)$ is defined
for every $r$ (i.e. $H(r)\neq 0$ for every $r$).
Then, the following estimates hold: 
\begin{itemize}
\item[$(i)$] for almost every $r\leq s\leq t$,
\begin{equation}\label{e:deriv}
\frac{d}{d\,\tau}\Big\vert_{\tau=s}
\left[\ln\left(\frac{H(\tau)}{\tau^{m-1}}\right)\right]
=\frac{2\,I(r)}{r}
\end{equation}
and
\begin{equation}\label{e:stima H}
\left(\frac{r}{t}\right)^{2I(t)}\frac{H(t)}{t^{m-1}}
\leq \frac{H(r)}{r^{m-1}}\leq
\left(\frac{r}{t}\right)^{2I(r)}\frac{H(t)}{t^{m-1}};
\end{equation}
\item[$(ii)$] if $I(t)>0$, then
\begin{equation}\label{e:stima D}
\frac{I(r)}{I(t)}\left(\frac{r}{t}\right)^{2I(t)}\frac{D(t)}{t^{m-2}}
\leq \frac{D(r)}{r^{m-2}}\leq
\left(\frac{r}{t}\right)^{2I(r)}\frac{D(t)}{t^{m-2}}.
\end{equation}
\end{itemize}
\end{corol}

\subsection{Proof of Theorem \ref{t:frequency}}
We assume, without loss of generality, that
$x=0$. 
$D$ is an absolutely continuous function
and 
\begin{equation}\label{e:D'}
D^\prime(r)=\int_{\de B_r}|Df|^2\, \quad \mbox{for a.e. $r$}.
\end{equation}
As for $H(r)$, note that $|f|$ is the composition
of $f$ with a Lipschitz function, and therefore
belongs to $W^{1,2}$. It follows that $|f|^2\in W^{1,1}$
and hence that $H\in W^{1,1}$.

In order to compute $H'$, note that the distributional
derivative of $|f|^2$ coincides with the approximate differential 
a.e. Therefore, Proposition \ref{p:chainapp} justifies 
(for a.e. $r$) the following computation:
\begin{align*}
H^\prime(r)&=
\frac{d}{dr}\int_{\de B_1}r^{m-1}\,|f(ry)|^2dy
= (m-1)r^{m-2}\int_{\de B_1}|f(ry)|^2dy+
\int_{\de B_1}r^{m-1} \frac{\de}{\de r} |f(ry)|^2\, dy\nonumber\\
&=\frac{m-1}{r}\int_{\partial B_r} |f|^2
+ 2 \int_{\partial B_r} \sum_i 
\langle \partial_\nu f_i, f_i\rangle.
\end{align*}
Using \eqref{e:cono}, we then conclude
\begin{equation}\label{e:H'3}
H^\prime (r)= \frac{m-1}{r} H (r) + 2 D(r).
\end{equation}
Note, in passing, that,
since $H$ and $D$ are continuous, $H\in C^1$ and
\eqref{e:H'3} holds pointwise.

If $H(r)=0$ for some $r$, then,
as already remarked, $f|_{B_r}\equiv 0$. In the opposite
case, we conclude that $I\in C \cap W^{1,1}_{loc}$. 
To show that $I$ is nondecreasing, it suffices
to compute its derivative a.e. and prove that it is
nonnegative.  Using 
\eqref{e:D'} and \eqref{e:H'3}, we infer that
\begin{align}\label{e:I'0}
I' (r)&=\frac{D(r)}{H(r)}+
\frac{r\,D^\prime(r)}{H(r)}-r\,D(r)\frac{H^\prime(r)}{H(r)^2}\nonumber\\
&=\frac{D(r)}{H(r)}+\frac{r\,D^\prime(r)}{H(r)}-
(m-1)\frac{D(r)}{H(r)}-2r\frac{D(r)^2}{H(r)^2}\nonumber\\
&=\frac{(2-m)D(r)+r\,D^\prime(r)}{H(r)}
-2\,r\,\frac{D(r)^2}{H(r)^2}\quad \mbox{for a.e. $r$.}
\end{align}
Recalling \eqref{e:cono} and \eqref{e:perparti} and using
the Cauchy--Schwartz inequality, from \eqref{e:I'0} we conclude that, for almost every
$r$,
\begin{equation}\label{e:I'}
I^\prime(r) =
\frac{r}{H(r)^2}\left\{
\int_{\de B_r(x)}
|\de_\nu f|^2
\cdot
\int_{\de B_r(x)}
|f|^2
-\left(\int_{\partial B_r (x)}
\sum_i \langle \partial_\nu f_i, f_i\rangle\,\right)^2\right\}\geq0.
\end{equation}

\subsection{Proof of Corollary \ref{c:homogeneous}}

Let $f$ be a $\D$-minimizing $Q$-valued function.
Then, $I(r)\equiv\alpha$
if and only if equality occurs in \eqref{e:I'} for almost every $r$, i.e. if and
only if there exist constants
$\lambda_r$ such that 
\begin{equation}\label{e:I cost}
f_i(y)=\lambda_r\,\de_\nu f_i(y),\;\textrm{for almost every $r$ and a.e. $y$ with $\abs{y}=r$.}
\end{equation}

Recalling \eqref{e:perparti} and using \eqref{e:I cost}, we infer that, for such $r$,
\begin{equation*}\label{e: I const2}
\alpha = I(r) = \frac{r\,D(r)}{H(r)}=\frac{r\int_{\partial B_r }
\sum_i \langle \partial_\nu f_i, f_i\rangle}{\int_{\de B_r}\sum_i
|f_i|^2}
\stackrel{\eqref{e:I cost}}{=} \frac{r\lambda_r\int_{\de B_r}
\sum_i |f_i|^2}{\int_{\de B_r}\sum_i |f_i|^2}
=r\lambda_r.
\end{equation*}
So, summarizing, $I(r)\equiv\alpha$ if and only if
\begin{equation}\label{e:I const3}
f_i(y)=\frac{\alpha}{|y|}\,\de_\nu f_i(y)
\quad \textrm{for almost every $y$.}
\end{equation}

Let us assume that \eqref{e:homogeneous} holds.
Then, \eqref{e:I const3} is
clearly satisfied and, hence, $I(r)\equiv\alpha$.
On the other hand, assuming that the frequency is constant,
we now prove \eqref{e:homogeneous}.
To this aim, let $\sigma_y=
\{r\,y\,:\,0\leq r\leq \varrho\}$ be the radius passing through $y\in \de B_1$.
Note that, for almost every $y$,
$f\vert_{\sigma_y} \in W^{1,2}$; so, for those $y$, recalling the
$W^{1,2}$-selection in Proposition \ref{p:Wselection-1}, we can write
$f\vert_{\sigma_y}=\sum_i \a{f_i\vert_{\sigma_y}}$, where $f_i\vert_{\sigma_y}:[0,\varrho]\ra \R{n}$ are $W^{1,2}$ functions.
By \eqref{e:I const3}, we infer
that $f_i\vert_{\sigma_y}$ solves the ordinary differential equation
\begin{equation*}
(f_i\vert_{\sigma_y})^\prime(r)=\frac{\alpha}{r}\,f_i\vert_{\sigma_y}(r),\quad\textrm{for
a.e. }\,r.
\end{equation*}
Hence, for a.e. 
$y\in \de B_{1}$ and for every $r\in(0,\varrho]$, 
$f_i\vert_{\sigma_y}(r)=r^\alpha\,f\left(y\right)$, 
thus concluding \eqref{e:homogeneous}.

\subsection{Proof of Corollary \ref{c:estimates frequency}}

The proof is a straightforward consequence of equation \eqref{e:H'3}.
Indeed, \eqref{e:H'3} implies, for almost every $s$,
\[
\frac{d}{d\,\tau}\Big\vert_{\tau=s}\left(\frac{H(\tau)}{\tau^{m-1}}\right)=
\frac{H^\prime(s)}{s^{m-1}}-\frac{(m-1)\,H(s)}{s^{m}}
\stackrel{\eqref{e:H'3}}{=}\frac{2\,D(s)}{s^{m-1}},
\]
which, in turn, gives \eqref{e:deriv}.
Integrating \eqref{e:deriv} and using the monotonicity of $I$,
one obtains \eqref{e:stima H}. Finally, \eqref{e:stima D} 
follows from \eqref{e:stima H}, using
the identity $I(r)=\frac{r\,D(r)}{H(r)}$.

\section{Blow-up of $\D$-minimizing $Q$-valued functions}
\label{s:blowup}

Let $f$ be a $Q$-function and assume $f(y)=Q\a{0}$
and $\D(f,B_\varrho(y))>0$ for every $\varrho$.
We define the blow-ups of $f$ at $y$ in the following way,
\begin{equation}\label{e:blowup}
f_{y,\varrho}(x)=\frac{\varrho^{\frac{m-2}{2}}\,f(\varrho\,x+y)}{\sqrt{\D(f,B_\varrho(y))}}.
\end{equation}
The main result of this section is
the convergence of blow-ups of $\D$-minimizing functions
to homogeneous $\D$-minimizing functions, which we call
{\em tangent functions}.

To simplify the notation, we will not display the subscript 
$y$ in $f_{y,\rho}$
when $y$ is the origin.

\begin{theorem}\label{t:blowup}
Let $f\in W^{1,2}(B_1,\Iq)$ be $\D$-minimizing.
Assume $f(0)=Q\a{0}$ and $\D(f,B_\varrho)>0$ for every $\varrho\leq1$. Then,
for any sequence $\{f_{\varrho_k}\}$ with $\rho_k\downarrow 0$,
a subsequence, not relabelled, converges
locally uniformly to a function $g:\R{m}\to \Iqs$
with the following properties:
\begin{itemize}
\item[$(a)$] $\D(g, B_1) = 1$ and $g|_\Om$ is $\D$-minimizing
for any bounded $\Om$;
\item[$(b)$] $g(x)=|x|^\alpha\,g\left(\frac{x}{|x|}\right)$, where
$\alpha=I_{0,f}(0) > 0$ is the frequency of $f$ at $0$.
\end{itemize}
\end{theorem}

Theorem \ref{t:blowup} is a direct consequence of the estimate
on the frequency function and of the following
convergence result for $\D$-minimizing functions.

\begin{propos}\label{p:energy conv}
Let $f_k\in W^{1,2}(\Om,\Iq)$ be $\D$-minimizing
$Q$-functions weakly converging to $f$.
Then, for every open $\Om^\prime\subset\subset\Om$, 
$f|_{\Om^\prime}$ is $\D$-minimizing and it holds moreover that
$\D(f,\Om^\prime)=\lim_k \D(f_k,\Om^\prime)$.
\end{propos}

\begin{remark}
In fact, a suitable modification of our proof
shows that the property of being $\D$-minimizing holds
on $\Omega$. However, we never need this stronger property
in the sequel.
\end{remark}

Assuming Proposition \ref{p:energy conv}, we prove Theorem
\ref{t:blowup}.

\begin{proof}[Proof of Theorem \ref{t:blowup}]
We consider any ball $B_N$ of radius $N$ centered at $0$.
It follows from estimate \eqref{e:stima D} that
$\D(f_\varrho, B_N)$ is uniformly bounded in $\varrho$.
Hence, the functions $f_\varrho$ are all $\D$-minimizing
and Theorem \ref{t:hoelder2} implies that
the $f_{\varrho_k}$'s are locally 
equi-H\"older continuous.
Since $f_\varrho(0)=Q\a{0}$, the $f_\varrho$'s
are also locally uniformly bounded and the Ascoli--Arzel\`a
theorem yields a subsequence
(not relabelled) converging uniformly on compact subsets of $\R{m}$ to
a continuous $Q$-valued function $g$.
This implies easily the weak convergence
(as defined in Definition \ref{d:weak convergence}),
so we can apply Proposition \ref{p:energy conv} and
conclude $(a)$ (note that $\D(f_\varrho,B_1)=1$ for every $\varrho$).
Observe next that, for every $r>0$,
\begin{equation}\label{e:constant freq}
I_{0,g}(r)=\frac{r\,\D(g,B_r)}{\int_{\de B_r}
|g|^2}=\lim_{\varrho\ra0}\frac{r\,\D(f_\varrho,B_r)}{\int_{\de B_r}
|f_\varrho|^2}=\lim_{\varrho\ra0}\frac{\varrho\,r\,\D(f,B_{\varrho\,r})}
{\int_{\de B_{\varrho\,r}} |f|^2}=I_{0,f}(0).
\end{equation}
So, $(b)$ follows from Corollary \ref{c:homogeneous}, once we have
shown that $I_{0,f}(0)>0$. Assume, by contradiction, that
$I_{0,f}(0)=0$. Then, by what shown so far,
the blowups $f_\varrho$ converge to a continuous
$0$-homogeneous function $g$, with $g(0)=Q\a{0}$.
This implies that $g\equiv Q\a{0}$, against conclusion $(a)$,
namely $\D(g,B_1)=1$.
\end{proof}

\begin{proof}[Proof of Proposition \ref{p:energy conv}]
We consider
the case of $\Om=B_1$: the general case is a routine
modification of the arguments 
(and, besides, we never need it in the sequel).
Since the $f_k$'s are $\D$-minimizing and, hence,
locally H\"older equi-continuous,
and since the $f_k$'s converge strongly in
$L^2$ to $f$,
they actually converge to $f$ uniformly on compact sets.
Set $D_{r}=\liminf_{k}\D(f_k,B_r)$
and assume by contradiction that $f|_{B_r}$ is not $\D$-minimizing
or $\D(f,B_r)<D_r$ for some $r<1$. Under this assumption,
we can find $r_0>0$ such that, for every $r\geq r_0$,
there exist a $g\in W^{1,2}(B_{r},\Iq)$ with
\begin{equation}\label{e:non minimal}
g\vert_{\de B_{r}}=f\vert_{\de B_{r}}
\quad\textrm{and}\quad
\gamma_r:= D_r - \D(g,B_{r})>0.
\end{equation}
Fatou's Lemma implies that
$\liminf_{k}\D(f_k,\de B_r)$ is finite
for almost every $r$,
\[
\int_{0}^1\liminf_{k\ra+\infty}\D(f_k,\de B_r)\,dr\leq
\liminf_{k\ra+\infty}\int_{0}^1\D(f_k,\de B_r)\,dr\leq C<+\infty.
\]
Passing, if necessary, to a subsequence, we can fix a radius
$r\geq r_0$ such that
\begin{equation}\label{e:fix r}
\D(f,\de B_r)\leq\lim_{k\ra+\infty}\D(f_k,\de B_r)\leq M<+\infty.
\end{equation}
We now show that \eqref{e:non minimal} contradicts
the minimality of $f_k$ in $B_r$ for large $n$.
Let, indeed, $0<\delta<r/2$ to be fixed later and consider the functions $\tilde f_k$
on $B_r$ defined by
\begin{equation*}
\tilde f_k(x)=
\begin{cases}
g\left(\frac{r\,x}{r-\delta}\right) & \textrm{for }\; x\in B_{r-\delta},\\
h_k(x) & \textrm{for }\; x\in B_r\setminus B_{r-\delta},
\end{cases}
\end{equation*}
where the $h_k$'s are the interpolations provided by
Lemma \ref{l:technical}
between $f_k\in W^{1,2}(\de B_r,\Iq)$
and $g\left(\frac{r\,x}{r-\delta}\right) \in W^{1,2}(B_{r-\delta},\Iq)$.
We claim that, for large $k$, the functions $\tilde f_k$
have smaller Dirichlet energy than $f_k$, thus contrasting the
minimizing property of $f_k$, and concluding the proof.
Indeed, recalling the estimate in Lemma \ref{l:technical}, we have
\begin{align*}
\D\big(\tilde f_k,B_r\big)\leq&
\D\big(\tilde f_k,B_{r-\delta}\big)+
C\,\delta\left[\D\big(\tilde f_k,\de B_{r-\delta}\big)+
\D\big(f_k,\de B_{r}\big)\right]
+\frac{C}{\delta}\hspace{-0.1cm}\int_{\de B_{r}}\cG\big(f_k,\tilde f_k\big)^2\\
\leq{}& \D(g,B_{r})+C\,\delta\,\D(g,\de B_{r})+
C\,\delta\,\D(f_k,\de B_{r})
+\frac{C}{\delta}\int_{\de B_{r}}\cG(f_k, g)^2.
\end{align*}
Choose now $\delta$ such that $4\,C\,\delta\,(M+1)\leq\gamma_r$,
where $M$ and $\gamma_r$ are the constants
in \eqref{e:fix r} and \eqref{e:non minimal}.
Using the uniform convergence of $f_k$ to $f$,
we conclude, for $k$ large enough,
\begin{align*}
\D\big(\tilde f_k,B_r\big) &\quad\stackrel{\mathclap{\eqref{e:non minimal},\,\eqref{e:fix r}}}{\leq}\quad
D_r-\gamma_r+C\,\delta\,M+
C\,\delta\,(M+1)+\frac{C}{\delta}\int_{\de B_{r}}\cG(f_k, f)^2,\\
& \quad\leq D_r-\frac{\gamma_r}{2}+\frac{C}{\delta}
\int_{\de B_{r}}\cG(f_k, f)^2 <D_r-\frac{\gamma_r}{4}.
\end{align*}
This gives the contradiction.
\end{proof}

\section{Estimate of the singular set}\label{s:structure}

In this section we estimate the Hausdorff dimension
of the singular set of $\D$-minimizing $Q$-valued functions
as in Theorem \ref{t:structure}.
The main point of the proof is contained in
Proposition \ref{p:Qset}, estimating 
the size of the set of singular points with multiplicity $Q$.
Theorem \ref{t:structure} follows then by an easy induction argument on $Q$.

\begin{propos}\label{p:Qset}
Let $\Omega$ be connected and $f\in W^{1,2}(\Omega,\Iqs)$ be $\D$-minimizing.
Then, either $f=Q\a{\zeta}$ with
$\zeta:\Om\ra\R{n}$ harmonic in $\Om$,
or the set
\begin{equation*}
\Sigma_{Q,f}=\left\{x\in\Om\,:\, f(x)=Q\a{y},\;y\in\R{n}\right\}
\end{equation*}
(which is relatively closed in $\Om$)
has Hausdorff dimension at most $m-2$ 
and it is locally finite for $m=2$.
\end{propos}

We will make a frequent
use of the function $\sigma:\Om\ra\N$ given by the formula
\begin{equation}\label{e:sigma}
\sigma(x)=\card(\supp f(x)).
\end{equation}
Note that $\sigma$ is lower semicontinuous because $f$ is continuous. 
This implies, in turn, that $\Sigma_{Q,f}$ is closed.

\subsection{Preparatory Lemmas}
We first state and prove two lemmas which will be used in the proof of
Proposition \ref{p:Qset}.
The first reduces Proposition \ref{p:Qset} to the case 
where all points of multiplicity 
$Q$ are of the form $Q\a{0}$. In order to state it,
we introduce the map $\etaa:\Iqs\ra\R{n}$ which takes each
measure $T=\sum_i \a{P_i}$ to its center of mass,
\begin{equation*}
\etaa(T)=\frac{\sum_i P_i}{Q}.
\end{equation*}

\begin{lemma}\label{l:harm+harm}
Let $f:\Om\ra \Iqs$ be $\D$-minimizing.
Then,
\begin{itemize}
\item[$(a)$] the function $\etaa\circ f:\Om\ra\R{n}$
is harmonic;
\item[$(b)$] for every
$\zeta:\Om\ra\R{n}$ harmonic, $g:=\sum_i \a{f_i+\zeta}$
is as well $\D$-minimizing.
\end{itemize}
\end{lemma}

\begin{proof}
The proof of $(a)$ follows from plugging $\psi(x,u)=\zeta (x)\in 
C_c^\infty(\Om,\R{n})$ in the
variations formula \eqref{e:OVbis} of Proposition \ref{p:first}.
Indeed, from the chain-rule \eqref{e:reale},
one infers easily that $Q\,D(\etaa\circ f)=\sum_iDf_i$ and hence, from
\eqref{e:OVbis} we get $\int \langle D(\etaa\circ f):D\zeta\rangle=0$.
The arbitrariness of $\zeta\in C^\infty_c (\Omega, \R{n})$ gives (a).

To show $(b)$, let $h$ be any $Q$-valued function
with $h\vert_{\de \Om}=f\vert_{\de \Om}$: we need to
verify that, if $\tilde{h}:= \sum_i \a{h_i+\zeta}$,
then $\D(g,\Om)\leq\D(\tilde{h},\Om)$.
From Almgren's form of the Dirichlet energy (see \eqref{e:aldir}),
we get
\begin{align}\label{e:harm+harm}
\D(g,\Om) & = \int_\Om \sum_{i,j}|\de_j g_i|^2
= \int_\Om \sum_{i,j}\left\{|\de_j f_i|^2+
|\de_j\zeta|^2+2\,\de_jf_i\,\de_j\zeta\right\}\nonumber\\
& \stackrel{\mathclap{\textrm{min. of }f}}{\leq}\; 
\int_\Om \sum_{i,j}\left\{|\de_j h_i|^2+
|\de_j\zeta|^2\right\}+2\int_\Om D(\etaa\circ f)\cdot D\zeta\nonumber\\
& = \D(\tilde{h},\Om)+2\int_\Om \left\{D(\etaa\circ f)-
D(\etaa\circ h)\right\}\cdot D\zeta.
\end{align}
Since $\etaa\circ f$ and $\etaa\circ h$ have
the same trace on $\de \Om$ and $\zeta$ is
harmonic, the last integral in \eqref{e:harm+harm} vanishes.
\end{proof}

The second lemma characterizes the blow-ups of
homogeneous functions and is the starting point of the 
reduction argument used in the proof of Proposition
\ref{p:Qset}. 

\begin{lemma}[Cylindrical blow-up]\label{l:second blowup}
Let $g:B_1\ra\Iqs$ be an $\alpha$-homogeneous and
$\D$-minimizing function with $\D(g,B_1)>0$
and set $\beta = I_{z,g} (0)$.
Suppose, moreover, that $g(z)=Q\a{0}$ for $z=e_1/2$.
Then, the tangent functions $h$ to $g$ at $z$
are $\beta$-homogeneous with $\D(h,B_1)=1$ and satisfy:
\begin{itemize}
\item[$(a)$] $h(s\,e_1)=Q\a{0}$ for every $s\in\R{}$;
\item[$(b)$] $h(x_1,x_2,\ldots,x_m)=\hat h (x_2,\ldots,x_m)$, where
$\hat h:\R{m-1}\ra\Iqs$ is $\D$-minimizing on any
bounded open subset of $\R{m-1}$.
\end{itemize}
\end{lemma}
\begin{proof}
The first part of the proof follows
from Theorem \ref{t:blowup}, while $(a)$ is straightforward.
We need only to verify $(b)$.
To simplify notations, we pose $x^\prime=(0,x_2,\ldots,x_m)$:
we show that $h(x^\prime)=h(s\,e_1+x^\prime)$ for every $s$
and $x^\prime$.
This is an easy consequence of
the homogeneity of both $g$ and $h$.
Recall that $h$ is the local uniform
limit of $g_{z,\varrho_k}$ for some $\rho_k\downarrow 0$ and
set  $C_k:=\D(g,B_{\varrho_k}(z))^{-1/2}$, $\beta =
I_{z,g} (0)$ and
$\lambda_k:=\frac{1}{1-2\varrho_k\,s}$, where $z= e_1/2$.
Hence, we have
\begin{align}
h(s\,e_1+x^\prime)\quad & \stackrel{\mathclap{\textrm{hom. of }h}}{=}\quad 
\lim_{k\uparrow\infty}C_k\,\frac{g_{z, \varrho_k}
\left(s \lambda_k\,e_1+
\lambda_k x^\prime\right)}
{\lambda_k^{\beta}}
= \lim_{k\uparrow\infty}C_k\,\frac{g
\left(\lambda_k\, z+\lambda_k\,\varrho_k\, x^\prime\right)}
{\lambda_k^{\beta}}\nonumber\\
& \stackrel{\mathclap{\textrm{hom. of }g}}{=}\quad
\lim_{\varrho\ra0}C_k\,\frac{{\lambda_k}^\alpha\,
g_{z,\varrho_k}
\left(x^\prime\right)}
{\lambda_k^{\beta}}
= h(x^\prime),\nonumber
\end{align}
where we used
$\lambda_k z+\lambda_k\,\varrho_k\, x^\prime=
z+s\lambda_k\,\varrho_k\, e_1+\lambda_k\,\varrho_k\, x^\prime$
and $\lim_{k\uparrow\infty}\lambda_k=1$.

The minimizing property of $\hat h$ is a consequence
of the $\D$-minimality of $h$. It suffices to show it
on every ball $B\subset \R{m-1}$ for which $\hat{h}|_{\de B}
\in W^{1,2}$. To fix ideas, assume $B$ to be centered
at $0$ and to have radius $R$. Assume the existence
of a competitor $\tilde{h}\in W^{1,2} (B)$ such that
$\D (\tilde{h}, B)\leq D(\hat{h}, B) - \gamma$
and $\tilde{h}|_{\de B} = \hat{h}|_{\de B}$. 
We now construct a competitor $h'$ for $h$ on a cylinder
$C_L = [-L, L]\times B_R$. First of all we define
\[
h' (x_1, x_2, \ldots, x_n)= \tilde{h} (x_2, \ldots, x_n)
\quad \mbox{for $|x_1|\leq L-1$.}
\]
It remains to ``fill in'' the two cylinders
$C^1_L = ]L-1, L[\times B_R$ and $C^2_L = 
]-L, -(L-1)[\times B_R$.
Let us consider the first cylinder. We need to define
$h'$ in $C^1_L$ in such a way that $h'=h$ on
the lateral surface $]L-1, L[\times \partial B_R$
and on the upper face $\{L\}\times B_R$ and
$h'=\tilde{h}$ on the lower face $\{L-1\}\times B_R$.
Now, since the cylinder $C^1_L$ is biLipschitz
to a unit ball, recalling Corollary \ref{c:fill},
this can be done with a $W^{1,2}$ map. 

Denote by $u$ and $v$ the upper and lower ``filling'' maps
in the case $L=1$  
By the $x_1$-invariance of our construction, the maps
\[
u_L (x_1, \ldots, x_m) := u (x_1-L, \ldots, x_m)
\quad\mbox{and}\qquad 
v_L (x_1, \ldots, x_m) = u (x_1+L, \ldots, x_m)
\] 
can be taken as filling maps for any $L\geq 1$.
Therefore, we can estimate
\begin{align*}
\D (h', C_L) - D(h, C_L)&\leq
\left(\D \left(h', C^1_L\cup C^2_L\right) - \D \left(h, C^1_L\cup C^2_L\right)\right)
- 2\,(L-1)\, \gamma\\
&=: \Lambda - 2\,(L-1)\,\gamma,
\end{align*}
where $\Lambda$ is a constant independent of $L$. 
Therefore, for a sufficiently
large $L$,
we have $D (h', C_L)< D(h, C_L)$ contradicting
the minimality of $h$ in $C_L$.
\end{proof}

\subsection{Proof of Proposition \ref{p:Qset}}
With the help of these two lemmas we conclude the proof of Proposition
\ref{p:Qset}.
First of all we notice that, by Lemma \ref{l:harm+harm},
it suffices to consider $\D$-minimizing function
$f$ such that $\etaa\circ f\equiv 0$.
Under this assumption, it follows that
$\Sigma_{Q,f}=\left\{x\,:\,f(x)=Q\a{0}\right\}$.
Now we divide the proof into two parts, being the
case $m=2$ slightly different from the others.

\medskip

{\bf The planar case $m=2$.} 
We prove that, except for the case where all sheets
collapse, $\Sigma_{Q,f}$ consists of isolated points.
Without loss of generality,
let $0\in\Sigma_{Q,f}$ and assume the existence
of $r_0>0$ such that
$\D(f,B_r)>0$ for every $r\leq r_0$ 
(note that, when we are not in this case, then $f\equiv Q\a{0}$ in a neighborhood
of $0$).
Suppose by contradiction that $0$ is not an isolated point
in $\Sigma_{Q,f}$, i.e. there exist
$x_k\ra0$ such that $f(x_k)=Q\a{0}$.
By Theorem \ref{t:blowup}, the blow-ups $f_{|x_k|}$ converge uniformly,
up to a subsequence, to some homogeneous
$\D$-minimizing function $g$, with
$\D(g,B_1)=1$ and $\etaa\circ g\equiv0$.
Moreover, since $f(x_k)$
are $Q$-multiplicity points,
we deduce that there exists $w\in\s^1$ such that $g(w)=Q\a{0}$.
Up to rotations, we can assume that $w=e_1$.
Considering the blowup of $g$ in the point $e_1/2$,
by Lemma \ref{l:second blowup}, we find
a new tangent function $h$ with the property
that $h (0,x_2)= \hat{h} (x_2)$
for some function $\hat h: \R{}\ra\Iq$ which is
$\D$-minimizing on every interval. Moreover, since
$\D (h, B_1)=1$, clearly
$\D\big(\hat h, I\big)>0$, where $I=[-1,1]$.
Note also that $\etaa\circ\hat h\equiv0$ and $\hat h (0)=Q\a{0}$.
From the $1$-d selection criterion in Proposition \ref{p:Wselection},
this is clearly a contradiction. Indeed, by a simple comparison
argument, it is easily seen that every $\D$-minimizing $1$-d
function $\hat{h}$ is an affine function 
of the form $\hat{h} (x) =\sum_i\a{L_i (x)}$ with
the property that either $L_i (x)\neq L_j (x)$ for every $x$
or $L_i (x)= L_j (x)$ for every $x$. Since $\hat h(0)= Q\a{0}$,
we would conclude that $\hat{h} = Q\a{L}$ for some linear
$L$. On the other hand, by $\etaa\circ\hat h\equiv0$
we would conclude $L=0$, contradicting $\D (\hat{h}, I)>0$.

We conclude that, if $x\in \Sigma_{Q,f}$,
either $x$ is isolated, or $U\subset \Sigma_{Q,f}$
for some neighborhood of $x$. Since $\Om$ is connected, 
we conclude that, either $\Sigma_{Q,f}$ consists of isolated points,
or $\Sigma_{Q,f}=\Om$.

\medskip

{\bf The case $m\geq3$.}
In this case we use the so-called Federer's reduction argument
(following closely the exposition in Appendix A of \cite{Sim}).
We denote by $\cH^t$ the Hausdorff $t$-dimensional
measure and by $\cH^t_\infty$ the Hausdorff
pre-measure defined by
\begin{equation}\label{e:Hpre}
\cH^t_\infty(A)=\inf\left\{\sum_{k\in\N}
\diam(E_k)^t\,:\,A\subset\cup_{k\in\N}E_k\right\}.
\end{equation}
We use this simple property of the Hausdorff
pre-measures $\cH^t_\infty$: if $K_l$ are compact sets
converging to $K$ in the sense of Hausdorff,
then
\begin{equation}\label{e:semicont H}
\limsup_{l\ra+\infty}\cH^t_\infty(K_l)\leq \cH^t_\infty(K).
\end{equation}
To prove \eqref{e:semicont H}, note first
that the infimum on \eqref{e:Hpre} can be taken over open coverings.
Next, given an open covering of $K$, use its compactness
to find a finite subcovering and the convergence
of $K_l$ to conclude that it covers $K_l$ for $l$ large enough
(see the proof of Theorem A.4 in \cite{Sim} for more details).

\medskip

\textit{Step $1$. Let $t>0$. If
$\cH^t_\infty\left(\Sigma_{Q,f}\right)>0$,
then there exists a function $g\in W^{1,2}(B_1,\Iq)$
with the following properties:
\begin{itemize}
\item[$(a_1)$] $g$ is a homogeneous $\D$-minimizing
function with $\D(g,B_1)=1$;
\item[$(b_1)$] $\etaa\circ g\equiv 0$;
\item[$(c_1)$]
$\cH^t_\infty\left(\Sigma_{Q,g}\right)>0$.
\end{itemize}
}

We note that $\cH^t_\infty$-almost every point $x\in\Sigma_{Q,f}$
is a point of positive $t$ density (see Theorem 3.6 in \cite{Sim}), i.e.
\begin{equation*}
\limsup_{r\ra0}
\frac{\cH^t_\infty\left(\Sigma_{Q,f}\cap B_r(x)\right)}
{r^t}>0.
\end{equation*}
So, since $\cH^t_\infty\left(\Sigma_{Q,f}\right)>0$,
from Theorem \ref{t:blowup} we conclude the
existence of a point $x\in\Sigma_{Q,f}$
and a sequence of radii $\varrho_k\ra0$ such that
the blow-ups $f_{x,2\varrho_k}$ converge uniformly
to a function $g$ satisfying $(a_1)$ and $(b_1)$, and
\begin{equation}\label{e:density2}
\limsup_{k\ra+\infty}
\frac{\cH^t_\infty\left(\Sigma_{Q,f}
\cap B_{\varrho_k}(x)\right)}
{{\varrho_k}^t}>0.
\end{equation}
From the uniform convergence of $f_{x,2\varrho_k}$
to $g$, we deduce easily that, up to subsequence,
the compact sets
$K_k=\overline{B_{\frac{1}{2}}}
\cap\Sigma_{Q,f_{x,2\varrho_k}}$
converge in the sense of Hausdorff
to a compact set $K\subseteq\Sigma_{Q,g}$.
So, from the semicontinuity property \eqref{e:semicont H},
we infer $(c_1)$,
\begin{align*}
\cH^t_\infty(\Sigma_{Q,g}) & \geq 
\cH^t_\infty(K)\geq
\limsup_{k\ra+\infty}
\cH^t_\infty(K_k)
\geq \limsup_{k\ra+\infty}
\cH^t_\infty(B_{\frac{1}{2}}\cap\Sigma_{Q,f_{x,2\varrho_k}})\\
&=\limsup_{k\ra+\infty}
\frac{\cH^t_\infty\left(\Sigma_{Q,f}
\cap B_{\varrho_k}(x)\right)}
{{\varrho_k}^t}\stackrel{\eqref{e:density2}}{>}0.
\end{align*}

\medskip

\textit{Step $2$. Let $t>0$ and $g$ satisfying $(a_1)$-$(c_1)$ of Step $1$.
Suppose, moreover, that there exists $1\leq l\leq m-2$,
with $l-1<t$, such that
\begin{equation}\label{e:less variables}
g(x)=\hat g(x_l,\ldots,x_m).
\end{equation}
Then, there exists a function $h\in W^{1,2}(B_1,\Iq)$
with the following properties:
\begin{itemize}
\item[$(a_2)$] $h$ is a homogeneous $\D$-minimizing
function with $\D(h,B_1)=1$;
\item[$(b_2)$] $\etaa\circ h\equiv 0$;
\item[$(c_2)$]
$\cH^t_\infty\left(\Sigma_{Q,h}\right)>0$;
\item[$(d_2)$] $h(x)=\hat h(x_{l+1},\ldots,x_m)$.
\end{itemize}
}

We notice that
$\cH^t_\infty\left(\R{l-1}\times\{0\}\right)=0$, being
$t>l-1$.
So, since $\cH^t_\infty\left(\Sigma_{Q,g}\right)>0$,
we can find a point
$0\neq x=(0,\ldots,0,x_l,\dots,x_m)\in\Sigma_{Q,g}$
of positive density for
$\cH^t_\infty\res\Sigma_{Q,g}$.
By the same argument of Step $1$,
we can blow-up at $x$ obtaining a function
$h$ with properties $(a_2)$, $(b_2)$ and $(c_2)$. Moreover,
using Lemma \ref{l:second blowup}, one immediately
infers $(d_2)$.

\medskip

\textit{Step $3$. Conclusion: Federer's reduction argument.}

Let now $t>m-2$ and suppose
$\cH^t\left(\Sigma_{Q,f}\right)>0$.
Then, up to rotations,
we may apply Step $1$ once and Step $2$ repeatedly
until we end up with a $\D$-minimizing function
$h$ with properties $(a_2)$-$(c_2)$ and
depending only on two variables,
$h(x)=\hat h(x_1,x_2)$.
This implies that
$\hat h$ is a planar $Q$-valued
$\D$-minimizing function such that
$\etaa\circ \hat h\equiv0$, $\D(\hat h,B_1)=1$ and
$\cH^{t-m+2}\left(\Sigma_{Q,\hat h}\right)>0$.
As shown in the proof of the planar case,
this is impossible, since
$t-m+2>0$ and the singularities
are at most countable.
So, we deduce that $\cH^t\left(\Sigma_{Q,f}\right)=0$,
thus concluding the proof.

\subsection{Proof of Theorem \ref{t:structure}}
Let $\sigma$ be as in \eqref{e:sigma}.
It is then clear that,
if $x$ is a regular point, then $\sigma$ is continuous at $x$.

On the other hand, let $x$ be a point of continuity of $\sigma$
and write $f(x) = \sum_{j=1}^J k_j \a{P_j}$, where
$P_i\neq P_j$ for $i\neq j$. 
Since the target of $\sigma$ is discrete, it turns out
that $\sigma\equiv J$ in a neighborhood $U$ of $x$. 
Hence, by the continuity of $f$, in a neighborhood $V\subset U$ of $x$,
there is a continuous decomposition $f=\sum_{j=1}^J \{f_j\}$ in
$k_j$-valued functions, with the property that $f_j (y)\neq f_i (y)$
for every $y\in V$ and $f_j = k_j \a{g_j}$ for each $j$. 
Moreover, it is easy to check that each $g_j$ must necessarily be a harmonic function, so
that $x$ is a regular point for $f$.
Therefore, we conclude
\begin{equation}\label{e:caratsing}
\Sigma_f = \{x: \mbox{$\sigma$ is discontinuous at $x$}\}.
\end{equation}

The continuity of $f$ implies easily the lower semicontinuity
of $\sigma$, which in turn shows, through \eqref{e:caratsing},
that $\Sigma$ is relatively closed.

In order to estimate the Hausdorff dimension of $\Sigma_f$,
we argue by induction on the number of values.
For $Q=1$ there is nothing to prove, since
$\D$-minimizing $\R{n}$-valued functions are
classical harmonic functions.
Next, we assume that the theorem holds for every $Q^*$-valued
functions, with $Q^*<Q$,
and prove it for $Q$-valued functions.
If $f=Q\a{\zeta}$ with
$\zeta$ harmonic, then $\Sigma_f=\emptyset$
and the proposition is proved.
If this is not the case, we consider first $\Sigma_{Q,f}$
the set of points of multiplicity $Q$:
it is a subset of $\Sigma_f$ and we know
from Proposition \ref{p:Qset} that
it is a closed subset of $\Om$
with Hausdorff dimension at most $m-2$ and at most
countable if $m=2$.
Then, we consider the
open set $\Om^\prime=\Om\setminus\Sigma_{Q,f}$.
Thanks
to the continuity of $f$,
we can find countable open balls $B_k$ such that
$\Om^\prime=\cup_k B_k$ and
$f\vert_{B_k}$ can be 
decomposed as the sum of two multiple-valued $\D$-minimizing
functions:
\[
f\vert_{B_k}=\a{f_{k,Q_1}}+\a{f_{k,Q_2}},\quad
\textrm{with $Q_1<Q$, $Q_2<Q$},
\]
and
\[
\supp(f_{k,Q_1}(x))\cap \supp(f_{k,Q_2}(x))=\emptyset
\quad\text{for every }\;x\in B_k.
\]
Clearly, it follows from this last condition
that
\begin{equation*}
\Sigma_f\cap B_k=\Sigma_{f_{k,Q_1}}\cup \Sigma_{f_{k,Q_2}}.
\end{equation*}
Moreover, $f_{k,Q_1}$ and $f_{k,Q_2}$ are both
$\D$-minimizing and, by inductive hypothesis, $\Sigma_{f_{k,Q_1}}$ and
$\Sigma_{f_{k,Q_2}}$ are closed subsets of $B_k$
with Hausdorff dimension at most $m-2$.
We conclude that
\[
\Sigma_f=\Sigma_{Q,f}\cup
\bigcup_{k\in\N} \left(\Sigma_{f_{k,Q_1}}
\cup \Sigma_{f_{k,Q_2}}\right)
\]
has Hausdorff dimension at most $m-2$ and it 
is at most countable if $m=2$.

\chapter{Intrinsic theory}

In this chapter
we develop more systematically the metric theory
of $Q$-valued Sobolev functions.
The aim is to provide
a second proof of all the propositions and
lemmas in Section \ref{s:Qsobolev}, independent of
Almgren's embedding and retraction $\xii$ and $\ro$.
Some of the properties proved in this section
are actually true for Sobolev spaces taking values
in fairly general metric targets, whereas some
others do depend on the specific structure of
$\Iqs$. 

\section{Metric Sobolev spaces}\label{s:metric1}

To our knowledge, metric space-valued Sobolev-type
spaces were considered for the first time by Ambrosio
in \cite{Amb} (in the particular case of $BV$ mappings). 
The same issue was then considered later by
several other authors in connection with different
problems in geometry and analysis (see
for instance \cite{GrSc}, \cite{KaSc1}, \cite{Se}, \cite{Jost},
\cite{JoZu}, \cite{CaLin} and \cite{Hei}).
The definition adopted
here differs slightly from that of
Ambrosio (see Definition \ref{d:W1p}) and was 
proposed later, for general exponents,
by Reshetnyak (see \cite{Res1} and \cite{Res}).
In fact, it turns out that the two points of view are equivalent,
as witnessed by the following proposition.

\begin{propos}\label{p:equivalence def}
Let $\Om\subset \R{n}$ be open and bounded. 
A $Q$-valued function $f$ belongs to $W^{1,p}(\Om,\Iq)$
if and only if there exists a function $\psi\in L^p(\Om,\R{+})$
such that, for every Lipschitz function $\phi:\Iq\ra\R{}$,
the following two conclusions hold:
\begin{itemize}
\item[$(a)$] $\phi\circ f\in W^{1,p}(\Om)$;
\item[$(b)$] $\left|D\left(\phi\circ f\right)(x)\right|\leq\Lip(\phi)\;\psi(x)$
for almost every $x\in \Om$.
\end{itemize}
\end{propos}

This fact was already remarked by Reshetnyak. The proof
relies on the observation that Lipschitz maps
with constant less than $1$ can be written as suprema
of translated distances. This idea, already
used in \cite{Amb}, 
underlies in a certain sense the embedding of separable
metric spaces in $\ell^\infty$, a fact
exploited first in the pioneering work \cite{Grom} by Gromov
(see also the works \cite{AmbKir2}, \cite{AmbKir1}
and \cite{HoKoShTy}, where
this idea has been used in various situations).

\begin{proof}
Since the distance function from a point is a Lipschitz map,
with Lipschitz constant $1$,
one implication is trivial. To prove the opposite,
consider a Sobolev
$Q$-valued function $f$: we claim that (a) and (b) hold
with $\psi = \big(\sum_j\varphi_j^2\big)^{1/2}$,
where the $\varphi_j$'s are the
functions in Definition \ref{d:W1p}.
Indeed, take a Lipschitz function $\phi\in\Lip(\Iq)$.
By treating separately the positive and the negative part
of the function, we can assume, without loss of generality,
that $\phi\geq 0$.
If $\{T_i\}_{i\in\N}\subset\Iq$ is a dense subset
and $L=\Lip(\varphi)$, it is a well known fact that
$\phi (T)=\inf_i \big\{\phi (T_i)+L\,\cG(T_i, T)\big\}$. Therefore,
\begin{equation}\label{e:lip as inf}
\phi\circ f=\inf_{i}\big\{\phi(T_i)
+L\,\cG(T_i,f)\big\}=:\inf_{i}g_i.
\end{equation}
Since $f\in W^{1,p}(\Om,\Iq)$,
each $g_i\in W^{1,p}(\Om)$ and the inequality
$|D (\phi\circ f)|\leq\sup_i |Dg_i|$ holds a.e. 
On the other hand, $|Dg_i|=L\,|D\cG(f,T_i)|\leq 
L\sqrt{\sum_j\varphi_j^2}$ a.e.
This completes the proof.
\end{proof}

In the remaining sections of this chapter, we first prove
the existence of $|\de_jf|$ (as defined in the Introduction)
and prove the explicit formula \eqref{e:def|D_jf|}.
Then, we introduce a metric on $W^{1,p} (\Om, \Iq)$, making 
it a complete metric space. This part of
the theory is in fact valid under fairly general
assumptions on the target space:
the interested reader will find suitable
analogs in the aforementioned papers.

\subsection{Representation
formulas for $|\partial_j f|$.}

\begin{propos}\label{p:|Df|}
For every Sobolev $Q$-valued function $f\in W^{1,p}(\Om,\Iq)$,
there exist $g_j\in L^p$,
for $j=1,\ldots,m$, with the following
two properties:
\begin{itemize}
\item[$(i)$] $\left|\de_j \cG(f,T)\right|\leq g_j$ a.e.
for every $T\in\Iq$;
\item[$(ii)$] if $\varphi_j\in L^p$ is 
such that 
$\left|\de_j \cG(f,T)\right|\leq\varphi_j$ for all $T\in\Iq$,
then $g_j\leq\varphi_j$ a.e.
\end{itemize}
These functions are unique and will be denoted by
$|\partial_j f|$. Moreover, chosen a countable dense subset
$\{T_i\}_{i\in\N}$ of $\Iq$, they satisfy
the equality \eqref{e:def|D_jf|}.
\end{propos}
\begin{proof} The uniqueness of the functions $g_j$ is
an obvious corollary of their property $(ii)$.
It is enough to prove that $g_j = \abs{\de_j f}$
as defined in \eqref{e:def|D_jf|} satisfies $(i)$, because
it obviously satisfies $(ii)$.
Let $T\in\Iq$ and $\{T_{i_k}\}\subseteq \{T_{i}\}$
be such that
$T_{i_k}\ra T$. Then, $\cG(f,T_{i_k})\ra \cG(f,T)$ in $L^p$
and, hence, for every $\psi\in C^\infty_c(\Omega)$,
\begin{equation}\label{e:basta}
\abs{\int \de_j\cG(f,T)\;\psi}=
\lim_{i_k\ra+\infty}\abs{\int \cG(f,T_{i_k})\;\de_j\psi}=
\lim_{i_k\ra+\infty}\abs{\int \de_j\cG(f,T_{i_k})\;\psi}
\leq \int g_j\,|\psi|.
\end{equation}
Since \eqref{e:basta} holds for every $\psi$, we conclude
$|\partial_j \cG (f,T)|\leq g_j$ a.e.
\end{proof}

\subsection{A metric on $W^{1,p}(\Om,\Iq)$}
Given $f$ and $g\in W^{1,p}(\Om,\Iq)$, define
\begin{equation}\label{e:distance}
d_{W^{1,p}}(f,g)=\norm{\cG(f,g)}{L^p}+\sum_{j=1}^m
\norm{\sup_{i}
\big|\de_j\cG(f,T_i)-\de_j\cG(g,T_i)\big|}{L^p}.
\end{equation}

\begin{propos}\label{p:metric space}
$\left(W^{1,p}(\Om,\Iq),d_{W^{1,p}}\right)$
is a complete metric space and
\begin{equation}\label{e:conv energy}
d_{W^{1,p}}(f_k,f)\ra0\quad\Rightarrow\quad|Df_k|
\stackrel{L^p}{\ra}|Df|.
\end{equation}
\end{propos}
\begin{proof}
The proof that $d_{W^{1,p}}$ is
a metric is a simple computation left to the reader;
we prove its completeness.
Let $\{f_k\}_{k\in\N}$ be a Cauchy sequence for
$d_{W^{1,p}}$. Then, it is a Cauchy sequence in
$L^p(\Om,\Iq)$.
There exists, therefore,
a function $f\in L^p(\Om,\Iq)$ such that $f_k\ra f$
in $L^p$.
We claim that $f$ belongs to $W^{1,p}(\Om,\Iq)$
and $d_{W^{1,p}}(f_k,f)\to 0$.
Since $f\in W^{1,p}(\Om,\Iq)$ if and only if $d_{W^{1,p}}(f,0)<\infty$,
it is clear that we need only to prove that $d_{W^{1,p}}(f_k,f)\ra0$.
This is a consequence of the following  simple observation:
\begin{align}\label{e:conv dist}
\norm{\sup_{i}
\big|\de_j\cG(f,T_i)-\de_j\cG(f_k,T_i)\big|}{L^p} & =
\sup_{P\in\mathcal{P}}\sum_{E_s\in P}
\norm{\de_j\cG(f,T_s)-\de_j\cG(f_k,T_s)}{L^p(E_s)}
\nonumber\\
& \leq
\lim_{l\ra+\infty}d_{W^{1,p}}(f_l,f_k),
\end{align}
where $\mathcal{P}$ is the family of finite measurable
partitions of $\Om$.
Indeed, by \eqref{e:conv dist}, 
\begin{equation*}
\lim_{k\ra+\infty}d_{W^{1,p}}(f_k,f)
\stackrel{\eqref{e:conv dist}}{\leq}
\lim_{k\ra+\infty}\left[\norm{\cG(f,f_k)}{L^p}
+m\lim_{l\ra+\infty}d_{W^{1,p}}(f_l,f_k)\right]
=0.
\end{equation*}
We now come to \eqref{e:conv energy}.
Assume $d_{W^{1,p}}(f_k,f)\ra0$ and
observe that
\[
\big||\de_jf_k|-|\de_jf_l|\big|=
\abs{\sup_{i}\abs{\de_j\cG(f_k,T_i)}-
\sup_{i}\abs{\de_j\cG(f_k,T_i)}}\leq
\sup_{i}\abs{\de_j\cG(f_k,T_i)-
\de_j\cG(f_k,T_i)}.
\]
Hence, one can infer
$\big\||\de_jf_k|-\left|\de_jf_l\right|\big\|_{L^p}
\leq d_{W^{1,p}}(f_k,f_l)$. This implies that $|Df_k|$
is a Cauchy sequence, from which the conclusion follows easily.
\end{proof}

\section{Metric proofs of the main theorems I}\label{s:metric2}

We start now with the metric 
proofs of the results in Section \ref{s:Qsobolev}.

\subsection{Lipschitz approximation}
In this subsection we prove a strengthened version
of Proposition \ref{p:lipapprox}. The proof uses,
in the metric framework, a standard truncation
technique and the Lipschitz extension Theorem \ref{thm.ext}
(see, for instance, $6.6.3$ in \cite{EG}). This last
ingredient is a feature of $\Iqs$ and, in general,
the problem of whether or not general Sobolev mappings
can be approximated with Lipschitz ones 
is a very subtle issue already when the target is a smooth
Riemannian manifold (see for instance \cite{ScUl}, \cite{Be},
\cite{HaLin} and \cite{HaRi}).
The truncation technique is, instead, valid in a much more
general setting, see for instance \cite{HoKoShTy}. 

\begin{propos}[Lipschitz approximation]\label{p:lipapprox metric}
There exists a constant $C=C(m,\Om,Q)$ with the
following property. For every $f\in W^{1,p}(\Om,\Iq)$
and every $\lambda>0$, there exists a $Q$-function
$f_\lambda$ such that ${\rm Lip}\, (f_\lambda)\leq C \,\lambda$,
\begin{equation}\label{e:approx2}
|E_\lambda| = \abs{\big\{x\in\Om\,:\,f(x)\neq
f_\lambda(x)\big\}}\leq \frac{C\||Df|\|^p_{L^p}}{\lambda^p}
\end{equation}
and $d_{W^{1,p}} (f, f_\lambda)\leq C d_{W^{1,p}} (f, Q\a{0})$.
Moreover, $d_{W^{1,p}}(f,f_\lambda) = o(1)$ and 
$|E_\lambda|=o (\lambda^{-p})$.
\end{propos}
\begin{proof}
We consider the case $1\leq p<\infty$
($p=\infty$ is immediate) and we set
\begin{equation*}
\Om_\lambda=\big\{x\in\Om\,:\,
M(\abs{Df})\leq\lambda\big\},
\end{equation*}
where $M$ is the Maximal Function Operator 
(see \cite{St} for the definition). By rescaling,
we can assume $\||Df|\|_{L^p} = 1$. As a consequence,
we can also assume $\lambda\geq C (m,\Om, Q)$,
where $C (m, \Om, Q)$ will be chosen later.

Notice that, for every $T\in\Iq$ and every $j\in\{1,\ldots,m\}$,
\[
M\big(\abs{\de_j\cG(f,T)}\big)\leq M(\abs{Df})\leq\lambda\quad\textrm{in }\Om_\lambda.
\]
By standard calculation (see, for example, $6.6.3$ in \cite{EG}),
we deduce that, for every $T$, $\cG(f,T)$ is
$(C\,\lambda)$-Lipschitz in $\Om_\lambda$,
with $C=C(m)$. Therefore,
\begin{equation}\label{e:truncation lip}
\big|\cG(f(x),T)-\cG(f(y),T)\big|\leq C\,\lambda\abs{x-y}\quad
\forall\; x,y\in\Om_\lambda \;\mbox{and}\; \forall\; T\in\Iq.
\end{equation}
From \eqref{e:truncation lip}, we get a Lipschitz
estimate for $f|_{\Omega_\lambda}$ by setting $T=f(x)$.
We can therefore use Theorem \ref{thm.ext}
to extend $f|_{\Omega_\lambda}$ to a Lipschitz function
$f_\lambda$ with $\Lip (f_\lambda)\leq C \lambda$.

The standard weak $(p-p)$ estimate for maximal functions 
(see \cite{St}) yields
\begin{equation}\label{e:estimate with M}
|\Om\setminus\Om_\lambda| 
\leq \frac{C}{\lambda^p}\int_{\Om\setminus\Om_{\lambda/2}}
|Df|^p\leq\frac{C}{\lambda^p}\;o(1),
\end{equation}
which implies \eqref{e:approx2} and $|E_\lambda|=o(\lambda^{-p})$.
Observe also that, from \eqref{e:estimate with M}, it follows that
\begin{equation}\label{e:aggiunta}
\int_{\Omega\setminus \Omega_\lambda}
|Df_\lambda|^p \leq C \int_{\Omega\setminus \Omega_{\lambda/2}}
|Df|^p.
\end{equation}

It remains to prove $d_{W^{1,p}} 
(f, f_\lambda)\leq C d_{W^{1,p}} (f, Q\a{0})$ and 
$d_{W^{1,p}} (f_\lambda, f)\to 0$. By \eqref{e:aggiunta},
it suffices to show
\begin{equation*}
\|\cG (f_\lambda, Q\a{0})\|_{L^p}\leq C d_{W^{1,p}} 
(f, Q\a{0})
\quad\mbox{and}\quad \|\cG (f_\lambda, f)\|_{L^p} \to 0\, .
\end{equation*}
We first choose the constant
$C(m, \Om, Q)\leq \lambda$ so to guarantee that 
$2 |\Omega_\lambda|\geq |\Omega|$. 
Set $g:= \cG (f, Q\a{0})$, 
$g_\lambda := \cG (f_\lambda, Q\a{0})$ and
$h= g-g_\lambda$. Let $\bar{h}$ be the 
average of $h$ over $\Omega$ and use the Poincar\'e inequality
and the fact that $h$ vanishes on $\Omega_\lambda$ to conclude that
\begin{equation*}
\frac{|\Omega|}{2} |\bar{h}|^p 
\leq |\Omega_\lambda| |\bar{h}|^p
\leq \int |h-\bar{h}|^p
\leq C \|D h\|_{L^p}^p
\leq C \hspace{-0.1cm}\int_{\Omega\setminus \Omega_\lambda}
\left(|D f|^p + |Df_\lambda|^p\right)
\leq C\hspace{-0.1cm}\int_{\Omega\setminus \Omega_{\lambda/2}} |Df|^p\, .
\end{equation*}
Therefore,
\begin{equation*}
\|h\|_{L^p}^p\leq
C\int_{\Omega\setminus \Omega_{\lambda/2}} |Df|^p\, .
\end{equation*}
So, using the triangle inequality, we conclude that
\begin{equation*}
\|\cG (f_\lambda, Q\a{0})\|_{L^p}
\leq \|\cG (f, Q\a{0})\|_{L^p} + C \||Df|\|_{L^p}
\leq C d_{W^{1,p}} (f, Q\a{0})
\end{equation*}
and
\begin{align}
\|\cG (f, f_\lambda)\|)_{L^p}
&= \|\cG (f, Q\a{0})\|_{L^p (\Omega\setminus
\Omega_\lambda)} + \|h\|_{L^p}\nonumber\\
&\leq \|\cG (f, Q\a{0})\|_{L^p (\Omega\setminus
\Omega_\lambda)} +
C \||Df|\|_{L^p (\Omega\setminus \Omega_{\lambda/2})}.
\label{e:aggiunta6}
\end{align}
Since $|\Omega\setminus \Omega_\lambda|\downarrow 0$,
the right hand side of \eqref{e:aggiunta6} converges to
$0$ as $\lambda\downarrow 0$.
\end{proof}

\subsection{Trace theory}
Next, we show the existence of the trace of a $Q$-valued
Sobolev function as defined in Definition \ref{d:Dirichlet problem}.
Moreover, we prove that the space of functions with given trace
$W^{1,p}_g(\Om,\Iq)$ defined in \eqref{e:trace g} is closed under weak convergence. A suitable trace theory can be build in 
a much more general setting (see the aforementioned papers). Here,
instead, we prefer to take advantage of Proposition 
\ref{p:lipapprox metric} to give a fairly short proof.

\begin{propos}\label{p:trace metric}
Let $f\in W^{1,p}(\Om,\Iq)$. Then, there exists an unique
$g\in L^{p}(\de\Om,\Iq)$
such that
\begin{equation}\label{e:def trace}
(\varphi\circ f)|_{\partial \Om}
=\varphi \circ g\qquad\textrm{for all}
\quad \varphi\in{\rm Lip}\,(\Iq).
\end{equation}
We denote $g$ by $f\vert_{\de\Om}$. Moreover, the following
set is closed under weak convergence:
\begin{equation*}
W^{1,2}_g(\Omega, \Iq) :=
\big\{ f\in W^{1,2} (\Omega, \Iq)\,:\,
f|_{\partial \Omega}= g\big\}.
\end{equation*}
\end{propos}
\begin{proof}
Consider a sequence of Lipschitz functions $f_k$ with
$d_{W^{1,p}}(f_k,f)\ra0$ (whose existence is ensured from
Proposition \ref{p:lipapprox metric}).
We claim that $f_k\vert_{\de\Om}$ 
is a Cauchy sequence in
$L^p(\de\Om,\Iq)$.
To see this, notice that, if $\{T_i\}_{i\in\N}$ is a dense subset of $\Iq$,
\begin{equation*}
\cG(f_k,f_l)=\sup_{i} \abs{\cG(f_k,T_i)-\cG(f_l,T_i)}.
\end{equation*}
Moreover, recalling the classical estimate for the trace of a real-valued
Sobolev functions,
$\norm{f\vert_{\de\Om}}{L^p}\leq C\norm{f}{W^{1,p}}$,
we conclude that
\begin{align}\label{e:difference cauchy}
\norm{\cG(f_k,f_l)}{L^p(\de\Om)}^p & \leq
C \int_{\Om}\cG(f_k,f_l)^p+
\sum_j\int_{\Om}|\de_j\cG(f_k,f_l)|^p\nonumber\\
& \leq C\int_{\Om}\cG(f_k,f_l)^p+\sum_j\int_{\Om}\sup_{i}
\abs{\de_j\cG(f_k,T_i)-\de_j\cG(f_l,T_i)}^p\nonumber\\
& \leq C\;d_{W^{1,p}}(f_k,f_l)^p,
\end{align}
(where we used the identity
$\left|\de_j\left(\sup_i g_i\right)\right|\leq\sup_i|\de_j g_i|$,
which holds true if there exists an $h\in L^p (\Omega)$ 
with $|g_i|,\,|Dg_i|\leq h\in L^p(\Om)$).

Let, therefore, $g$ be the $L^p$-limit of $f_k$.
For every $\varphi\in\Lip(\Iq)$, we clearly have that
$(\varphi\circ f_k)\vert_{\de\Om} \to \varphi \circ g$ in $L^p$.
But, since $\varphi\circ f_k\ra \varphi\circ f$ in $W^{1,p}(\Om)$,
the limit of $(\varphi\circ f_k)\vert_{\de\Om}$ is exactly
$(\varphi\circ f)\vert_{\de\Om}$. This shows \eqref{e:def trace}.
We now come to the uniqueness. Assume that $g$
and $\hat g$ satisfy \eqref{e:def trace}. Then, $\cG(g,T_i)=\cG\left(\hat g,T_i\right)$
almost everywhere on $\de\Om$ and for every $i$. This implies
\[
\cG\left(g,\hat g\right)=\sup_i\left|\cG(g,T_i)-\cG\left(\hat g,T_i\right)\right|=0
\quad \textrm{a.e. on }\; \Om,
\]
i.e. $g=\hat g$ a.e.

Finally, as for the last assertion of the proposition,
note that
$f_k\weak f$ in the sense of Definition \ref{d:weak convergence}
if and only if $\varphi\circ f_k\weak \varphi \circ f$ 
for any Lipschitz function $\varphi$. Therefore, the proof that
the set $W^{1,2}_g$ is closed is a direct consequence of the
corresponding fact for classical Sobolev spaces of
real-valued functions.
\end{proof}

\subsection{Sobolev embeddings} The following
proposition is an obvious consequence of
the definition and holds under much more general
assumptions. 

\begin{propos}[Sobolev Embeddings]\label{p:Sembeddings metric}
The following embeddings hold:
\begin{itemize}
\item[$(i)$] if $p<m$, then $W^{1,p}(\Om,\Iq)
\subset L^{q}(\Om,\Iq)$ for
every $q\in [1,p^*]$, where $p^*=\frac{m\,p}{m-p}$,
and the inclusion is compact when $q<p^*$;
\item[$(ii)$] if $p=m$, then $W^{1,p}(\Om,\Iq)
\subset L^{q}(\Om,\Iq)$, for
every $q\in [1,+\infty)$, with compact inclusion. 
\end{itemize}
\end{propos}
\begin{remark} In Proposition \ref{p:Sembeddings}
we have also shown that
\begin{itemize}
\item[$(iii)$] if $p>m$, then $W^{1,p}(\Om,\Iq)
\subset C^{0,\alpha}(\Om,\Iq)$, for
$\alpha=1-\frac{m}{p}$, with compact inclusion.
\end{itemize}
It is not difficult to give an intrinsic proof of
it. However, in the regularity theory of Chapters 3 and
5, (iii) is used only in the case $m=1$, 
which has already been shown 
in Proposition \ref{p:Wselection-1}. 
\end{remark}
\begin{proof}
Recall that
$f\in L^p(\Om,\Iq)$ if and only if $\cG(f,T)\in L^p(\Om)$
for some (and, hence, any) $T$.
So, the inclusions in $(i)$ and $(ii)$ 
are a trivial corollary of the usual
Sobolev embeddings for real-valued functions, which
in fact yields the inequality 
\begin{equation}\label{e:aggiunta10}
\|\cG (f, Q\a{0})\|_{L^q (\Omega)}
\leq C (n, \Omega, Q) d_{W^{1,p}} (f, Q\a{0})\, .
\end{equation}

As for the compactness of the embeddings when $q<p^*$,
consider a sequence $\{f_k\}_{k\in\N}$ of $Q$-valued
Sobolev functions with equibounded 
$d_{W^{1,p}}$-distance from a point:
\begin{equation*}
d_{W^{1,p}}(f_k,Q\a{0})=\norm{\cG(f_k,Q\a{0})}{L^p}+
\sum_j\norm{|\de_jf_k|}{L^p}\leq C<+\infty.
\end{equation*}
For every $l\in \N$, let $f_{k,l}$ be the
function given by Proposition \ref{p:lipapprox metric}
choosing $\lambda = l$.

From the Ascoli--Arzel\`a Theorem and
a diagonal argument, we find a subsequence (not relabelled) $f_{k}$
such that, for any fixed $l$, $\{f_{k,l}\}_k$ is a 
Cauchy sequence in $C^0$.
We now use this to show that $f_{k}$ is a Cauchy sequence
in $L^q$. Indeed, 
\begin{equation}\label{e:emb compact}
\norm{\cG(f_{k},f_{k^\prime})}{L^q}\leq
\norm{\cG(f_{k},f_{k,l})}{L^q}+
\norm{\cG(f_{k,l},f_{k^\prime,l})}{L^q}+
\norm{\cG(f_{k^\prime,l},f_{k^\prime})}{L^q}.
\end{equation}
We claim that the first and third terms are
bounded by $C \, l^{1/q-1/p^*}$. It suffices to show it
for the first term. By Proposition \ref{p:lipapprox metric},
there is a constant $C$ such that $d_{W^{1,p}} (f_{k,l}, Q\a{0})
\leq C$ for every $k$ and $l$.
Therefore, we infer
\begin{align*}
\|\cG(f_{k},f_{k,l})\|_{L^q}^q
&\leq C\int_{\{f_k\neq f_{k,l}\}}
\big[\cG (f_k, Q\a{0})^q + \cG (f_{k,l}, Q\a{0})^q\big]\\
&\leq \Big(\|\cG (f_k, \a{0})\|_{L^{p^*}}^q
+ \|\cG (f_{k,l}, \a{0})\|_{L^{p^*}}^q\Big) 
|\{f_k\neq f_{k,l}\}|^{1-q/p^*} \leq C l^{q/p^*-1},
\end{align*}
where in the last line we have used
\eqref{e:aggiunta10} (in the critical case $p^*$)
and the H\"older inequality.

Let $\eps$ be a given positive number. Then
we can choose $l$ such that 
the first and third term in
\eqref{e:emb compact} are both less than $\eps/3$, independently
of $k$. On the other hand, since $\{f_{k,l}\}_k$ is
a Cauchy sequence in $C^0$, there is an $N$ such that
$\|\cG (f_{k,l}, f_{k',l})\|_{L^q}\leq \eps/3$ for
every $k,k'>N$. Clearly, for $k,k'>N$, we then have
$\|\cG(f_k, f_{k'})\|\leq\eps$. This shows that $\{f_k\}$
is a Cauchy sequence in $L^q$ and hence completes the proof of $(i)$.
The compact inclusion in $(ii)$ is analogous.
\end{proof}

\subsection{Campanato--Morrey estimate}
We conclude this section by giving another proof of the 
Campanato--Morrey estimate in
Proposition \ref{p:Campanato Morrey}.

\begin{propos}\label{p:Campanato Morrey metric}
Let $f\in W^{1,2}(B_1,\Iq)$ and $\alpha\in(0,1]$ be such that
\begin{equation*}
\int_{B_r}|Df|^2
\leq A\; r^{m-2+2\alpha}\quad
\mbox{for a.e. $r\in ]0,1]$.}
\end{equation*}
Then, for every $0<\delta<1$,
there is a constant $C=C(m,n,Q,\delta)$ such that
\begin{equation}\label{e:Campanato Morrey metric}
\sup_{x,y\in
\overline{B_{\delta}}}\frac{\cG(f(x),f(y))}{\abs{x-y}^\alpha}
=: \left[f\right]_{C^{0,\alpha}(\overline{B_{\delta}})}
\leq C\,\sqrt{A}.
\end{equation}
\end{propos}
\begin{proof} Let $T\in\Iq$ be given. Then,
\begin{equation*}
\int_{B_r}|D\cG(f,T)|^2 \leq \int_{B_r}|Df|^2
\leq A\;r^{m-2+2\alpha} \quad
\mbox{for a.e. $r\in ]0,1]$}.
\end{equation*}
By the classical estimate (see $3.2$ in \cite{HanLin}), $\cG(f,T)$ is
$\alpha$-H\"older with
\begin{equation*}
\sup_{x,y\in
\overline{B_{\delta}}}\frac{\abs{\cG(f(x),T)-\cG(f(y),T)}}
{\abs{x-y}^\alpha}
\leq C\,\sqrt{A},
\end{equation*}
where $C$ is independent of $T$.
This implies easily \eqref{e:Campanato Morrey metric}.
\end{proof}

\section{Metric proofs of the main theorems II}
\label{s:metric3}

We give in this section metric proofs of 
the two remaining results of Section \ref{s:Qsobolev}:
the Poincar\'e inequality
in Proposition \ref{p:poincare'} and the
interpolation Lemma \ref{l:technical}.

\subsection{Poincar\'e inequality} 

\begin{propos}[Poincar\'e inequality]\label{p:poincare' metric}
Let $M$ be a connected bounded Lipschitz 
open set of a Riemannian manifold.
Then, for every $1\leq p<m$, there exists 
a constant $C=C(p,m,n,Q,M)$
with the following property:
for every function $f\in W^{1,p}(M,\Iq)$,
there exists a point $\overline f\in \Iq$ such that
\begin{equation}\label{e:poincare' metric}
\left(\int_{M} \cG (f, \overline f)^{p^*}\right)^{\frac{1}{p^*}}
\leq C \left(\int_{M}|Df|^{p}\right)^{\frac{1}{p}},
\end{equation}
where $p^*=\frac{m\,p}{m-p}$.
\end{propos}

A proof of (a variant of) this Poincar\'e-type inequality 
appears already, 
for the case $p=1$ and a fairly general target,
in the work of Ambrosio \cite{Amb}. Here we use, however, a 
different approach, based on the existence
of an isometric embedding of 
$\Iqs$ into a separable Banach space. We then exploit 
the linear structure of this larger space to take averages.
This idea, which to our knowledge appeared first 
in \cite{HoKoShTy}, works in a much more general framework,
but, to keep our presentation easy, we will
use all the structural advantages of dealing with the
metric space $\Iqs$.

The key ingredients of the proof are
the lemmas stated below. The first one
is an elementary fact, exploited first by Gromov
in the context of metric geometry (see \cite{Grom})
and used later to tackle many problems
in analysis and geometry on metric spaces
(see \cite{AmbKir2}, \cite{AmbKir1} and \cite{HoKoShTy}).
The second is an extension of a 
standard estimate in the theory of Sobolev spaces.
Both lemmas will be proved at the end of
the subsection.

\begin{lemma}\label{l:emb}
Let $(X,d)$ be a complete separable metric space.
Then, there is an isometric embedding $i:X\to B$
into a separable Banach space. 
\end{lemma}

\begin{lemma}\label{l:lemma poincare}
For every $1\leq p<m$ and $r>0$,
there exists a constant $C=C(p,m,n,Q)$ such that,
for every $f\in W^{1,p}(B_r,\Iq)\cap {\rm Lip}\,(B_r,\Iq)$
and every $z\in B_r$,
\begin{equation}\label{e:poinc1}
\int_{B_r}\cG(f(x),f(z))^p dx\leq C\,r^{p+m-1}\int_{B_r}
|Df|(x)^p\,|x-z|^{1-m}\,dx.
\end{equation}
\end{lemma}

\begin{proof}[Proof of Proposition \ref{p:poincare' metric}]
\textit{Step $1$.}
We first assume $M=B_r\subset\R{m}$ and $f$ Lipschitz. 
We regard $f$ as a map taking values in
the Banach space $B$ of Lemma \ref{l:emb}.
Since $B$ is a Banach space, we can integrate 
$B$-valued functions on Riemannian manifolds using 
the Bochner integral. Indeed,
being $f$ Lipschitz and $B$ a separable Banach space,
in our case it is straightforward to check that $f$ is integrable
in the sense  of Bochner (see 
\cite{DieUhl}; in fact the theory of the Bochner 
integral can be applied in much more general situations).

Consider therefore the average of $f$ on $M$, which
we denote by $S_f$. We will show that
\begin{equation}\label{e:poincare in A}
\int_{B_r} \|f-S_f\|_B^{p}
\leq C r^p \int_{B_r}|Df|^{p}.
\end{equation}
First note that, by the usual
convexity of the Bochner integral,
\begin{equation*}
\|f(x)-S_f\|_B \leq 
\mint \| f(z) -  f(x)\|_B\, dz
= \mint \cG (f(z), f(x))\, dz.
\end{equation*}
Hence, \eqref{e:poincare in A}
is a direct consequence of Lemma \ref{l:lemma poincare}: 
\begin{align*}
\int_{B_r} \|f(x)-S_f\|^{p}_B\,dx
&\leq  \int_{B_r} \mint_{B_r} \cG(f(x),f(z))^{p}\, dz\,dx
\nonumber\\
& \leq  C\,r^{p+m-1}\mint_{B_r}
\int_{B_r} |w-z|^{1-m}|Df|(w)^p \,dw\,dz\nonumber\\
&\leq  C\,r^{p}\int_{B_r}|Df|(w)^p\, dw.
\end{align*}

\textit{Step $2$.}
Assuming $M=B_r\subset\R{m}$ and $f$ Lipschitz,
we find a point $\overline f$ such that
\begin{equation}\label{e:p-p lip B}
\int_{B_r}\cG\left(f,\overline f\right)^p\leq C r^p \int_{B_r}|Df|^p.
\end{equation}
Consider, indeed, $\overline f\in\Iq$ 
a point such that
\begin{equation}\label{e:projection}
\|S_f-\overline f\|_B=\min_{T\in\Iq}\|S_f-T\|_B.
\end{equation}
Note that $\overline{f}$ exists because $\Iq$ is 
locally compact.
Then, we have
\begin{align*}
\int_{B_r} \cG \left(f, \overline f\right)^{p}
& \leq C \int_{B_r} \|f-S_f\|_B^{p}+
\int_{B_r} \|S_f- \overline f\|_B^{p}\nonumber\\
&\stackrel{\mathclap{\eqref{e:poincare in A},\,\eqref{e:projection}}}{\leq}\quad 
C\, r^p \int_{B_r}|Df|^{p}
+ C\int_{B_r} \|S_f- f\|_B^{p}
\stackrel{\eqref{e:poincare in A}}{\leq} 
C\, r^p \int_{B_r}|Df|^{p}.
\end{align*}

\textit{Step $3$.} Now we consider the case of a generic
$f\in W^{1,p}(B_r,\Iq)$.
From the Lipschitz approximation Theorem
\ref{p:lipapprox metric}, we find a sequence of Lipschitz
functions $f_k$ converging to $f$, $d_{W^{1,p}}(f_k,f)\ra0$.
Fix, now, an index $k$ such that
\begin{equation}\label{e:position}
\int_{B_r}\cG(f_k,f)^p\leq r^p \int_{B_r}|Df|^p
\quad\textrm{and}\quad
\int_{B_r}|Df_k|^p\leq 2\int_{B_r}|Df|^p,
\end{equation}
and set $\overline f=\overline{f_k}$, with the $\overline{f_k}$
found in the previous step.
With this choice, we conclude
\begin{equation}\label{e:p-p B}
\int_{B_r}\cG\left(f,\overline f\right)^p \leq
C \int_{B_r}\cG(f,f_k)^p+
\int_{B_r}\cG\left(f_k,\overline f_k\right)^p\,
\stackrel{\eqref{e:p-p lip B},\,\eqref{e:position}}{\leq} C\,
r^p \int_{B_r}|Df|^p.
\end{equation}

\textit{Step $4$.} Using classical Sobolev embeddings,
we prove \eqref{e:poincare' metric} in the case of
$M=B_r$.
Indeed, since $\cG(f,\overline f)\in W^{1,p}(B_r)$,
we conclude
\begin{equation*}
\norm{\cG(f,\overline f)}{L^{p^*}} \leq
\norm{\cG(f,\overline f)}{W^{1,p}}
\stackrel{\eqref{e:p-p B}}{\leq} C
\left(\int_{B_r}|Df|^p
\right)^{\frac{1}{p}}.
\end{equation*}

\textit{Step $5$.}
Finally, we drop the hypothesis of $M$ being a ball.
Using the compactness and connectedness 
of $\overline M$, we cover $M$ by finitely many domains
$A_1, \ldots, A_N$ 
biLipschitz to a ball such that $A_k\cap \cup_{i<k} A_i
\neq \emptyset$.
This reduces the proof of the general statement to
that in the case $M=A\cup B$, where $A$ and $B$ are
two domains such that
$A\cap B\neq\emptyset$ and the Poincar\'e
inequality is valid for both.
Under these assumptions, denoting by $f_A$ and $f_B$ two means
for $f$ over $A$ and $B$, we estimate
\begin{equation*}
\cG(f_A,f_B)^{p^*}=\mint_{A\cap B}\cG(f_A,f_B)^{p^*}
\leq C \mint_{A}\cG(f_A,f)^{p^*}+
C\mint_{B}\cG(f,f_B)^{p^*}\leq C
\left(\int_M |Df|^p\right)^{\frac{p^*}{p}}.
\end{equation*}
Therefore,
\begin{align*}
\int_{A\cup B}\cG(f,f_A)^{p^*} & \leq
\int_{A}\cG(f,f_A)^{p^*}+
\int_{B}\cG(f,f_A)^{p^*}\\
& \leq \int_{A}\cG(f,f_{A})^{p^*}
+C \int_{B}\cG(f,f_{B})^{p^*}
+C\, \cG (f_A, f_B)^{p^*} |B|
\leq C \left(\int_M
|Df|^p\right)^{\frac{p^*}{p}}.
\end{align*}
\end{proof}

\begin{proof}[Proof of Lemma \ref{l:emb}] We choose a point $x\in X$
and consider the Banach space $A:= \{f\in \Lip (X, \R{}):
f(x)=0\}$ with the norm $\|f\|_A= \Lip (f)$. Consider
the dual $A'$ and let $i:X\to A'$ be the mapping
that to each $y\in X$ associates the element $[y]\in A'$
given by the linear functional $[y] (f) = f (y)$. First
of all we claim that $i$ is an isometry, which amounts
to prove the following identity:
\begin{equation}\label{e:isom}
d(z,y) = \|[y]-[z]\|_{A'} =
\sup_{f (x)=0,\, \Lip (f)\leq 1} |f (y) - f(z)|\, 
\quad \forall x,y\in X.
\end{equation}
The inequality $|f(y)-f(z)|\leq d(y,z)$ follows
from the fact that $\Lip (f)=1$. On the other hand,
consider the function $f(w):= d (w,y)-d(y,x)$.
Then $f(x)=0$, $\Lip (f)=1$ and $|f(y)-f(z)|= d (y,z)$.

Next, let $C$ be the subspace generated by finite linear
combinations of elements of $i (X)$. Note that $C$ is separable
and contains $i(X)$: its closure in $A'$ is
the desired separable Banach space $B$.
\end{proof}

\begin{proof}[Proof of Lemma \ref{l:lemma poincare}]
Fix $z\in B_r$. Clearly the restriction of $f$
to any segment $[x,z]$ is Lipschitz. Using Rademacher,
it is easy to justify the following inequality
for a.e. $x$:
\begin{equation}\label{e:dist-grad1}
\cG(f(x),f(z))\leq|x-z|\int_0^1|Df|(z+t(x-z))\,dt.
\end{equation}
Hence, one has
\begin{align}\label{e:dist-grad2}
\int_{B_r\cap \de B_s(z)}\cG(f(x),f(z))^p\,dx
& \stackrel{\mathclap{\eqref{e:dist-grad1}}}{\leq} 
\int_{B_r\cap\de B_s(z)}\int_0^1|x-z|^p\,|Df|(z+t(x-z))^p\,dt\,dx
\nonumber\\
& \leq  s^p\int_0^1\int_{B_r\cap \de B_{ts}(z)}
t^{1-n}|Df|(w)^p\, dw\, dt\nonumber\\
& = s^{p+m-1}\int_0^1\int_{B_r\cap \de B_{ts}(z)}
|w-z|^{1-m}|Df|(w)^p\, dw\, dt\nonumber\\
& \leq s^{p+m-2}\int_{B_r}
|w-z|^{1-m}|Df|(w)^p\, dw.
\end{align}
Integrating in $s$ the inequality \eqref{e:dist-grad2},
we conclude \eqref{e:poinc1},
\begin{equation*}
\int_{B_r}\cG(f(x),f(z))^p\, dx\leq
C\,r^{p+m-1}\int_{B_r}
|w-z|^{1-m}|Df|(w)^p\, dw.
\end{equation*}
\end{proof}

\subsection{Interpolation Lemma}
We prove in this section Lemma \ref{l:technical} (the statement
below is, in fact, slightly simpler: Lemma \ref{l:technical}
follows however from elementary scaling arguments).
In this case, the proof relies in an essential way
on the properties of $\Iqs$ and we believe that
generalizations are possible only under some structural
assumptions on the metric target.

\begin{lemma}[Interpolation Lemma]\label{l:technical metric}
There exists a constant $C=C(m,n,Q)$ with the following property.
For any $g,\,\tilde g\in W^{1,2}(\de B_1,\Iq)$,
there is 
$h\in W^{1,2}(B_1\setminus B_{1-\eps},\Iq)$
such that
\begin{equation*}
h\left(x\right)=g(x),\qquad
h\left((1-\eps)\,x\right)=\tilde g(x),\quad\textrm{for }x\in\de B_1,
\end{equation*}
and
\begin{equation*}
\D(h,B_1\setminus B_{1-\eps})
\leq C \left\{\eps \,\D(g,\de B_1)+\eps\,\D(\tilde g,\de B_{1})+
\eps^{-1}\int_{\de B_1}
\cG\left(g,\tilde g\right)^2 \right\}\, .
\end{equation*}
\end{lemma}
\begin{proof}
For the sake of clarity, we divide the proof into two steps:
in the first one we prove the lemma in a simplified geometry 
(two parallel hyperplanes instead of two
concentric spheres); then, we adapt the 
construction to the case of interest.

\textit{Step 1. Interpolation between parallel planes.} We let $A=[-1,1]^{m-1}$,
$B= A\times [0,\eps]$ and 
consider two functions $g,\,\tilde g\in W^{1,2}(A,\Iq)$. 
We then want to find a function $h:B\ra\Iq$ such that
\begin{equation}\label{e:marginals}
h(x,0)=g(x)\quad\textrm{and}\quad h(x,\eps)=\tilde g(x);
\end{equation}
\begin{equation}\label{e:estimate energy}
\D(h,B)\leq C\left(\eps\,\D(g,A)+\eps\,\D(\tilde g,A)+
\eps^{-1}\int_A\cG(g,\tilde g)^2\right),
\end{equation}
where the constant $C$ depends only on $m$, $n$ and $Q$.

For every $k\in\N_+$, set $A_k=[-1-k^{-1},1+k^{-1}]^{m-1}$, 
and decompose $A_k$ in the union of
$(k+1)^{m-1}$ cubes
$\left\{C_{k,l}\right\}_{l=1,\ldots,(k+1)^{m-1}}$ with
disjoint interiors, side
length equal to $2/k$ and faces parallel to the
coordinate hyperplanes.
We denote by $x_{k,l}$ their centers. Therefore,
$C_{k,l}=x_{k,l}+\left[-\frac{1}{k},\frac{1}{k}\right]^{m-1}$.
Finally, we subdivide $A$ into the cubes $\{D_{k,l}\}_{l=1, \ldots,
k^{m-1}}$ of side $2/k$ and having the points $x_{k,l}$ as vertices, 
(so $\{D_{k,l}\}$ is the decomposition ``dual'' to $\{C_{k,l}\}$;
see Figure \ref{f:dual}).

\begin{figure}[htbp]
\begin{center}
    \input{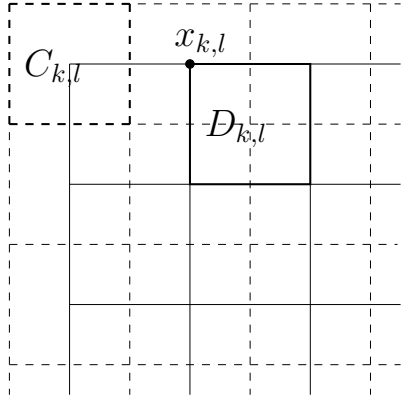}
    \caption{The cubes $C_{k,l}$ and $D_{k,l}$.}
    \label{f:dual}
\end{center}
\end{figure}

On each $C_{k,l}$ take a mean $\overline{g}_{k,l}$ of $g$
on $C_{k,l}\cap A$. On $A_k$
we define the piecewise constant functions $g_k$
which takes the constant value $\overline{g}_{k,l}$ on each $C_{k,l}$: 
\[
g_k\equiv \overline g_{k,l}\quad\textrm{in }
\; C_{k,l}, \quad\textrm{with}\quad
\int_{C_{k,l}\cap A}\cG(g,\overline g_{k,l})^2\leq \frac{C}{k^2}
\int_{C_{k,l}\cap A}|Dg|^2.
\]
In an analogous way, we define $\tilde{g}_k$ from $\tilde g$
and denote by $\tilde{g}_{k,l}$ the corresponding averages.
Note that $g_k\ra g$ and $\tilde g_k\ra\tilde g$
in $L^2(A,\Iq)$.

We next define a Lipschitz function $f_k:B\ra\Iq$.
We set $f_k (x_{k,l}, 0)=\bar{g}_{k,l}$ and $f_k (x_{k,l}, \eps) =
\tilde{g}_{k,l}$. We then use Theorem \ref{thm.ext}
to extend $f_k$ on the $1$-skeleton of the cubical
decomposition given by $D_{k,l}\times [0,\eps]$. We apply
inductively Theorem \ref{thm.ext} to extend $f_k$
to the $j$-skeletons.

If $V_{k,l}$ and $Z_{k,l}$ denote, respectively, the set of vertices of 
$D_{k,l}\times\{0\}$ and $D_{k,l}\times \{\eps\}$,
we then conclude that 
\begin{equation}\label{e:stimalip2}
\Lip (f_k|_{D_{k,l}\times \{\eps\}})\leq C\,
\Lip (f_k|_{Z_{k,l}})
\quad\text{and}\quad
\Lip (f_k|_{D_{k,l}\times\{0\}})\leq
C\,\Lip (f_k|_{V_{k,l}}).
\end{equation}
Let $(x_{k,i},0)$
and $(x_{k,j},0)$ be two adjacent vertices in $V_{k,l}$.
Then,
\begin{align}\label{e:near1}
\cG(f_k(x_{k,i},0),f_k(x_{k,j},0))^2 &=
\cG(g_k(x_{k,i}),g_k(x_{k,j}))^2 =
\mint_{C_{k,i}\cap C_{k,j}\cap A}
\cG(g_k(x_{k,i}),g_k(x_{k,j}))^2\nonumber\\
& \leq  C \mint_{C_{k,i}\cap A}\cG(\overline g_{k,i},g)^2
+C\mint_{C_{k,j}\cap A}\cG(g,\overline g_{k,j})^2\nonumber\\
& \leq \frac{C}{k^{m+1}}\int_{C_{k,i}\cup C_{k,j}}|Dg|^2.
\end{align}
In the same way, if $(x_{k,i},\eps)$
and $(x_{k,j},\eps)$ are two adjacent vertices in $Z_{k,l}$,
then
\begin{equation*}
\cG(f_k(x_{k,i},\eps),f_k(x_{k,j},\eps))^2
\leq \frac{C}{k^{m+1}}\int_{C_{k,i}\cup C_{k,j}}
|D\tilde g|^2.
\end{equation*}
Finally, for $(x_{k,i},0)$ and $(x_{k,i},\eps)$, we have
\begin{equation*}
\cG\big(f_k(x_{k,i},0),f_k(x_{k,i},\eps)\big)^2
= \eps^{-2}\,\cG(g_{k,i},\tilde g_{k,i})^2\leq
\mint_{C_{k,i}\cap A}\eps^{-2}\,\cG(g_{k},\tilde g_k)^2.
\end{equation*}
Hence, if $\{C_{k,\alpha}\}_{\alpha=1,\ldots, 2^{m-1}}$
are all the cubes intersecting $D_{k,l}$,
we conclude that the Lipschitz constant of $f_k$
in $D_{k,l}\times [0,\eps]$ is bounded in the following way:
\begin{equation*}
\Lip(f_k\vert_{D_{k,l}\times[0,\eps]})^2\leq
\frac{C}{k^{m-1}} \int_{\cup_{\alpha} C_{k,\alpha}}
\big(|Dg|^2+|D\tilde g|^2
+\eps^{-2}\cG(g_k,\tilde g_k)^2\big).
\end{equation*}
Observe that each $C_{k,\alpha}$ intersects at most $N$
cubes $D_{k,l}$, for some dimensional constant $N$. Thus,
summing over $l$, we conclude
\begin{equation}\label{e:energy approx}
\D(f_k,A\times[0,\eps])\leq C\left(
\eps \int_A |Dg|^2+\eps\int_A |D\tilde g|^2+
\eps^{-1}\int_A \cG(g_k,\tilde g_k)^2\right).
\end{equation}
Next, having fixed $D_{k,l}$, consider one of its
vertices, say $x'$. By \eqref{e:stimalip2} and \eqref{e:near1}, we conclude
\[
\max_{y\in D_{k,l}} \cG (f_k (y,0), f_k (x', 0))^2
\leq \frac{C}{k^{m+1}}\int_{\cup_\alpha C_{k,\alpha}}
|Dg|^2.
\]
For any $x\in D_{k,l}$, $g_k (x)$ is equal to
$f_k (x',0)$ for some vertex $x'\in D_{k,l}$. Thus, we can
estimate
\begin{equation}\label{e:tracciag}
\int_A \cG (f_k (x,0), g_k (x))^2\, dx
\leq \frac{C}{k^2} \int_A |Dg|^2.
\end{equation}
Recalling that $g_k\to g$ in $L^2$, we conclude, 
therefore, that $f_k (\cdot, 0)$ converges to
$g$. A similar conclusion can be inferred for $f_k (\cdot, \eps)$.

Finally, from \eqref{e:energy approx}
and \eqref{e:tracciag}, we conclude a uniform bound on
$\||f_k|\|_{L^2 (B)}$. Using the compactness of the embedding
$W^{1,2}\subset L^2$, we conclude the existence of a subsequence
converging strongly in $L^2$ to a function $h\in W^{1,2} (B)$.
Obviously, $h$ satisfies \eqref{e:estimate energy}.
We now want to show that \eqref{e:marginals} holds.

Let $\delta\in ]0, \eps[$ and assume that
$f_k (\cdot , \delta)\to f (\cdot, \delta)$ in $L^2$
(which in fact holds for a.e. $\delta$). Then,
a standard argument shows that
\[
\int_A \cG (f (x,\delta), g (x))^2\, dx
= \lim_{k\uparrow\infty}
\int_A \cG (f_k (x,\delta), g_k (x))^2 \, dx
\leq \limsup_{k\uparrow\infty} \delta \||Df_k|\|^2_{L^2 (B)}
\leq C\delta.
\]
Clearly, this implies that $f (\cdot,0) = g$.
An analogous computation shows $f(\cdot,\eps)=\tilde{g}$.

\medskip

\textit{Step $2$. Interpolation between two spherical shells.}
In what follows, we denote by $D$ the closed
$(m-1)$-dimensional ball
and assume that $\phi_+: D\to \de B_1\cap \{x_m\geq 0\}$ 
is a diffeomorphism. Define $\phi_-: D\to
\de B_1 \cap \{x_m\leq 0\}$ by simply setting 
$\phi_- (x) = - \phi_+ (x)$. 
Next, let $\phi : A\to D$ be a biLipschitz 
homeomorphism, where $A$ is the set in Step 1, and set 
\[
\varphi_{\pm}= \phi_{\pm}\circ \phi,\quad  
g_{k,\pm}=g\circ\ph_{\pm}\,\quad\mbox{and}
\quad \tilde g_{k,\pm}=\tilde g\circ\ph_{\pm}.
\] 
Consider
the Lipschitz approximating functions constructed in Step $1$,
$f_{k,+}:A\times[0,\eps]\ra\Iq$ interpolating between
$g_{k,+}$ and $\tilde{g}_{k,-}$.

Next, to construct $f_{k,-}$,
we use again the cell decomposition of Step 1. We follow
the same procedure to attribute the 
values $f_{k,-} (x_{k,l}, 0)$ and $f_{k,-} (x_{k,l},\eps)$
on the vertices $x_{k,l}\not\in \de A$. We instead set   
$f_{k,-} (x_{k,l}, 0) = f_{k,+} (x_{k,l}, 0)$ and 
$f_{k,-} (x_{k,l}, \eps) = f_{k,+} (x_{k,l}, \eps)$ when
$x_{k,l}\in \de A$. Finally, when using Theorem \ref{thm.ext}
as in Step 1, we take care to set $f_{k,+}= f_{k,-}$
on the skeleta lying in $\partial A$ and we define
\[
f_k(x)=
\begin{cases}
f_{k,+}(\varphi_{+}^{-1} (x/|x|),1-\abs{x}) & \textrm{if } x_m\geq0\\
f_{k,-}(\varphi_{-}^{-1} (x/|x|),1-\abs{x}) & \textrm{if } x_m\leq0\, .
\end{cases}
\]
Then, $f_k$ is a Lipschitz map. We want to use
the estimates of Step 1 in order to conclude the
existence of a sequence converging to a function $h$ which satisfies
the requirements of the proposition. This is straightforward on
$\{x_m\geq 0\}$. On $\{x_m\leq 0\}$ we just have to control
the estimates of Step 1 for vertices lying on $\de A$. 
Fix a vertex $x_{k,l}\in \de A$.
 
In the procedure of Step 1, $f_{k,-} (x_{k,l}, 0)$
and $f_{k,-} (x_{k,l}, \eps)$ are defined by taking the averages
$h_{k,l}$ and $\tilde{h}_{k,l}$ for
$g \circ \varphi_-$ and $\tilde{g}\circ \varphi_-$ 
on the cell $C_{k,l}\cap A$. In the procedure specified above
the values of  $f_{k,-} (x_{k,l}, 0)$
and $f_{k,-} (x_{k,l}, \eps)$ are given by the averages of
$g\circ \varphi_+$ and $\tilde{g}\circ \varphi_+$, which we denote by
$g_{k,l}$ and $\tilde{g}_{k,l}$.
However, we can estimate the difference in the following way
\[
|g_{k,l}- h_{k,l}|\leq \frac{C}{k^{m+2}} \int_{E_{k,l}} |Dg|^2,
\]
where $E_{k,l}$ is a suitable cell in $\de B_1$ containing
$\varphi_+ (C_{k,l})$ and $\varphi_- (C_{k,l})$. Since these
two cells have a face in common and $\varphi_{\pm}$
are biLipschitz homeomorphisms, we can estimate the diameter of $E_{k,l}$
with $C/k$ (see Figure \ref{f:sfera}). 
Therefore the estimates \eqref{e:energy approx} and
\eqref{e:tracciag} proved in Step 1 hold with (possibly) worse constants. 
\end{proof}

\begin{figure}[htbp]
\begin{center}
    \input{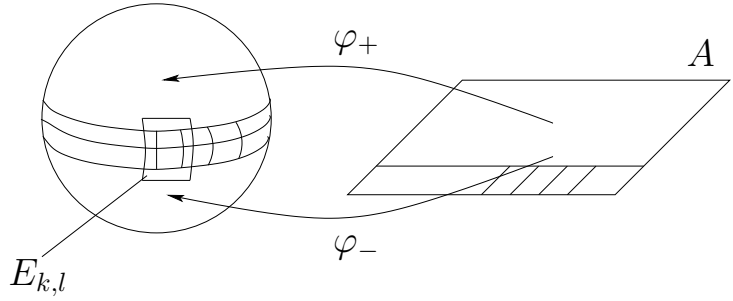}
    \caption{The maps $\varphi_\pm$ and the cells $E_{k,l}$.}
    \label{f:sfera}
\end{center}
\end{figure}

\chapter{The improved estimate of the singular set in $2$ dimensions}

In this final part of the paper we prove Theorem
\ref{t:finite}. The first section gives a more stringent
description of $2$-d tangent functions to $\D$-minimizing
functions. The second section uses a comparison
argument to show a certain rate of convergence for
the frequency function of $f$.
This rate implies the uniqueness of the tangent function.
In Section \ref{s:isolated}, we use this uniqueness to 
get a better description of a $\D$-minimizing functions
around a singular point: an induction argument on $Q$ yields
finally Theorem \ref{t:finite}.

Throughout the rest of the paper we use the notation introduced
in Remark \ref{r:complex notation} and sometimes
use $(r,\theta)$ in place of $r\, e^{i\theta}$.

\section{Characterization of $2$-d tangent $Q$-valued functions}
\label{s:2d classification}

In this section we analyze further
$\D$-minimizing functions $f: \disk\ra\Iqs$ which are homogeneous, that is
\begin{equation}\label{e:homogeneous function}
f(r, \theta)=r^\alpha\,g (\theta)
\quad \mbox{for some $\alpha> 0$.}
\end{equation}
Recall that, for $T=\sum_i \a{T_i}$ we denote by $\etaa (T)$
the center of mass $Q^{-1} \sum_i T_i$.

\begin{propos}\label{p:2d classification}
Let $f:\disk\ra\Iqs$ be a nontrivial, $\alpha$-homogeneous function which is
$\D$-minimizing. Assume in addition that $\etaa \circ f =0$. Then,
\begin{itemize}
\item[$(a)$] $\alpha=\frac{n^*}{Q^*}\in\Q$,
with ${\rm MCD}\, (n^*,Q^*)=1$;
\item[$(b)$] there exist injective ($\R{}$-)linear maps
$L_j:\C\to \R{n}$
and $k_j\in \N$ such that
\begin{equation}\label{e:hom decomposition}
f(x)=k_0\a{0}+\sum_{j=1}^J k_j \sum_{z^{Q^*}=x}
\a{L_j \cdot z^{n^*}} =:
k_0\a{0} + \sum_{j=1}^J k_j \a{f_j (x)}.
\end{equation}
Moreover, $J\geq 1$ and $k_j\geq 1$ for all $j\geq 1$.
If $Q^*=1$, either $J\geq 2$ or $k_0>0$. 
\item[$(c)$] For any $i\neq j$ and any $x\neq 0$,
the supports of $f_i(x)$ and $f_j (x)$
are disjoint. 
\end{itemize}
\end{propos}
\begin{proof}
Let $f$ be a homogeneous $\D$-minimizing $Q$-valued function.
We decompose $g=f|_{\s^1}$ into irreducible
$W^{1,2}$ pieces
as described in Proposition \ref{p:Wselection}. 
Hence, we can write $g (\theta) = k_0 \a{0} + \sum_{j=1}^J k_j \a{g_j (x)}$, where
\begin{itemize}
\item[(i)] $k_0$ might vanish, while $k_j>0$ for every $j>0$,
\item[(ii)] the $g_j$'s are all distinct, 
$Q_j$-valued irreducible $W^{1,2}$ maps such that
$g_j (x)\neq Q\a{0}$ for some $x\in \s^1$.
\end{itemize}
By the characterization of irreducible pieces,
there are $W^{1,2}$ maps $\gamma_j:\s^1\to \R{n}$ such that
\begin{equation}\label{e:gamma_j}
g_j (x)= \sum_{z^{Q_j} = x} \a{\gamma_j (z)}.
\end{equation}
Recalling \eqref{e:homogeneous function}, we extend
$\gamma_j$ to a function $\beta_j$ on the disk by setting
$\beta_j (r, \theta)=r^{\alpha\,Q_j} \gamma_j (\theta)$ 
and we conclude that
\[
f (x) = k_0 \a{0} + \sum_{j=1}^J 
\sum_{z^{Q_j}=x} \a{\beta_j (z)}
=: k_0 \a{0} + \sum_{j=1}^J k_j \a{f_j (x)}.
\]

It follows that each $f_j$ is an $\alpha$-homogeneous,
$\D$-minimizing function which assumes
values different from $Q\a{0}$ somewhere. By Lemma \ref{l:rolling},
$\beta_j$ is necessarily a $\D$-minimizing $\R{n}$-valued
function. Since $\beta_j$ is $(\alpha\, Q_j)$-homogeneous,
its coordinates must be homogeneous harmonic polynomials.
Moreover, $\beta_j$ does not vanish identically.
Therefore, we conclude that $n_j=\alpha\, Q_j$ is a positive integer.
Thus, the components of each $\beta_j$ are linear combinations
of the harmonic functions $(r,\theta)\mapsto r^{n_j} \cos (n_j\theta)$
and $(r,\theta)\mapsto r^{n_j} \sin (n_j\theta)$.
It follows that there are 
(nonzero) $\R{}$-linear map $L_j : \C\to\R{n}$ such that
$\beta_j (z) = L_j\cdot z^{n_j}$.

Next, let $n^*$ and $Q^*$ be the two positive integers determined by
$\alpha = n^*/Q^*$ and ${\rm MCD}\, (n^*, Q^*)=1$.
Since $n_j/Q_j = \alpha = n^*/Q^*$,
we necessarily have $Q_j = m_j Q^*$
for some integer $m_j=\frac{n_j}{n^*}\geq 1$.
Hence,
\[
g_j (x)=\sum_{z^{m_j Q^*}=x} \a{L_j \cdot z^{m_j n^*}}.
\]
However, if $m_j>1$, then
$\supp(g_j)\equiv Q^*\neq Q_j$, so that
$g_j$ would not be irreducible. Therefore,
$Q_j = Q^*$ for every $j$. 

Next, since $\D (f, \disk) >0$, $J\geq 1$. If $Q^*=1$, 
$J=1$ and $k_0=0$, then $f = Q\a{f_1}$ and $f_1$
is an $\R{n}$-valued function. But then $f_1=\etaa\circ f 
= 0$, contradicting $\D (f, \disk)>0$. 
Moreover, again using the irreducibility
of $g_j$, for all $x\in \s^1$, the points
\[
L_j \cdot z^{n^*}\quad\mbox{with}\quad z^{Q^*}=x
\]
are all distinct.
This implies that $L_j$ is injective.
Indeed, assume by contradiction that 
$L_j\cdot v=0$ for some $v\neq 0$. Then, necessarily
$Q^*\geq 2$ and,
without loss of generality, we can assume
that $v=e_1$.
Let $x= e^{i\theta/n^*}\in \s^1$,
with $\theta/Q^* = \pi/2 - \pi/Q^*$,
and let us consider the set 
\[
R:= \{z^{n^*}\in\s^1:
z^{Q^*}= x\} = \{e^{i (\theta + 2\pi k)/Q^*}\}.
\] 
Therefore $w_1 = e^{i\theta/Q^*}$
and $w_2 = e^{i(\theta + 2\pi)/Q^*} = e^{i\pi- i\theta/Q^*}$
are two distinct elements of $R$. However, it is easy to see
that $w_1 - w_2 = 2\cos (\theta/Q^*) e_1$.
Therefore, $L_j w_1 = L_j w_2$, which is a contradiction.
This shows that $L_j$ is injective and concludes the proof of $(b)$.

Finally, we argue by contradiction for $(c)$. 
If (c) were false, up to rotation of the plane and 
relabelling of the $g_i$'s, we assume that
$\supp (g_1(0))$ and $\supp g_2 (0)$ have a point $P$ in common.
We can, then, choose the functions
$\gamma_1$ and $\gamma_2$ of \eqref{e:gamma_j} so that 
\[
\gamma_1 (0)=\gamma_1 (2\pi) = 
\gamma_2 (0)= \gamma_2 (2\pi)=P.
\]
We then define $\xi:\disk\to\R{n}$ in the following way:
\[
\xi (r,\theta) = \left\{ \begin{array}{ll}
r^{2\,\alpha\, Q^*} \,\gamma_1 (2\theta) 
& \mbox{if $\theta\in [0, \pi]$},\\
r^{2\,\alpha\, Q^*}\, \gamma_2 (2\theta) & 
\mbox{if $\theta\in [\pi, 2\pi]$}.
\end{array}\right.
\]
Then, it is immediate to verify that
\begin{equation}\label{e:doppiogiro}
\a{f_1(x)}+\a{f_2(x)}=\sum_{z^{2Q^*}=\,x} \a{\xi (z)}.
\end{equation}
Therefore, $f$ can be decomposed as 
\[
f (x) = \sum_{z^{2\,Q^*}=x} \a{\xi (z)}
+ \left\{k_0 \a{0} + (k_1-1) \a{f_1 (x)} + (k_2-1)
\a{f_2 (x)} + \sum_{j\geq J} k_j\a{f_i (x)}\right\}.
\]
It turns out that the map in \eqref{e:doppiogiro}
is a $\D$-minimizing function, and, hence, that 
$\xi$ is a $(2\,\alpha\,Q^*)$-homogeneous
$\D$-minimizing function. Since $2\,\alpha\,Q^*= 2\,n^*$
we conclude the
existence of a linear $L:\C\to \R{n}$ such that
\[
\a{f_1 (x)}+\a{f_2 (x)} = \sum_{z^{2Q^*}=x} \a{L\cdot
z^{2n^*}}= 2 \sum_{z^{Q^*}=x} \a{L\cdot z^{n^*}}.
\]
Hence, for any $x\in \s^1$, the cardinality of the support
of $\a{g_1 (x)}+\a{g_2 (x)}$ is at most $Q^*$. 
Since each $g_i$ is irreducible,
the cardinality of the support of $\a{g_i (x)}$ 
is everywhere exactly $Q^*$.
We conclude thus that $g_1 (x)=g_2 (x)$ for every $x$,
which is a contradiction to assumption (ii) in our 
decomposition.
\end{proof}

\section{Uniqueness of $2$-d tangent functions}\label{s:unique}

The key point of this section is the rate of convergence
for the frequency function, as stated in Proposition \ref{p:rate I}.
We use here the functions
$H_{x,f}$, $D_{x,f}$ and $I_{x,f}$ introduced in Definition \ref{d:frequency}
and drop the subscripts when $f$ is clear from the context and $x=0$.

\begin{propos}\label{p:rate I}
Let $f\in W^{1,2}(\disk,\Iq)$ be $\D$-minimizing,
with $\D(f,\disk)>0$ and set $\alpha = I_{0,f} (0) = I (0)$.
Then, there exist constants $\gamma>0$, $C>0$, $H_0>0$ and $D_0>0$ such that,
for every $0<r\leq1$,
\begin{equation}\label{e:rate I}
0\leq I (r) - \alpha\leq C\,r^\gamma,
\end{equation}
\begin{equation}\label{e:rate HD}
0\leq\frac{H(r)}{r^{2\alpha+1}}-H_0\leq C\,r^\gamma
\quad\textrm{and}\quad
0\leq\frac{D(r)}{r^{2\alpha}}-D_0\leq C\,r^\gamma.
\end{equation}
\end{propos}

The proof of this result follows computations similar
to those of \cite{Ch}. A simple corollary of \eqref{e:rate I} and \eqref{e:rate HD}
is the uniqueness of tangent functions.

\begin{theorem}\label{t:unique blowup}
Let $f:\disk\ra\Iqs$ be a $\D$-minimizing $Q$-valued functions,
with $\D(f,\disk)>0$ and $f(0)=Q\a{0}$.
Then, there exists a unique tangent map $g$ to $f$ at $0$ (i.e. the maps
$f_{0, \rho}$ defined in \eqref{e:blowup} converge locally uniformly to $g$).
\end{theorem}

In the first subsection we prove Theorem \ref{t:unique blowup} assuming 
Proposition \ref{p:rate I}, which will be then proved in the second subsection.

\subsection{Proof of Theorem \ref{t:unique blowup}}
Set $\alpha=I_{0,f}(0)$ and note that, by Theorem \ref{t:blowup} and Proposition
\ref{p:rate I}, $\alpha = D_0/H_0>0$, where $D_0$ and $H_0$ are as
in \eqref{e:rate HD}. Without loss of generality, we might assume $D_0=1$.
So, by \eqref{e:rate HD}, recalling the definition of blow-up $f_\varrho$, it follows that
\begin{equation}\label{e:mod blowup}
f_\varrho (r, \theta)= 
\varrho^{-\alpha} f (r\, \varrho, \theta)\,
(1+O(\varrho^{\gamma/2})).
\end{equation}
Our goal is to show 
the existence of a limit function (in the uniform topology)
for the blow-up $f_\varrho$.
From \eqref{e:mod blowup}, it is enough to show the existence of a
uniform
limit for the functions $h_\varrho(r,\theta)=\varrho^{-\alpha}f_\varrho(r\,\varrho,\theta)$.
Since $h_{\varrho}(r,\theta)=r^\alpha h_{r\,\varrho}(1,\theta)$,
it suffices to prove the existence of a uniform limit for
$h_\varrho\vert_{\s^1}$. On the other hand, 
the family of functions $\{h_\varrho\}_{\varrho>0}$ is
equi-H\"older
(cp. with Theorem \ref{t:blowup} and \eqref{e:rate HD} 
in Proposition \ref{p:rate I}). Therefore,
the existence of an uniform limit is equivalent to the existence of an $L^2$ limit.

So, we consider $r/2\leq s\leq r$ and estimate 
\begin{align}\label{e:l2 distance1}
\int_0^{2\pi}\cG\left(h_r,h_s\right)^2 & =
\int_0^{2\pi}\cG\left(\frac{f(r,\theta)}{r^\alpha},\frac{f(s,\theta)}{s^\alpha}
\right)^2 d\theta
\leq \int_0^{2\pi}\left(\int_s^r\left|\frac{d}{dt}
\left(\frac{f(t,\theta)}{t^\alpha}
\right)\right|dt\right)^2d\theta\nonumber\\
& \leq (r-s)\int_0^{2\pi}\int_s^r\left|\frac{d}{dt}
\left(\frac{f(t,\theta)}{t^\alpha}
\right)\right|^2 dt\,d\theta.
\end{align}
This computation can be easily justified
because $r\mapsto f(r,\theta)$ is a $W^{1,2}$ function
for a.e. $\theta$. 
Using the chain rule in Proposition \ref{p:chain} and the
variation formulas \eqref{e:cono}, \eqref{e:perparti}
in Proposition \ref{p:import}, we
estimate \eqref{e:l2 distance1} in the following way:
\begin{align}\label{e:l2 distance2}
\int_0^{2\pi}\cG\left(h_r, h_s\right)^2 & \leq
(r-s)\int_0^{2\pi}\int_s^r\sum_i\left\{
\alpha^2\frac{|f_i|^2}{t^{2\alpha+2}}+\frac{|\de_\nu f_i|^2}{t^{2\alpha}}
-2\,\alpha\,\frac{\langle\de_\nu f_i,f_i\rangle}{t^{2\alpha+1}}\right\}
\nonumber\\
& \stackrel{\mathclap{\eqref{e:cono},\;\eqref{e:perparti}}}{=}\quad (r-s)\int_s^r\left\{
\alpha^2\,\frac{H(t)}{t^{2\alpha+3}}+\frac{D^\prime(t)}{2\,t^{2\alpha+1}}
-2\,\alpha\,\frac{D(t)}{t^{2\alpha+2}}\right\}dt\nonumber\\
& = (r-s)\int_s^r\left\{
\frac{1}{2t}\left(\frac{D(t)}{t^{2\alpha}}\right)^\prime
+\alpha^2\,\frac{H(t)}{2\,t^{2\alpha+3}}
-\alpha\,\frac{D(t)}{t^{2\alpha+2}}\right\}dt\nonumber\\
& =  (r-s)\int_s^r\left\{
\frac{1}{2t}\left(\frac{D(t)}{t^{2\alpha}}\right)^\prime
+\alpha\,\frac{H(t)}{2\,t^{2\alpha+3}}
\big(\alpha
-I_{0,f}(t)\big)\right\}dt\nonumber\\
&\leq (r-s)\int_s^r
\frac{1}{2t}\left(\frac{D(t)}{t^{2\alpha}}\right)^\prime dt
= (r-s)\int_s^r
\frac{1}{2t}\left(\frac{D(t)}{t^{2\alpha}}-D_0\right)^\prime dt
\end{align}
where the last inequality follows from the monotonicity of the
frequency function, which implies, in particular, that
$\alpha\leq I_{0,f}(t)$ for
every $t$.
Integrating by parts the last integral of \eqref{e:l2 distance2},
we get
\begin{multline*}
\int_0^{2\pi}\cG\left(h_r, h_s\right)^2 \leq
(r-s)\left[\frac{1}{2\,r} \left(\frac{D(r)}{r^{2\alpha}} - D_0\right)
-\frac{1}{2\,s}\left(\frac{D(s)}{s^{2\alpha}} - D_0\right)\right]+\\
+(r-s)
\int_s^r \frac{1}{2t^2} \left(
\frac{D(r)}{r^{2\alpha}}
- D_0\right).
\end{multline*}
Recalling that $0\leq D(r)/r^{2\alpha}- D_0\leq C r^\gamma$ 
and $s=r/2$ we estimate
\begin{equation}\label{e:l2 distance3}
\int_0^{2\pi}\cG\left(h_r, h_s\right)^2 \leq
\frac{r-s}{s}\, r^\gamma + (r-s) \int_s^r \frac{1}{2t^{2-\gamma}}
\leq C r^\gamma.
\end{equation}
Let now $s\leq r$ and choose $L\in\N$ such that
$r/2^{L+1}< s\leq r/2^{L}$. Iterating \eqref{e:l2 distance3},
we reach
\begin{equation*}
\norm{\cG\left(h_r,h_s\right)}{L^2}\leq
\sum_{l=0}^{L-1}\norm{\cG\left(h_{r/2^l},h_{r/2^{l+1}}\right)}{L^2}+
\norm{\cG\left(h_{r/2^{L}},h_s\right)}{L^2}
\leq 
\sum_{l=0}^{L}\frac{r^{\gamma/2}}{\left(2^{\gamma/2}\right)^l}\leq
C\,{r^{\gamma/2}}.
\end{equation*}
This shows that $h_\varrho|_{\s^1}$ is a Cauchy sequence in $L^2$ and,
hence, concludes the proof.

\subsection{Proof of Proposition \ref{p:rate I}}
The key of the proof is the following estimate:
\begin{equation}\label{e:estimate I'}
I^\prime(r)\geq \frac{2}{r}\left(\alpha +\gamma-I(r)\right)
\left(I-\alpha\right).
\end{equation}
We will prove \eqref{e:estimate I'} in a second step.
First we show how to conclude
the various statements of the proposition.

\medskip

\textit{Step 1. \eqref{e:estimate I'}$\Longrightarrow$ Proposition
\ref{p:rate I}.} 
Since $I$ is monotone nondecreasing
(as proved in Theorem \ref{t:frequency}), there exists $r_0>0$ such that
$\alpha +\gamma-I(r)\geq \gamma/2$ for every $r\leq r_0$.
Therefore,
\begin{equation}\label{e:estimate I'2}
I^\prime(r)\geq \frac{\gamma}{r}\left(I(r)-\alpha\right)
\quad \forall \;r\leq r_0.
\end{equation}
Integrating the differential inequality \eqref{e:estimate I'2}, we get
the desired conclusion:
\begin{equation*}\label{e:estimate I}
I(r)-\alpha\leq r^\gamma\left(I(r_0)-\alpha\right)=C\,r^\gamma.
\end{equation*}
From the computation of $H^\prime$ in \eqref{e:H'3}, we deduce
easily that
\begin{equation}\label{e:H'4}
\left(\frac{H(r)}{r}\right)^\prime=\frac{2\,D(r)}{r}.
\end{equation}
This implies the following identity:
\begin{equation}\label{e:H'5}
\left(\log\frac{H(r)}{r^{2\alpha+1}}\right)^\prime=
\left(\log\frac{H(r)}{r}-\log r^{2\alpha}\right)^\prime=
\left(\frac{H(r)}{r}\right)^\prime-\frac{2\alpha}{r}
\stackrel{\eqref{e:H'4}}{=}\frac{2}{r}\left(I(r)-\alpha\right)\geq0.
\end{equation}
So, in particular, we infer the monotonicity of $\log\frac{H(r)}{r^{2\alpha+1}}$ and, hence,
of $\frac{H(r)}{r^{2\alpha+1}}$.
We can, therefore, integrate \eqref{e:H'5} and use
\eqref{e:rate I} in order to achieve that, for $0<s<r\leq 1$
and for a suitable constant $C_{\gamma}$,
the function
\begin{equation*}
\log\frac{H(r)}{r^{2\alpha+1}}-C_{\gamma}\,r^\gamma=
\log\left(\frac{H(r)\,e^{-C_{\gamma}\,r^\gamma}}{r^{2\alpha+1}}\right)
\end{equation*}
is decreasing.
So, we conclude the existence of the following limits:
\begin{equation*}
\lim_{r\ra0}\frac{H(r)\,e^{-C_{\gamma}\,r^\gamma}}{r^{2\alpha+1}}=
\lim_{r\ra0}\frac{H(r)}{r^{2\alpha+1}}=H_0>0,
\end{equation*}
with the bounds, for $r$ small enough,
\begin{equation*}
\frac{H(r)}{r^{2\alpha+1}}\left(1-C\,r^\gamma\right)
\leq \frac{H(r)\,e^{-C_{\gamma}\,r^\gamma}}{r^{2\alpha+1}}
\leq
H_0\leq \frac{H(r)}{r^{2\alpha+1}}.
\end{equation*}
This easily concludes the first half of \eqref{e:rate HD}.
The rest of \eqref{e:rate HD}
follows from the following identity:
\[
\frac{D(r)}{r^{2\alpha}} -D_0= (I(r) - I_0)\, \frac{H(r)}{r^{2\alpha+1}}
+ I_0 \,\left(\frac{H(r)}{r^{2\alpha+1}} - H_0\right).
\]
Indeed, both addendum are positive and bounded by $C\,r^\gamma$.

\medskip

\textit{Step 2. Proof of \eqref{e:estimate I'}.}
Recalling the computation
in \eqref{e:I'0}, 
\eqref{e:estimate I'} is equivalent to
\[
\frac{r\,D^\prime(r)}{H(r)}-\frac{2\,I(r)^2}{r}\geq
\frac{2}{r}\big(\alpha +\gamma-I(r)\big)
\left(I(r)-\alpha\right),
\]
which, in turn, reduces to
\begin{equation}\label{e:estimate I'3}
\left(2\alpha+\gamma\right)D(r)\leq
\frac{r\,D^\prime(r)}{2}+\frac{\alpha(\alpha+\gamma)\,H(r)}{r}.
\end{equation}
To prove \eqref{e:estimate I'3}, we exploit once again
the harmonic competitor constructed in the proof of
the H\"older regularity for the planar case in Proposition \ref{p:basic}.
Let $r>0$ be a fixed radius and
$f(re^{i\theta}) =g(\theta)=\sum_{j=1}^J\a{g_j (\theta)}$ 
be an irreducible
decomposition as in Proposition \ref{p:Wselection}.
For each irreducible $g_j$, we find 
$\gamma_j\in W^{1,2} (\s^1, \R{n})$ and $Q_j$ such that
\[
g_j (\theta) = \sum_{i=1}^{Q_j} \a{\gamma_j \left(
\frac{\theta+2\pi i}{Q_j}\right)}.
\]
We write now the different quantities in \eqref{e:estimate I'3}
in terms of the Fourier coefficients of the $\gamma_j$'s.
To this aim, consider the Fourier expansions of the $\gamma_j$'s,
\begin{equation*}
\gamma_j(\theta)=\frac{a_{j,0}}{2}+\sum_{l=1}^{+\infty}r^l
\big\{a_{j,l}\cos(l\,\theta)+b_{j,l}\sin(l\,\theta)\big\},
\end{equation*}
and their harmonic extensions
\begin{equation*}
\zeta_j(\varrho,\theta)=\frac{a_{j,0}}{2}+\sum_{l=1}^{+\infty}\varrho^l
\big\{a_{j,l}\cos(l\,\theta)+b_{j,l}\sin(l\,\theta)\big\}.
\end{equation*}
Recalling Lemma \ref{l:rolling}, we infer
the following equalities:
\begin{equation}\label{e:fourier D'}
D^\prime(r) =
2\sum_j\D(g_j,r\,\s^1)=
\sum_j\frac{2\,\D(\gamma_j,r\,\s^1)}{Q_j}
=
2\,\pi\sum_j\sum_l \frac{r^{2l-1}\,l^2}{Q_j}\left(
a_{j,l}^2+b_{j,l}^2\right),
\end{equation}
\begin{equation}\label{e:fourier H}
H(r)=
\sum_j\hspace{-0.05cm}\int_{r\,\s^1}\left|g_j\right|^2
=\sum_j Q_j \int_{r\,\s^1}
\left|\gamma_j\right|^2
=\pi\sum_j Q_j\left\{\frac{r\,a_{j,0}^2}{2}+\sum_l r^{2l+1}
\left(
a_{j,l}^2+b_{j,l}^2\right)\right\}.
\end{equation}
Finally, using the minimality of $f$,
\begin{equation}\label{e:fourier D}
D(r) \leq \sum_j\D(\zeta_j,B_r)
=\pi\sum_j \sum_l r^{2l}\, l\left(
a_{j,l}^2+b_{j,l}^2\right).
\end{equation}
We deduce from \eqref{e:fourier D'}, \eqref{e:fourier H} and
\eqref{e:fourier D} that, to prove \eqref{e:estimate I'3},
it is enough to find a $\gamma$ such that
\begin{equation*}\label{e:estimate I'4}
\left(2\alpha+\gamma\right)\,l\leq
\frac{l^2}{Q_j}+\alpha\,(\alpha+\gamma)\,Q_j,
\quad
\mbox{for every $l\in\N$ and every $Q_j$,}
\end{equation*}
which, in turn, is equivalent to
\begin{equation}\label{e:estimate I'5}
\gamma\,Q_j \,(l-\alpha\,Q_j)\leq(l-\alpha\,Q_j)^2.
\end{equation}
Note that the $Q_j$'s depend on $r$,
the radius we fixed. However, they are always natural numbers less
or equal than $Q$.
It is, hence, easy to verify that the following $\gamma$ satisfies
\eqref{e:estimate I'5}:
\begin{equation}\label{e:gamma}
\gamma=\min_{1\leq k\leq Q}
\left\{\frac{\lfloor\alpha\,k\rfloor+1-\alpha\,k}{k}\right\}.
\end{equation}

\section{The singularities of $2$-d $\D$-minimizing functions
are isolated}\label{s:isolated}

We are finally ready to prove Theorem \ref{t:finite}.

\begin{proof}[Proof of Theorem \ref{t:finite}]
Our aim is to prove that, if $f:\Omega\to \Iq$
is $\D$-minimizing, then the singular points of $f$
are isolated.
The proof is by induction on the number
of values $Q$.
The basic step of the induction procedure, $Q=1$,
is clearly trivial, since $\Sigma_f=\emptyset$.
Now, we assume that the claim is true
for any $Q^\prime<Q$ and we will show that it holds for $Q$ 
as well. 

So, we fix $f: \R{2}\supset \Omega\to \Iq$ $\D$-minimizing.
Since the function $f-Q\a{\etaa\circ f}$ is still $\D$-minimizing
and has the same singular set as $f$ (notations
as in Lemma \ref{l:harm+harm}),
it is not restrictive to assume $\etaa\circ f\equiv 0$.

Next, let $\Sigma_{Q,f}=\{x: f(x) = Q\a{0}\}$
and recall that, by the proof of Theorem \ref{t:structure},
either $\Sigma_{Q,f}=\Om$ or
$\Sigma_{Q,f}$ consists of isolated points.
Assuming to be in the latter case,
on $\disk\setminus \Sigma_{Q,f}$,
we can locally decompose $f$ as the sum of a $Q_1$-valued and
a $Q_2$-valued $\D$-minimizing function with $Q_1, Q_2<Q$.
We can therefore use the inductive hypothesis to
conclude that the points of $\Sigma_f\setminus \Sigma_{Q,f}$ 
are isolated. It remains to show
that no $x\in\Sigma_{Q,f}$ is the limit 
of a sequence of points in $\Sigma_f\setminus \Sigma_{Q,f}$. 

Fix $x_0\in \Sigma_{Q,f}$. Without loss of generality, 
we may assume $x_0=0$. Note that $0\in\Sigma_{Q,f}$ implies
$D(r)>0$ for every $r$ such that $B_r\subset \Om$.
Let $g$ be the tangent
function to $f$ in $0$ .
By the characterization in Proposition \ref{p:2d classification},
we have
\[
g =k_0\a{0}+\sum_{j=1}^Jk_j\a{g_j},
\]
where the $g_j$'s are $Q^*$-valued functions
satisfying $(a)$-$(c)$ of Proposition \ref{p:2d classification}
(in particular $\alpha=n^*/Q^*$ is the frequency in $0$).
So, we are necessarily in one of the following
cases:
\begin{itemize}
\item[(i)] $\max \{k_0, J-1\} >0$;
\item[(ii)] $J=1$, $k_0=0$ and $k_1<Q$.
\end{itemize}
If case (i) holds, we define
\begin{equation}\label{e:choice eps}
d_{i,j}:=\min_{x\in \s^1}\;
\dist \big(\supp (g_i (x)),\supp( g_j (x))\big)
\quad\textrm{and}\quad\eps= \min_{i\neq j}\frac{d_{i,j}}{4}.
\end{equation}
By Proposition \ref{p:2d classification}(c), we have
$\eps>0$.
From the uniform convergence of
the blow-ups to $g$,
there exists $r_0>0$ such that
\begin{equation}\label{e:splitting}
\cG\left(f(x),g(x)\right)\leq\eps\,|x|^\alpha
\quad\textrm{for every }
\; |x|\leq r_0.
\end{equation}
The choice of $\eps$ in \eqref{e:choice eps} and
\eqref{e:splitting} easily implies the existence of
$f_j$, with $j\in \{0, \ldots, J\}$,
such that $f_0$ is a $W^{1,2}$ $k_0$-valued function, each 
$f_j$ is a $W^{1,2}$ $(k_j\,Q^*)$-valued function for $j>0$,
and
\begin{equation}\label{e:splitting2}
f\vert_{B_{r_0}}=\sum_{j=0}^J \a{f_j}.
\end{equation}
It follows that each $f_j$ is
a $\D$-minimizing function.
The sum \eqref{e:splitting2} contains at least two terms:
so each $f_j$ take less than $Q$ values and  
we can use our inductive hypothesis 
to conclude that
$\Sigma_f\cap B_{r_0}=\bigcup_j 
\Sigma_{f_j}\cap B_{r_0}$ consists of isolated points.

If case (ii) holds, then $k\, Q^*=Q$, with $k<Q$, and $g$ is of the
form 
\[
g (x) = \sum_{z^{Q^*}=x} k \a{L\cdot z^{n^*}}, 
\]
where $L$ is injective.
In this case, set 
\[
d (r):= \min_{z_1^{Q^*} = z_2^{Q^*},\, z_1\neq z_2,\, |z_i|=r^{1/Q^*}}
|L\cdot z_1^{n^*}- L\cdot z_2^{n^*}|.
\]
Note that 
\[
d (r) = c\,r^\alpha\quad\textrm{and}
\quad\max_{|x|=r}\, \dist \big(\supp (f(x)),\supp(g(x))\big) = o (r^\alpha).
\]
This implies the existence of $r>0$ and
$\zeta\in C (B_{r}, \Is{k})$ such that 
\[
f(x)=\sum_{z^{Q^*}=x} \a{\zeta (z)}
\qquad \mbox{for $|x|< r$.}
\]
Set $\rho = r^{Q^*}$.
If $x\neq B_\rho\setminus 0$ and $\sigma<\min\{|x|, \rho-|x|\}$,
then obviously $\zeta\in W^{1,2} (B_\sigma (x))$.
Thus, $\zeta\in W^{1,2} (B_\rho\setminus B_\sigma)$
for every $\sigma>0$. On the other hand,
after the same computations as in Lemma \ref{l:rolling},
it is easy to show that $\D (\zeta, B_\rho \setminus B_\sigma)$
is bounded independently of $\rho$. We conclude
that $\zeta\in W^{1,2} (B_\rho\setminus \{0\})$. This implies
that $\zeta\in W^{1,2} (B_\rho)$ (see below) and hence we can apply
the same arguments of Lemma \ref{l:rolling}
to show that $\zeta$ is $\D$-minimizing. Therefore,
by inductive hypothesis, $\Sigma_\zeta$ consists of isolated
points. So, $\zeta$ is necessarily regular in a punctured
disk $B_{\sigma}(0)\setminus \{0\}$, which implies the
regularity of $f$ in the punctured disk $B_{\sigma^{1/Q^*}}\setminus
\{0\}$.

For the reader's convenience, we give a
short proof of the claim $\zeta\in W^{1,2} (B_\rho)$.
This is in fact a consequence of the 
identity $W^{1,2} (B_\rho\setminus \{0\}) = W^{1,2} (B_\rho)$
for classical Sobolev spaces, a byproduct of the
fact that $2$-capacity of a single point in the plain is finite.

Indeed, we claim
that, for every $T\in \Is{k}$,
the function
$h_T := \cG (\zeta, T)$ belongs to $W^{1,2} (B_\rho)$.
Fix a test function
$\varphi\in C^{\infty}_c (B_\rho)$ and denote by
$\Lambda^i$ the distributional derivative $\partial_{x_i} h_T$ 
in $B_\rho\setminus \{0\}$. For
every $\sigma\in (0, \rho)$ let
$\psi_\sigma\in C^\infty_c (B_\sigma)$ be a cutoff
function with the properties:
\begin{itemize}
\item[(i)] $0\leq \psi_\sigma\leq 1$;
\item[(ii)] $\|D \psi_\sigma\|_{C^0}\leq C \sigma^{-1}$,
where $C$ is a geometric constant independent of
$\sigma$.
\end{itemize}
Then, 
\begin{align}
\int h_T\, \partial_{x_i} \varphi &=
\int h_T\, \partial_{x_i} (\varphi\, \psi_\sigma)
+ \int h_T\, \partial_{x_i} ((1-\psi_\sigma)\, \varphi)\nonumber\\
&= \underbrace{\int_{B_\sigma} h_T\, 
\partial_{x_i} (\varphi\, \psi_\sigma)}_{(I)}
- \underbrace{\int \Lambda^i ((1-\psi_\sigma)\,
\varphi)}_{(II)} .\nonumber
\end{align}
Letting $\sigma\downarrow 0$, (II) converges
to $\int \Lambda^i \varphi$.
As for (I), we estimate it as follows:
\begin{equation*}
|(I)| \leq \|\partial_{x_i} (\varphi\, \psi_\sigma)\|_{L^2 (B_\sigma)}\,
\|h_T\|_{L^2 (B_\sigma)} .
\end{equation*}
By the absolute continuity of the integral, 
$\|h_T\|_{L^2 (B_\sigma)}\to 0$ as $\sigma\downarrow 0$.
On the other hand, we have the pointwise inequality
$|\partial_{x_i} (\varphi\, \psi_\sigma)|\leq C (1 + \sigma^{-1})$.
Therefore, $\|\partial_{x_i} (\varphi\, \psi_\sigma)\|_{L^2 (B_\sigma)}$
is bounded independently of $\sigma$. This shows
that $(I)\downarrow 0$ and hence we conclude the identity
$\int h_T\, \partial_{x_i} \varphi = - \int \Lambda^i \varphi$.
Thus, $\Lambda$ is the distributional derivative of
$h_T$ in $B_\rho$. 
\end{proof}

\begin{remark}
Theorem \ref{t:finite} is optimal.
There are $\D$-minimizing functions
for which the singular set is not empty.
Any holomorphic varieties which can be written as
graph of a multi-valued function is $\D$-minimizing.
For example, the function
\[
\disk\ni z\mapsto \a{z^{\frac{1}{2}}}+\a{-z^{\frac{1}{2}}}\in \I{2}(\R{4}),
\]
whose graph is the complex variety
$\mathscr{V}=\{(z,w)\in \C^2\,:\,|z|<1,\,w^2=z\}$,
is an example of a $\D$-minimizing function with a singular point in the origin.
A proof of this result is contained in \cite{Alm}. 
The question will be addressed also 
in \cite{Sp}.
\end{remark}

\backmatter
\nocite{*}
\bibliographystyle{amsalpha}
\bibliography{QVal}

\end{document}